\documentclass[12pt]{article}

\usepackage[margin=1in]{geometry}
\usepackage{setspace}
\usepackage{amsmath,amssymb,amsthm}
\usepackage{bm}
\usepackage[mathscr]{euscript}

\usepackage{newtxtext,newtxmath}

\usepackage{subcaption,siunitx,booktabs}
\usepackage{multirow}
\usepackage{array}
\usepackage{comment}
\usepackage{thmtools}
\usepackage{thm-restate}

\usepackage{authblk}

\usepackage{tikz}
\usetikzlibrary{
    shapes.geometric, 
    arrows.meta,     
    positioning,
    decorations.markings,
    calc,
    fit
}

\usepackage[round]{natbib}

\usepackage{hyperref}
\hypersetup{
    colorlinks=true,
    linkcolor=blue,
    citecolor=blue,
    urlcolor=blue
}

\newtheorem{definition}{Definition}

\newcolumntype{L}[1]{>{\raggedright\arraybackslash}p{#1}}
\newcolumntype{C}[1]{>{\centering\arraybackslash}p{#1}}
\newcolumntype{R}[1]{>{\raggedleft\arraybackslash}p{#1}}

\newcommand{\mb}[1]{{\color{black} #1}}
\newcommand{\rmb}[1]{{\color{black} #1}}

\newcommand{\CSO}{\mb{\texttt{CSO}}}
\newcommand{\CSOFix}{\mb{\texttt{CF}}}
\newcommand{\Pref}{\mb{\texttt{URC}}}
\newcommand{\PrefFix}{\mb{\texttt{UF}}}
\newcommand{\PrefUB}{\mb{\texttt{URC-UB}}}
\newcommand{\PrefLB}{\mb{\texttt{URC-LB}}}

\newcommand{\PWLSet}{\mathcal{I}_{\mathrm{pwl}}}
\newcommand{\PWLIdx}{\iota}
\newcommand{\PWLCount}{I_{\mathrm{pwl}}}
\newcommand{\PWLDelta}{\Delta}
\newcommand{\PWLSlope}{\mu}

\newcommand{\PWLCSO}{\texttt{PWL-CSO}}
\newcommand{\PWLURC}{\texttt{PWL-URC}}

\newcommand{\coverage}{\mb{\delta^\texttt{cov}}}

\usepackage{amsthm}

\theoremstyle{definition}
\newtheorem{example}{Example}

\DeclareMathOperator*{\argmin}{argmin}

\newenvironment{APPENDICES}{\appendix}{}

\title{\textbf{Evacuation Planning for Disaster Preparedness:\\An Adaptive Robust Optimization Approach}}

\author[1]{Jaehyuk Kim}
\author[2]{Merve Bodur}
\author[1]{Maria E. Mayorga}
\author[1]{Osman Y. {\"O}zalt{\i}n}

\affil[1]{Edward P. Fitts Department of Industrial and Systems Engineering, North Carolina State University, Raleigh, NC 27695, USA (\texttt{\{jkim226, oyozalti, memayorg\}@ncsu.edu})}
\affil[2]{School of Mathematics and Maxwell Institute for Mathematical Sciences, University of Edinburgh, UK (\texttt{merve.bodur@ed.ac.uk})}

\date{}

\begin{document}

\maketitle

\begin{abstract}
\noindent \textbf{Problem definition:} Evacuation planning for disaster preparedness requires making critical decisions under uncertainty before the number and spatial distribution of evacuees are known, including shelter location, evacuation route assignment, and relief supply prepositioning. Because these decisions are highly interdependent, planners must balance the competing objectives of maximizing relief demand coverage and minimizing evacuation time. 
\noindent \textbf{Methodology/results:} We propose, to our knowledge, the first adaptive robust evacuation planning model to jointly optimize shelter locations, evacuation route assignments, relief supply prepositioning, and post-disaster relief item distribution. The model minimizes the worst-case weighted sum of unmet demand for relief items across shelters and the congestion-dependent evacuation time. We characterize key theoretical complexity drivers of the resulting problem with mixed-integer recourse and develop a partition-and-bound algorithm that maintains tractability by selectively partitioning only the most critical subpartition of the uncertainty set while producing strong upper and lower bounds. To quantify the value of centralized route planning, we also formulate a user route choice alternative in which evacuees choose among acceptable routes. Computational experiments on the Sioux Falls network quantify the operational value of centralized route planning, which reduces worst-case unmet demand and evacuation time by up to 90.6\% and 79.3\%, respectively, relative to decentralized user route choice. Adaptive post-disaster supply redistribution further improves relief demand coverage.
\noindent \textbf{Managerial implications:} Coordination between evacuation routing and relief distribution creates substantial operational value under uncertainty. Centralized route planning primarily mitigates congestion by coordinating evacuee flows across shelters, whereas adaptive redistribution primarily improves relief demand coverage when relief supplies are scarce or inflexibly prepositioned. These findings help planners determine when centralized routing and adaptive redistribution provide the greatest operational value under different resource and service-priority settings.
\end{abstract}

\vspace{1em}

\bigskip

\section{Introduction}
\label{S4:introduction}

Natural disasters continue to impose substantial human and economic costs worldwide, despite advances in forecasting and emergency management. In 2024 alone, nearly 400 disasters were recorded globally, affecting more than 167 million people and causing approximately US \$242 billion in damages~\citep{CRED2024}. In the United States, the annual number of billion-dollar disaster events has increased from fewer than ten on average during the 1980s to more than twenty-five in recent years~\citep{Climate2025}. As disasters become more frequent and costly, emergency management agencies face growing pressure to develop preparedness plans
that ensure the efficient use of limited resources and improve service outcomes for affected populations.

Mass evacuation is among the most complex logistical operations in disaster response. To support a large-scale evacuation, planners must decide which shelters to open, how to route evacuees through the transportation network, and where to preposition relief supplies, such as food, blankets, and hygiene kits, for displaced populations. These decisions are tightly interconnected. The set of selected shelters determines feasible evacuation destinations; shelter and route decisions together shape the distribution of evacuees across the network; and the resulting shelter utilization determines local demand for relief supplies. Poor coordination of preparedness decisions can therefore intensify traffic congestion, overwhelm some shelters while leaving others underutilized, and leave evacuees without essential relief supplies during a critical stage of the response.

This coordination problem is further complicated by uncertainty. It is hard to estimate the size of the evacuee population in a region, and even harder to estimate the proportion of evacuees with limited accommodation options due to economic barriers, lack of a social support network, or constraints in hotel and short-term rental availability. Therefore, before a disaster occurs, emergency planners typically do not know how many evacuees will rely on public shelters or how demand will be distributed across open sites. 
Yet shelter locations, evacuation routes, and relief supply prepositioning must all be determined before this uncertainty is resolved. Once shelter demand is realized, however, prepositioned supplies can be redistributed to better match actual needs. Thus, although strategic preparedness decisions must be committed ex ante, operational relief distribution decisions can adapt to realized conditions. This naturally gives rise to a two-stage decision framework in which preparedness decisions are made before the disaster, followed by adaptive post-disaster relief distribution based on realized shelter demand, as illustrated in Figure~\ref{F:problem-description}.

\begin{figure}[t!]
\centering
\includegraphics[scale=0.7]{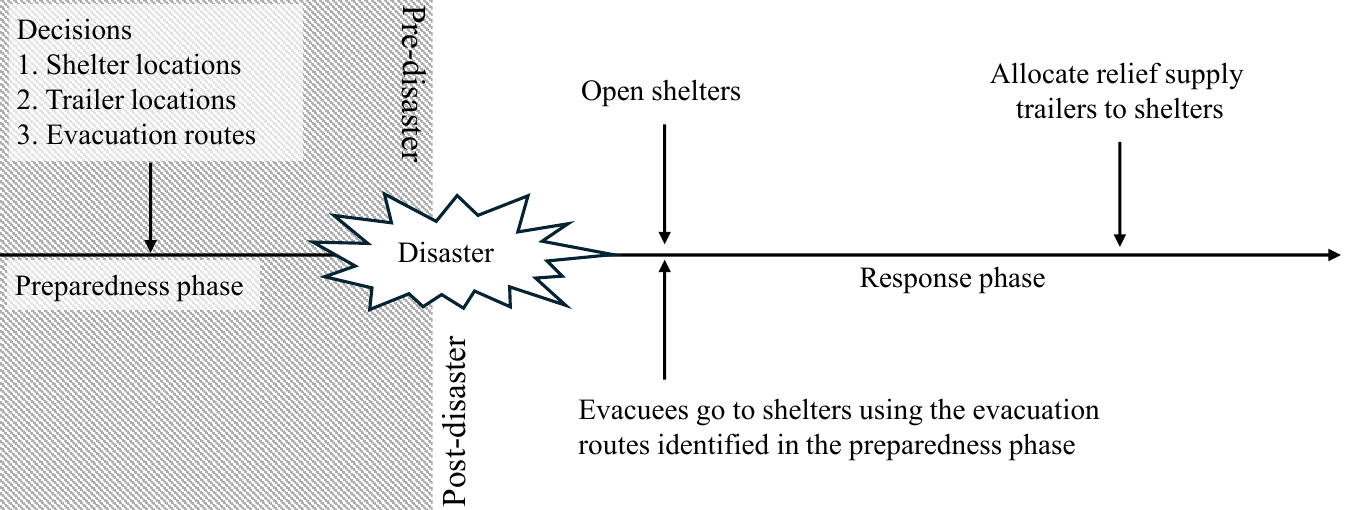}
\caption{Timeline of events. Shelter locations are selected, relief supplies are prepositioned, and evacuation routes are assigned pre-disaster. Shelters are opened, and the response phase begins when a disaster occurs. Relief supplies are distributed based on the shelter demand realization in the second stage.}
\label{F:problem-description}
\end{figure}

This decision structure is not yet fully reflected in the evacuation planning literature, as detailed in Section \ref{S4:literature}. While evacuation planning models have considered shelter location, route assignment, and relief supply planning in various combinations, their integrated treatment under uncertainty remains limited. In particular, closely related models typically do not combine implementable pre-disaster planning with adaptive post-disaster relief distribution. Moreover, many existing approaches rely on stochastic programming, which requires specifying demand scenarios and their associated probabilities. Such information is often difficult to estimate reliably because major disasters are infrequent, high-impact events with limited historical data. These challenges motivate the use of adaptive robust optimization, which supports implementable preparedness decisions while explicitly accounting for uncertainty without requiring probability distributions.

To support coordinated evacuation preparedness under uncertainty, we develop a two-stage adaptive robust mixed-integer optimization model that integrates shelter location, evacuation routing, relief supply prepositioning, and adaptive post-disaster relief distribution. The model captures the two-stage decision structure of the problem by making preparedness decisions in the pre-disaster phase while allowing adaptive post-disaster redistribution of relief supplies in response to realized shelter needs. We seek preparedness plans that perform well across a range of possible demand realizations by minimizing the worst-case weighted sum of total unmet relief item demand and total evacuation time. To solve the resulting problem with mixed-integer recourse, we develop a tree-based nested uncertainty set partitioning approach that yields tight upper and lower bounds while maintaining computational tractability. We also formulate a user route choice benchmark to quantify the operational value of centralized route planning. Computational experiments demonstrate that centralized route planning can substantially improve the trade-off between unmet demand and evacuation time, particularly when multiple shelters are opened, and relief resources are limited. The results further show that adaptive post-disaster redistribution provides important flexibility and that system performance is highly sensitive to the relative emphasis placed on service coverage versus evacuation efficiency.

The contributions of this paper can be summarized as follows:
\begin{itemize}
\item We develop, to our knowledge, the first two-stage adaptive robust evacuation planning framework that jointly optimizes strategic relief supply prepositioning and adaptive post-disaster relief distribution under demand uncertainty while integrating shelter location and evacuation route planning.

\item Within the proposed framework, we formulate two traffic assignment variants: a centralized constrained system-optimal model and a decentralized user route choice model. The latter introduces a novel decision-dependent structure in which evacuees choose among acceptable routes, with the choice set determined by shelter opening decisions. Together, these variants enable an assessment of the operational value of centralized route planning under uncertainty.

\item We provide theoretical complexity insights for adaptive robust evacuation planning\rmb{, showing that the computational difficulty arises from multiple layers of the problem: the deterministic version is NP-hard and the full adaptive robust decision problem with a bounded polyhedral uncertainty set is \(\Sigma^p_2\)-complete. We also identify a polynomially solvable recourse subproblem, clarifying which part remains tractable.}

\item Building on the \rmb{complexity} insights, we develop a partition-and-bound algorithm 
\rmb{with a problem-tailored active-partition refinement rule}
that produces strong upper and lower bounds while controlling the growth of the search tree. Computational experiments show that the method achieves better solution quality within comparable solution time to existing partition-based benchmarks. Although designed to efficiently solve the proposed evacuation planning model, the method is generic and can be applied to other two-stage adaptive robust optimization problems with mixed-integer recourse.

\item We provide computational and managerial insights through a case study on the Sioux Falls network, demonstrating the practical tractability of the proposed approach and characterizing the trade-off between unmet relief item demand and evacuation time, the value of centralized route planning, and the benefits of adaptive post-disaster relief distribution under varying resource availability and operational priorities.
\end{itemize}

The remainder of the paper is organized as follows. Section~\ref{S4:literature} reviews the relevant literature and positions our work within the fields of evacuation planning, humanitarian logistics, and robust optimization. Section~\ref{S4:problem_description} presents the proposed adaptive robust evacuation planning models. Section~\ref{sec:complexity} provides complexity insights for the problem. Section~\ref{S4:solution_approach} develops the partition-and-bound solution methodology. Section~\ref{S4:case_study} presents computational experiments and managerial insights based on the Sioux Falls network. Finally, Section~\ref{S4:conclusions} concludes the paper and discusses directions for future research.

\section{Literature Review}
\label{S4:literature}

Disaster preparedness and response have been extensively studied, addressing a wide range of decisions such as the pre-positioning of relief supplies~\citep{velasquez2020prepositioning, li2024distributional} and the location of dispensing points for mass prophylaxis~\citep{ramirez2015point}.  \citet{bayram2016} presents a comprehensive review of 
evacuation planning models, and  \cite{amideo2019optimising} discusses implementation challenges in optimizing shelter location and evacuation routing. 
We review 
evacuation planning models that consider shelter-location, shelter-assignment, route-assignment, and traffic-assignment structures in different combinations.
Table~\ref{T:lit_review_table} 
summarizes the reviewed papers along the decision, uncertainty, and traffic-assignment dimensions most relevant to our work. We use this comparison to position the gap addressed in this paper: disaster preparedness decisions under uncertain evacuee demand integrated with adaptive post-disaster relief distribution.

\noindent \textbf{Deterministic Models.}  
A first stream of studies coordinates shelter and routing decisions when demand and network conditions are treated as known. \cite{sherali1991} proposed a centralized, system-optimal model for selecting shelter locations and evacuation routes to minimize total travel time, while   
\cite{kongsomsaksakul2005} developed a deterministic bilevel model in which a government authority selects shelter locations and evacuees choose shelters and routes.
Subsequent models broadened this shelter-routing problem by considering primary and backup routes~\citep{coutinho2012}, relief item distribution and emergency medical service deployment~\citep{sheu2014}, acceptable evacuation paths~\citep{bayram2015}, or the timing of shelter openings and evacuation orders~\citep{gama2016}. These studies underscore the importance of coordinating shelter and routing decisions, but they do not address how preparedness decisions should account for uncertainty.

\begin{table}[b!]
{\small
\centering
\begin{tabular}{l ccccccc}
\toprule
\multirow{2}{*}{} 
  & \multicolumn{5}{c}{Decisions} 
  &  Uncertainty
  & Traffic \\ 
\cmidrule(lr){2-6} 
  &  SL &  SA &  RA & PS & Other & 
\rmb{Approach}  & Assignment \\
\midrule
\cite{sherali1991}         & \checkmark  & \checkmark &  \checkmark &  &   &  Det & SO  \\
\cite{kongsomsaksakul2005} & \checkmark  &  &   &  &   &  Det & UE  \\
\cite{coutinho2012}        & \checkmark  & \checkmark & \checkmark &  &   & Det  & SO  \\
\cite{sheu2014}            & \checkmark  & \checkmark &   &  &   &  Det & SO  \\
\cite{bayram2015}          & \checkmark  & \checkmark & \checkmark  &  &   &  Det & CSO \\
\cite{gama2016}            & \checkmark  & \checkmark &   &  & EO  &  Det & SO  \\
\midrule
\cite{shen2008}            & \checkmark  &   &  &   &  &  Scen & UE  \\
\cite{kulshrestha2011}     & \checkmark  &  &   &   & SC  & Scen & UE  \\
\cite{li2012}              & \checkmark  &  &  &   &  &  Scen & UE  \\
\cite{bayram2018}          & \checkmark  & \checkmark & \checkmark  &   &  &  Scen & CSO \\
\cite{liang2019}           & \checkmark  & \checkmark & \checkmark &   &  & Scen  & SO  \\
\cite{esposito2021}        & \checkmark  & \checkmark & \checkmark &   &  & Scen  & SO  \\
\midrule
\textbf{This paper}        & \checkmark  & \checkmark  & \checkmark & \checkmark  & DS  & Rob & CSO, URC \\
\bottomrule
\end{tabular}
\caption[]{Summary of the reviewed papers. Decisions: Shelter Location (SL), Shelter Assignment (SA), Shelter Capacity (SC), Route Assignment (RA), Prepositioning Relief Supplies (PS), Distributing Relief Supplies (DS), Evacuation Order (EO).  Uncertainty modeling: deterministic (Det), finite scenario/discrete uncertainty formulation
(Scen), and robust optimization 
(Rob). Traffic assignment approaches: user equilibrium (UE), user route choice (URC), system optimal (SO), and constrained system optimal (CSO).}\label{T:lit_review_table}
}
\end{table}

\noindent \textbf{Stochastic Models.} A second stream incorporates uncertain demand, travel times, road capacities, or shelter availability. \cite{shen2008} proposed a bilevel model that minimizes the conditional value-at-risk of the maximum travel-time regret,  
with route travel times determined via Beckmann's transformation of the user equilibrium model~\citep{sheffi1985urban}. 
\cite{kulshrestha2011} proposed a {\it static} robust optimization model based on a discrete uncertainty set with three levels: minimum, nominal, and maximum, for shelter demand. Their model determines the number of shelters, their locations, and capacities, while evacuees choose shelters and routes. 
Other models use scenario-based bilevel \citep{li2012}, two-stage stochastic \citep{bayram2018, esposito2021} or chance-constrained \citep{liang2019} formulations to address uncertainty in evacuee counts, road capacity, shelter availability, travel conditions, evacuation modes, or network disruption. These models show how uncertainty can be incorporated into evacuation planning, but they do not fully match the integrated preparedness structure considered here. In \cite{bayram2018}, for example, shelter locations are selected in the first stage, while evacuees are assigned to shelters and routes after uncertainty is realized; thus, shelter locations are the only implementable pre-disaster decisions. By contrast, our setting requires pre-disaster shelter-location, route-assignment, and supply-prepositioning decisions, while allowing relief distribution to adapt after demand is realized.

\noindent \textbf{\rmb{Positioning of this paper.}} The reviewed studies in Table~\ref{T:lit_review_table} do not integrate shelter location, route assignment to reach shelters, relief supply prepositioning, and adaptive post-disaster relief distribution under uncertain evacuee counts. To our knowledge, we present the first model to jointly optimize these decisions. We address this gap through a risk-averse, two-stage adaptive robust optimization framework, which is appropriate for rare, high-impact disasters where reliable scenario probabilities may be unavailable. Our model balances evacuation travel time, a central objective in evacuation planning, with unmet relief demand at shelters. Table~\ref{T:lit_review_table} also highlights differences in traffic-assignment approaches, which motivate our proposed user route choice model as a benchmark for assessing the value of centralized route planning.

\section{Robust Evacuation Planning Models} 
\label{S4:problem_description}
We propose two-stage robust mixed-integer optimization models for evacuation planning to ensure the availability of relief supplies at shelters during the early stages of disaster response. We consider a road network consisting of nodes that represent potentially impacted evacuee locations $\mathcal{O}$ (origins), candidate shelter locations $\mathcal{J}$, and candidate trailer locations $\mathcal{L}$ for prepositioning relief supplies. 
The number of shelter evacuees at each origin is uncertain. We represent this uncertainty through a generic uncertainty set $\Xi$. Each element $\xi_i$ of the realization vector $\xi\in\Xi$ specifies the corresponding evacuee count at origin $i\in\mathcal{O}$. The models follow a two-stage decision structure. In the first stage, before the evacuee counts are known, evacuation planners choose which shelters to open, where to preposition relief supply trailers, and how to make traffic assignment decisions.  We assume that each shelter evacuee requires one set of relief items; thus, shelter populations determine relief demand. Because shelter populations depend on both realized evacuee counts and traffic assignment decisions that direct evacuees to shelters, shelter-level relief demand is endogenous.
In the second stage, consistent with current practice, pallets of prepositioned relief items are distributed from nearby trailers to shelters to ensure timely availability~\citep{maghfiroh2018dynamic}.
 Finally, unmet relief demand is evaluated, and traffic conditions induced by the evacuation plan determine the evacuation time.

\begin{table}[ht!]
\caption{Summary of the notation used in the evacuation planning models.}
\label{T:ch4_notation}
\centering
\renewcommand{\arraystretch}{0.75}
\scalebox{0.88}{
\begin{tabular}{p{2cm} p{15.5cm}}
\toprule
\multicolumn{2}{l}{\textbf{Sets and indices}}\\
\cmidrule {1-2}
$\mathcal{O}$ & evacuee locations, i.e., areas potentially impacted by a disaster\\
$\mathcal{J}$ & candidate shelter locations \\
$\mathcal{L}$ & candidate trailer locations \\
$\mathcal{L}_j \subseteq \mathcal{L}$ & locations from which prepositioned relief supplies can be sent to shelter $j \in \mathcal{J}$\\
$\mathcal{R}_{ij}$ &  evacuation routes from origin $i \in \mathcal{O}$ to shelter $j \in \mathcal{J}$\\
$\mathcal{A}$ &arcs in the network (road segments)\\
$\PWLSet$ & index set of piecewise-linear volume delay function segments\\
$\Xi$ & uncertainty set for evacuee counts $\xi$ \\
\cmidrule {1-2}
\multicolumn{2}{l}{\textbf{Parameters}}\\
\cmidrule {1-2}
$\xi_{i}$ &  realization of the shelter evacuee count at origin  $i \in \mathcal{O}$ \\
$d_{r}$ & length of route $r$\\ 
$\hat{d}_{i \rmb{j}}$ & length of the longest route between origin $i \in \mathcal{O}$ and shelter $j \in \mathcal{J}$\\ 
$c_{a}$ & capacity of road segment $a \in \mathcal{A}$\\
$t^0_a$ & free-flow travel time on road segment $a \in \mathcal{A}$ \\
$\alpha,\beta$ & parameters of the volume-delay function \\
$N$ & number of shelters to open\\
$S$ & number of relief supply trailers\\
$K$ & number of pallets on a trailer\\
$m$ & number of relief item sets on a pallet (each evacuee needs one set of relief items)\\
$\lambda$ & objective function weight on evacuation time \\
$\vartheta_r$ & route attractiveness measure used in the multinomial logit model \\
$\varepsilon_r$ & error in the utility estimate of route $r$ \\
\cmidrule {1-2}
\multicolumn{2}{l}{\textbf{Decision variables}}\\
\cmidrule {1-2}
$b_{j} \in \{0,1\}$ &1 if shelter $j \in \mathcal{J}$ is open, 0 otherwise\\
$s^1_\ell \in \mathbb{Z}_{+}$ &number of trailers prepositioned at location $\ell \in \mathcal{L}$\\
$y_{r} \in [0,1]$ & proportion of shelter evacuees at origin $i \in \mathcal{O}$ assigned to route $r \in \cup_{j \in \mathcal{J}}\mathcal{R}_{ij}$ \\
$s_{\ell j}^{2}(\xi) \in \mathbb{Z}_{+}$ & number of pallets distributed from location $\ell \in \rmb{\mathcal{L}_j}$ to shelter $j \in \mathcal{J}$\\
$u_{j}(\xi) \in \mathbb{R}_{+}$ & unmet relief demand at shelter $j \in \mathcal{J}$\\
$x_a(\xi) \in \mathbb{R}_{+}$ & traffic volume on road segment $a \in \mathcal{A}$\\
$\tau_r \in \{0,1\}$ & indicator variable denoting the acceptability of route $r$ in the \Pref\ model \\
$h_r  \in [0,1]$ & auxiliary variable used to enforce route acceptability in the \Pref\ model \\
$\pi_r  \in [0,1]$ & proportion of evacuees selecting route $r$ in the \Pref\ model\\
\bottomrule
\end{tabular}
}
\end{table}
Within this two-stage structure, we develop two models that differ in traffic assignments. The first is the constrained system optimal (\CSO) model, which serves as the main planning model and allocates evacuees across acceptable routes (Section~\ref{with_route_planning}). The second is the user route choice (\Pref) benchmark, which allows evacuees to choose among acceptable routes (Section~\ref{with_route_choice}). 
Table~\ref{T:ch4_notation} summarizes the notation used in our models.

\subsection{Constrained System-Optimal Model}
\label{with_route_planning}
\rmb{The constrained system-optimal model, denoted by \CSO, is the centralized planning formulation. It chooses the first-stage decision vector $(b,s^1,y)$ before the shelter evacuee realization is known: $b$ selects the open shelters, $s^1$ positions relief-supply trailers, and $y$ assigns shelter evacuees across acceptable routes. For each realization $\xi\in\Xi$, the adaptive recourse variables $s^2_{\ell j}(\xi)$ and $u_j(\xi)$ determine pallet shipments and unmet relief demand, respectively, while $x_a(\xi)$ records the traffic volume. The formulation is given by:}
\begin{subequations}\label{originalmodel}
\begin{align}
z_{\CSO} := \min \ & \max_{\xi \in \Xi} \  \sum_{j \in \rmb{\mathcal{J}}} u_j(\xi) + \lambda \sum_{a \in \mathcal{A}} T_a(x_a(\xi))x_a(\xi)\label{obj-function_model1} \\
\text{s.t.} \ &\label{eq:tot_open_shelter_model1} \sum_{j \in \rmb{\mathcal{J}}} b_{j} = N,\quad \sum_{\rmb{\ell \in \mathcal{L}}}s^{1}_{\rmb{\ell}} = S,\\
\label{eq:route_assign_model1} & \sum_{j \in \rmb{\mathcal{J}}}  \sum_{r \in \mathcal{R}_{ij}} y_{r} = 1 & i \in \mathcal{O}, \\
&\label{eq:route_open_shelter_model1} \sum_{r \in \mathcal{R}_{ij}} y_{r} \leq b_{j}  &  i \in \mathcal{O}, j \in \rmb{\mathcal{J}},\\
&\label{eq:route_restict_model1} \sum_{j \in \rmb{\mathcal{J}}}  \sum_{r \in \mathcal{R}_{ij}:d_{r} > \hat{d}_{i \rmb{k}}} y_{r} + b_{\rmb{k}} \leq 1 &  i \in \mathcal{O},  \rmb{k \in \mathcal{J}},\\
\label{eq:unmet_model1} &   \sum_{i \in \mathcal{O}} \sum_{r \in \mathcal{R}_{ij}} \xi_{i}y_{r} \leq  u_j(\xi)+ m\sum_{\ell \in \rmb{\mathcal{L}_j}} s^{2}_{\ell j}(\xi)  & j \in \rmb{\mathcal{J}}, \xi \in \Xi,\\
\label{eq:arc_capa_model1} &   \sum_{i \in \mathcal{O}} \sum_{j \in \rmb{\mathcal{J}}} \sum_{r \in \mathcal{R}_{ij}:a \in r} \xi_{i}y_{r} ~\rmb{=}~ x_a(\xi) & a \in \mathcal{A}, \xi \in \Xi,\\
&\label{eq:flexible_trailer_capa_model1}\sum_{j \in \rmb{\mathcal{J}} \, : \, \ell \in \rmb{\mathcal{L}_j}} s^{2}_{\ell j}(\xi) \leq K s^{1}_{\ell} &  \ell \in \rmb{\mathcal{L}}, \xi \in \Xi,\\
\label{eq:non_neg_model11} & b_j \in \{0,1\}, s^1_{\rmb{\ell}} \in \mathbb{Z}_{+}, y_{r} \geq 0 & \hspace{-1cm} i \in \mathcal{O}, j \in \rmb{\mathcal{J}}, \rmb{\ell \in \mathcal{L}}, r \in \mathcal{R}_{ij},\\
\label{eq:non_neg_model1} & s^{2}_{\ell j}(\xi) \in \mathbb{Z}_{+}, u_{j}(\xi), x_a(\xi) \geq 0 & \hspace{-1cm} j \in \rmb{\mathcal{J}}, \ell \in \rmb{\mathcal{L}}, a \in \mathcal{A}, \xi \in \Xi.
\end{align}
\end{subequations} 
\rmb{The objective~\eqref{obj-function_model1} minimizes the worst-case weighted sum of unmet relief demand and evacuation time. The trade-off between these terms is operationally important because shelter and trailer locations that improve relief coverage may also concentrate traffic on a part of the road network. The volume-delay function $T_a(\cdot)$ is specified in Section~\ref{ttf}.}
Constraints~\eqref{eq:tot_open_shelter_model1} limit the number of open shelters and relief supply trailers.
\rmb{Constraints \eqref{eq:route_assign_model1}-\eqref{eq:route_open_shelter_model1} assign evacuees  to  routes leading to open shelters.} 
\rmb{Constraints~\eqref{eq:route_restict_model1} model route acceptability. If shelter $k$ is open, then  $y_r=0$ for any route $r$ from origin $i$ with $d_r>\hat{d}_{ik}$; equivalently, evacuees can be assigned to route $r$ only if $d_r\le \min_{k\in\mathcal{J}:b_k=1}\hat{d}_{ik}$. Thus evacuees need not be assigned to the closest open shelter, but they cannot be assigned to a shelter through a route that is longer than all routes to another open shelter.}
\rmb{For each realization $\xi$, constraints~\eqref{eq:unmet_model1} require the relief demand at each shelter to be satisfied by adaptive relief distribution from prepositioned trailers, with any shortfall recorded as unmet demand.}
\rmb{Constraints~\eqref{eq:arc_capa_model1} calculate the traffic volume $x_a(\xi)$ on road segment $a$.}
Constraints \eqref{eq:flexible_trailer_capa_model1} limit \rmb{the relief supplies shipped from each trailer location $\ell$ to $K s^1_\ell$ prepositioned pallets.} 
Constraints~\eqref{eq:non_neg_model11}-\eqref{eq:non_neg_model1} 
\rmb{impose the variable domains.} 

Let $z$ denote the worst-case value of the objective function. The epigraph reformulation of \CSO\ expresses the inner maximization in~\eqref{obj-function_model1} as a constraint over all realizations $\xi \in \Xi$. 
\begin{subequations}\label{armio}
\begin{align} 
(\text{ARMIO}) \ z_{\CSO} =  \min \ & z \\
\text{s.t.} \ & \sum_{j \in \rmb{\mathcal{J}}} u_j(\xi) + \lambda \sum_{a \in \mathcal{A}} T_a(x_a(\xi))x_a(\xi) \leq z &  \xi \in \Xi, \label{eq:ch3_rob_cost}\\
& \eqref{eq:tot_open_shelter_model1}-\eqref{eq:non_neg_model1}. \notag
\end{align}
\end{subequations}
When $\Xi$ is continuous, semi-infinite constraints~\eqref{eq:ch3_rob_cost} ensure that $z$ is equal to the maximum weighted sum of unmet relief demand and evacuation time over all realizations $\xi \in \Xi$. Section~\ref{S4:solution_approach} develops a partition-and-bound approach to solve this formulation.

\subsection{\rmb{User Route Choice Model}}
\label{with_route_choice}

\rmb{The \CSO\ model centrally prescribes traffic assignments $y_r$. Although the route acceptability criterion excludes routes with excessive detours, some evacuees may still be assigned to routes that are less desirable than those they would select on their own. To capture this behavioral consideration, we introduce a user route choice model, denoted by \Pref. This model retains the same shelter-opening, trailer-prepositioning, and adaptive distribution framework as \CSO, but replaces centrally assigned routes with evacuee-driven route choices among acceptable alternatives. In the computational experiments, we compare the solutions obtained under \Pref\ and \CSO\ to quantify the value of centralized route assignment.}

We employ a multinomial logit model to represent evacuees' route preferences. The utility of route $r$ is given by $f_r=\vartheta_r+\varepsilon_r$, where $\vartheta_r$ is a measurable route attractiveness parameter and $\varepsilon_r$ is the unobservable, random component~\citep{daganzo1977stochastic}.   In the case study, $\vartheta_r$ is quantified by the inverse of the route length, and $\varepsilon_r$ is assumed to be independently and identically distributed with log-Weibull distribution~\citep{ljubic2018outer}. 
The distinctive feature of \Pref \ is that the choice set is {\it decision dependent}: a route can be chosen by evacuees only if its destination shelter is open and the route satisfies the acceptability condition used in \CSO.
Let $\tau_r \in \{0,1\}$  equal to $1$ if route $r \in \mathcal{R}_{ij}$ from origin $i$ to shelter $j$ is acceptable.
The proportion of evacuees at origin $i$  selecting route $r$ is then given by
\begin{align}\label{ratiopref}
\pi_{r} = \frac{v_r\tau_r}{\sum_{\rmb{k \in \mathcal{J}}}\sum_{r' \in \mathcal{R}_{ik}}v_{r'}\tau_{r'}}, \ i \in \mathcal{O}, j \in \rmb{\mathcal{J}}, r \in \mathcal{R}_{ij},
\end{align}
where $v_{r} = e^{\vartheta_{r}}>0$ and $\pi_r \in [0,1]$ is the proportion of evacuees selecting route $r$. 
\rmb{For~\eqref{ratiopref} to be well defined, each origin $i \in \mathcal{O}$ must have at least one acceptable route. We enforce this by} 
$ \sum_{j \in \rmb{\mathcal{J}}}  \sum_{r \in \mathcal{R}_{ij}} \tau_{r} \geq 1, \ i \in \mathcal{O}$. 
At feasible points satisfying this requirement, the denominator in~\eqref{ratiopref} is positive, and the ratio can be represented equivalently by multiplying both sides by the denominator.
\begin{subequations}\label{prefmodel}
\begin{align}
z_{\Pref} := \min \ & \max_{\xi \in \Xi} \  \sum_{j \in \rmb{\mathcal{J}}} u_j(\xi) + \lambda \sum_{a \in \mathcal{A}} T_a(x_a(\xi))x_a(\xi)\label{obj-function_model2} \hspace{-1.3cm} &  \\
\text{s.t.} \ &\eqref{eq:tot_open_shelter_model1}, \eqref{eq:flexible_trailer_capa_model1}, \nonumber\\
& \label{eq:route_select_model2} \tau_{r} + b_{\rmb{k}} \leq 1 & \hspace{-1.3cm} 
 i \in \mathcal{O}, j \in \rmb{\mathcal{J}}, r \in \mathcal{R}_{ij}, \rmb{k} \in \rmb{\mathcal{J}}:d_{r} > \hat{d}_{i \rmb{k}},\\
 &\label{eq:open_route_model2} \tau_{r}  + h_{r} = b_{j}  &  i \in \mathcal{O}, j \in \rmb{\mathcal{J}}, r \in \mathcal{R}_{ij},\\
 &\label{eq:activate_default} h_{r} \leq \sum_{\rmb{k} \in \rmb{\mathcal{J}}:d_{r} > \hat{d}_{i \rmb{k}}} b_{\rmb{k}}  & i \in \mathcal{O}, j \in \rmb{\mathcal{J}}, r \in \mathcal{R}_{ij},\\
  \label{eq:route_assign_model2} & \sum_{j \in \rmb{\mathcal{J}}}  \sum_{r \in \mathcal{R}_{ij}} \tau_{r} \geq 1 & i \in \mathcal{O}, \\
 & \pi_{r}\sum_{\rmb{k \in \mathcal{J}}}\sum_{r' \in \mathcal{R}_{ik}} v_{r'}\tau_{r'} = v_{r}\tau_{r} & i \in \mathcal{O}, j \in \rmb{\mathcal{J}}, r \in \mathcal{R}_{ij}, \label{choice-model-const_}\\
\label{eq:unmet_model2} &   \sum_{i \in \mathcal{O}} \sum_{r \in \mathcal{R}_{ij}} \xi_{i} \pi_{r} \leq u_j(\xi) + m \sum_{\ell \in \rmb{\mathcal{L}_j}} s^{2}_{\ell j}(\xi)  & j \in \rmb{\mathcal{J}}, \xi \in \Xi,\\
\label{eq:arc_capa_model2} &   \sum_{i \in \mathcal{O}} \sum_{j \in \rmb{\mathcal{J}}} \sum_{r \in \mathcal{R}_{ij}:a \in r} \xi_{i}\pi_{r} ~\rmb{=}~ x_a(\xi) & a \in \mathcal{A}, \xi \in \Xi,\\
&\label{eq:non_neg_model22} b_j, \tau_{r} \in \{0,1\}, s^1_{\rmb{\ell}} \in \mathbb{Z}_{+}, h_{r}, \pi_{r}  \geq 0 & \hspace{-0.5cm} i \in \mathcal{O}, j \in \rmb{\mathcal{J}}, \rmb{\ell \in \mathcal{L}}, r \in \mathcal{R}_{ij},\\
&\label{eq:non_neg_model2} s^{2}_{\ell j}(\xi) \in \mathbb{Z}_{+}, u_{j}(\xi), x_a(\xi) \geq 0 & \hspace{-0.5cm} j \in \rmb{\mathcal{J}}, \ell \in \rmb{\mathcal{L}}, a \in \mathcal{A}, \xi \in \Xi.
\end{align}
\end{subequations}
\rmb{Constraints~\eqref{eq:route_select_model2}--\eqref{eq:activate_default} define which routes enter the choice set through $\tau_r$ variables.}
\rmb{Constraint~\eqref{eq:route_select_model2} forces $\tau_r=0$ if any open shelter $k$ makes route $r$ too long, i.e., if $d_r>\hat{d}_{ik}$. Constraint~\eqref{eq:open_route_model2} forces $\tau_r=0$, for route $r \in \mathcal{R}_{ij}$, when shelter $j$ is closed. Conversely, if $j$ is open and $d_r\le\hat{d}_{ik}$ for every open shelter $k$, then $h_r=0$ from~\eqref{eq:activate_default}, and~\eqref{eq:open_route_model2} forces $\tau_r=1$. Constraint~\eqref{eq:route_assign_model2} is the nonempty-choice set condition.}
  \rmb{The remaining constraints follow the \CSO{} structure, with $\pi_r$, the proportion of evacuees choosing route $r$,  replacing route assignments $y_r$. Constraint~\eqref{choice-model-const_} formulates the multinomial logit model. Constraints~\eqref{eq:unmet_model2} define unmet relief demand, and constraints~\eqref{eq:arc_capa_model2} calculate the traffic volume on arcs.}
  
  Let  $\mathcal{V} = \{\rmb{(\pi_r,\tau_r) \in \mathbb{R}_+ \times \{0,1\} }  : \eqref{eq:route_assign_model2} \text{ and } \eqref{choice-model-const_} \ \text{are satisfied for } i \in \mathcal{O}, j \in \rmb{\mathcal{J}}, r \in \mathcal{R}_{ij}\}$. Furthermore, let $\mathcal{R}_i:=\cup_{j\in\mathcal{J}}\mathcal{R}_{ij}$ be the set of routes for origin $i$. 
Constraints~\eqref{choice-model-const_} include bilinear $\pi_{r}\tau_{r'}$ terms. A direct linearization of these terms requires \rmb{$O(\sum_{i\in\mathcal{O}}|\mathcal{R}_i|^2)$} auxiliary variables and McCormick inequalities, which can significantly increase the model size. To address this challenge, we develop an efficient reformulation approach that requires $O(|\mathcal{O}|)$ additional continuous variables, and $O(\sum_{i\in\mathcal{O}}|\mathcal{R}_i|)$ McCormick inequalities. We define the following lifted mixed-binary set 
\begin{equation} \label{RlPolyB}
Q := \left\{ \rmb{(\pi,\tau,q):}
  \begin{array}{lr}
                 \pi_r = v_{r} \tau_r q_i & i \in \mathcal{O}, j \in \rmb{\mathcal{J}}, r \in \mathcal{R}_{ij},\\
                \sum_{\rmb{j \in \mathcal{J}}}\sum_{r \in \mathcal{R}_{ij}} \pi_{r} = 1 & i \in \mathcal{O},\\
                \pi_r\geq0,\ \tau_r\in\{0,1\},\ q_i \geq 0 & \quad i \in \mathcal{O}, j \in \rmb{\mathcal{J}}, r \in \mathcal{R}_{ij}.
  \end{array}
\right.
\end{equation}
The formulation of $Q$ includes an additional continuous variable $q_i$ per origin. \rmb{Let 
$\text{Proj}_{(\pi,\tau)}(Q)$ denote the projection of $Q$ onto the $(\pi,\tau)$-space.}

\begin{restatable}{proposition}{projection}
\label{projection}
\rmb{$\text{Proj}_{(\pi,\tau)}(Q) = \mathcal{V}$.}
\end{restatable}


Based on Proposition~\ref{projection}, \rmb{proven in Appendix \ref{app:projection},} we replace constraints~\eqref{eq:route_assign_model2} and~\eqref{choice-model-const_} with~\eqref{RlPolyB}.  \mb{It remains to linearize the products in~\eqref{RlPolyB}. Define $\underline v_i:=\min\{v_{\bar r}:\bar r\in\cup_{j\in\mathcal{J}}\mathcal{R}_{ij}\}>0$ and $U_{ir}:=v_r/\underline v_i$. Since~\eqref{eq:route_assign_model2} ensures at least one acceptable route for origin $i$, every feasible lifted point satisfies $q_i\le1/\underline v_i$ and hence $0\le v_r q_i\le U_{ir}$. The product of the binary variable $\tau_r$ and the bounded continuous term $v_rq_i$ is represented exactly by McCormick inequalities $\pi_r\le v_r q_i$, $\pi_r\le U_{ir}\tau_r$, and $\pi_r\ge v_rq_i-U_{ir}(1-\tau_r)$, together with $\pi_r\ge0$. Thus the exact linear reformulation uses one auxiliary variable for each origin and route-wise McCormick constraints, rather than route-pair products.}

\subsection{Volume-Delay Function \rmb{and Piecewise-Linear Approximation}} 
\label{ttf}
Traffic assignment models use a volume delay function $T(x)$ to express the
travel time on each road segment as a function of the traffic volume $x$, capacity $c$, and free-flow time $t^0$. \rmb{Such functions typically multiply} $t^0$ \rmb{by} a congestion 
\rmb{factor that} increases \mb{with} the volume-to-capacity ratio ~\citep{spiess1990conical}. Many different volume delay functions have been proposed in the literature~\citep{saric2018volume}.
 The BPR function~\citep{BPR1964}, $T(x)= t^0 \Big(1 + \alpha \Big( \frac{x}{c} \Big)^{\beta}\Big),$
is widely used in evacuation planning literature due to its minimal data requirements and convexity~\citep{kulshrestha2011,bayram2018,liang2019}. We model the evacuation time using the BPR function with $\alpha=0.15$ and $\beta=4$~\citep{bayram2015}.

Objective functions of the upper- and lower-bounding subproblems in our solution approach include $\sum_{a \in \mathcal{A}}T_a(x_a(\xi))x_a(\xi)$  for every $\xi$. Optimizing this nonlinear term in each iteration can \rmb{substantially increase solution time.}
\cite{bayram2015} reformulated the BPR function using second-order conic programming (SOCP) constraints.  
However, for $\beta=4$, this reformulation requires introducing four additional variables and three conic constraints for each $a \in \mathcal{A}, \ \xi \in \Xi$ pair, \rmb{which can limit the size of the instances that can be solved within practical time limits.}
Alternatively, the BPR function $T(x)$ can be approximated by a piecewise-linear function. However, previous studies \rmb{applied this} approach using a different set of breakpoints for each arc's flow~\citep{wei2019efficient, afkham2022balancing}. We propose a new piecewise-linear approximation that uses a single set of breakpoints for the flow-to-capacity ratio across the entire network.  For $\rmb{\rho}=x/c$, we have
\begin{align*}
&T(x)x =  t^0 \Big(1 + \alpha \Big( \frac{x}{c} \Big)^{\beta}\Big)x = t^0c \Big(\frac{x}{c}\Big) + t^0 c \alpha \Big(\frac{x}{c}\Big)^{\beta+1}= t^0 c \left[ \rmb{\rho} + \alpha \rmb{\rho}^{\beta+1} \right].
\end{align*}
Let $g(\rmb{\rho})= \rmb{\rho} + \alpha \rmb{\rho}^{\beta+1}$ and $\{(\rmb{\rho}_{\PWLIdx}, g_{\PWLIdx})\}_{\PWLIdx \in \PWLSet}$\rmb{, where $\PWLSet=\{1,\ldots,\PWLCount\}$,} denote a set of breakpoints for the piecewise-linear approximation of $g(\rmb{\rho})$ over $[0,\rmb{\rho}_{\PWLCount}]$ such that $0=\rmb{\rho}_0 <\rmb{\rho}_1 < \rmb{\rho}_2 < \ldots <  \rmb{\rho}_{\PWLCount}$, and $g_{\PWLIdx}=g(\rmb{\rho}_{\PWLIdx})$. Here, $\rmb{\rho}_{\PWLCount}$ is the maximum flow-to-capacity ratio across all arcs. Furthermore, let ${\PWLSlope}_{\PWLIdx}$ be the slope of line segment $\PWLIdx \in \PWLSet$. 
\rmb{Since \(g(\rho)\) is strictly increasing on \([0,\infty)\), each segment slope is strictly positive: \(\PWLSlope_{\PWLIdx}=\frac{g(\rho_{\PWLIdx})-g(\rho_{\PWLIdx-1})}
{\rho_{\PWLIdx}-\rho_{\PWLIdx-1}}>0.\) Moreover, because \(g\) is strictly convex, these slopes are increasing, i.e., \( 0<\PWLSlope_1< \PWLSlope_2<\cdots<\PWLSlope_{\PWLCount}. \)}
We define decision variables $x^{\PWLIdx}\geq 0$  to represent the traffic volume associated with segment $\PWLIdx \in \PWLSet$ such that $x=\sum_{\PWLIdx \in \PWLSet}x^{\PWLIdx}$. Then, the piecewise-linear approximation of $T(x)x$ is given by $t^0\sum_{\PWLIdx \in \PWLSet} {\PWLSlope}_{\PWLIdx} x^{\PWLIdx}$. \rmb{With this approximation, the \CSO \ objective becomes}
\begin{align}
z_{\CSO} \approx \rmb{z_{\PWLCSO} := } & \min \  \max_{\xi \in \Xi} \  \sum_{j \in \rmb{\mathcal{J}}} u_j(\xi) + \lambda \sum_{a \in \mathcal{A}}  \sum_{\PWLIdx \in \PWLSet}t^0_a \PWLSlope_{\PWLIdx} x^{\PWLIdx}_a(\xi).
\label{obj-function_model1_pwl} 
\end{align}
Furthermore, for $\PWLDelta_{\PWLIdx}  =\rmb{\rho}_{\PWLIdx}-\rmb{\rho}_{\PWLIdx-1}, \forall \PWLIdx \in \PWLSet$, we replace constraints~\eqref{eq:arc_capa_model1} with
\begin{subequations}
\begin{align}
    \label{eq:arc_capa_model1_pwl} &   \sum_{i \in \mathcal{O}} \sum_{j \in \rmb{\mathcal{J}}} \sum_{r \in \mathcal{R}_{ij}:a \in r} \xi_{i}y_{r} = \sum_{\PWLIdx \in \PWLSet}x^{\PWLIdx}_a(\xi) & a \in \mathcal{A}, \xi \in \Xi,\\
        & 0 \leq x^{\PWLIdx}_a(\xi) \leq c_a \PWLDelta_{\PWLIdx}  & \PWLIdx \in \PWLSet\setminus \{\PWLCount\}, a \in \mathcal{A}, \xi \in \Xi. \label{eq:arc_capa_model1_pwlbound}
\end{align}
\end{subequations}

\rmb{Thus, each original arc traffic load variable $x_a(\xi)$ is approximated by $\PWLCount$ segment load variables $x_a^{\PWLIdx}(\xi)$. This substitution yields linear objective terms and linear traffic flow constraints, enhancing the tractability of the bounding subproblems in our computational experiments. In what follows, we will refer to the resulting piecewise-linear adaptive robust formulation as \PWLCSO. Similarly, after applying the same substitution in the \Pref \ model, we obtain the \PWLURC \ model.}

\rmb{The construction has two modeling implications. First, as the segment slopes are nondecreasing, the linear minimization over segment flows fills lower-slope segments before higher-slope segments. For arc loads in $[0,\rho_{\PWLCount}c_a]$, the resulting value is the linear interpolation between consecutive breakpoints of the BPR  function; since this function is convex, the interpolation lies above the true curve on each segment. Second, the last segment is left unbounded, i.e., not considered in~\eqref{eq:arc_capa_model1_pwlbound}. This preserves feasibility when a route assignment creates an arc load above the largest breakpoint, consistent with treating capacity through a soft congestion penalty rather than a hard flow cap. If an arc load exceeds $\rho_{\PWLCount}c_a$, the model returns a value by extending the last linear segment.}

{\color{black}

\section{Complexity Insights}
\label{sec:complexity}

The proposed evacuation planning models belong to the class of two-stage adaptive robust optimization problems with mixed-integer adaptive recourse, a fixed recourse matrix, and an uncertain technology matrix. We first show that this general problem family is computationally difficult, and then analyze the \PWLCSO{} model to identify where the main challenges arise and which substructures remain tractable. We use the finite rational encoding convention\footnote{That is, the complexity results are stated for finitely encoded decision problems: numerical data are rational and binary encoded; route sets, uncertainty set descriptions, and piecewise-linear data are explicitly part of the input; and continuous decisions and uncertainty realizations are represented by rational certificates of polynomial encoding length when polynomial-hierarchy membership is discussed. Under this convention, fixed certificates can be checked in polynomial time.} in complexity class results summarized in Table~\ref{tab:complexity_summary}. 
All proofs are presented in ~\ref{app:complexity_proofs}.

\begin{table}[!htbp]
\centering
\caption{Summary of formal complexity results in Section~\ref{sec:complexity}.}
\label{tab:complexity_summary}
\renewcommand{\arraystretch}{0.85}
\scalebox{0.8}{
\begin{tabular}{p{5.4cm} p{8.2cm} l}
\toprule
\textbf{Scope} & \textbf{Result} & \textbf{Reference} \\
\midrule
General ARO problem class & $\Sigma^p_3$-complete & Prop.~\ref{prop:rhs-to-coeff-ARO} \\
Adaptive recourse problem & Polynomial-time via min-cost network flow & Prop.~\ref{prop:complexity-recourse} \\
Deterministic problem & NP-complete via $p$-median & Prop.~\ref{prop:complexity-det} \\
Robust verification & coNP-hard via CLIQUE & Prop.~\ref{prop:complexity-fixedplan} \\
\midrule
Full \PWLCSO{} with &  $\Sigma^p_2$-hard & Prop.~\ref{prop:complexity-sigma2} \\
bounded polyhedral & In $\Sigma^p_2$ & Prop.~\ref{prop:complexity-sigma2-member} \\
uncertainty set & $\Sigma^p_2$-complete & Cor.~\ref{cor:complexity-sigma2-complete} \\
\bottomrule
\end{tabular}
}
\end{table}

\begin{restatable}{proposition}
{proprhstocoeffARO}
\label{prop:rhs-to-coeff-ARO}
The feasibility problem for the class of two-stage robust programs with linear first-stage constraints, mixed-integer adaptive recourse, a deterministic recourse matrix, deterministic right-hand sides, and continuous budgeted uncertainty that enters affinely in the technology matrix is $\Sigma^p_3$-complete.
\end{restatable}

The result is shown by a reduction from the right-hand-side uncertain ARO family of ~\cite{goerigk2024complexity}, and provides the first insight into why the general class of two-stage adaptive robust optimization problems considered in this paper is computationally difficult. However, \PWLCSO \ has a more specific structure than the general class. In particular, uncertainty enters through nonnegative traffic assignment terms ($\xi_i y_r$);  evacuation routes are explicitly specified as inputs; and adaptive recourse is a shipment problem. This motivates a model-specific complexity analysis to understand which parts of \PWLCSO{} inherit the difficulty of the general two-stage robust optimization problem class and which parts exhibit more exploitable structures.

\begin{definition}
\label{def:PWLCSO}
Given an instance of \PWLCSO{} and a threshold
\(\eta\), the decision problem asks whether there exist feasible first-stage decisions
\((b,s^1,y)\) 
\rmb{whose induced worst-case adaptive objective value is}
at most \(\eta\).
\end{definition}

\subsubsection*{Hardness of the recourse problem.} The inner recourse problem is not the main source of difficulty in \CSO. The first-stage decisions and realized evacuee counts determine the relief demand at shelters and the travel time on each arc. The only remaining adaptive decision is allocating prepositioned relief supplies to shelters, a deterministic shipment problem aimed at minimizing unmet demand. We show that this problem has a network-flow structure.
\begin{restatable}{proposition}
{propcomplexityrecourse}
\label{prop:complexity-recourse}
Given first-stage decisions  $(b,s^1,y)$ and a realization of evacuee counts $\xi$, the optimal recourse decisions in \CSO\ can be computed in polynomial time.
\end{restatable}
As a result, the computational difficulty of \CSO{} lies not in distributing relief supplies under a realized scenario, but in choosing first-stage decisions that perform well across all realizations.

\subsubsection*{Hardness of the \PWLCSO{}.} We next show the difficulty of the specific decisions at different levels of the \PWLCSO{}\ and the hardness of the overall problem. In particular, the first result demonstrates the difficulty of first-stage decisions even in the absence of uncertainty, while the second result shows the difficulty of verifying the worst-case performance of a given set of first-stage decisions under uncertainty in the evacuee counts. 
\begin{restatable}{proposition}
{propcomplexitydet}
\label{prop:complexity-det}
The deterministic decision version of \PWLCSO \ is NP-complete.
\end{restatable}
Consequently, optimizing the first-stage decisions in \PWLCSO\  requires solving an NP-hard problem even in the absence of uncertainty. We next show that verifying whether the total worst-case evacuation time remains below a given threshold for fixed first-stage decisions is coNP-hard.

\begin{restatable}{proposition}
{propcomplexityfixedplan}
\label{prop:complexity-fixedplan}
For fixed first-stage decisions in \PWLCSO{}, verifying whether the total worst-case evacuation time remains below a given threshold is coNP-hard.
\end{restatable}

Propositions \ref{prop:complexity-det} and \ref{prop:complexity-fixedplan} demonstrate that both \rmb{optimizing} 
the first-stage decisions and verifying their robustness are 
\rmb{computationally} 
difficult in \PWLCSO{}. The following results yield a specific complexity classification for \PWLCSO.

\begin{restatable}{proposition}
{propcomplexitysigmahard}
\label{prop:complexity-sigma2}
\rmb{The decision version of \PWLCSO{} with a bounded polyhedral uncertainty set is $\Sigma^p_2$-hard.}
\end{restatable}

\rmb{The proof reduces from the quantified complement-form 
 threshold problem for Max-3SAT. Shelter-opening decisions encode the existential variables, the extreme points of the uncertainty set encode the universal variables, and the PWL approximation of the evacuation time is used to count the satisfied clauses.}
\rmb{This hardness result is next matched by a corresponding upper-bound result.}

\begin{restatable}{proposition}
{propcomplexitysigmamember}
\label{prop:complexity-sigma2-member}
\rmb{The decision version of \PWLCSO{} with a 
polyhedral uncertainty set is in $\Sigma^p_2$.}
\end{restatable}

The first-stage decisions constitute the existential certificate, while the realization of uncertainty serves as the universal certificate. Once they are fixed, the predicate determining whether a scenario value exceeds a given threshold can be verified by computing the optimal adaptive recourse in polynomial time by Proposition~\ref{prop:complexity-recourse}. 
\rmb{Combining membership with hardness yields the completeness classification.}

\begin{restatable}{corollary}
{corcomplexitysigmacomplete}
\label{cor:complexity-sigma2-complete}
\rmb{The decision version of \PWLCSO{} with a bounded polyhedral uncertainty set is $\Sigma^p_2$-complete.}
\end{restatable}
}

These complexity insights provide a concrete rationale for our solution approach. Since a central difficulty in  \PWLCSO{} arises due to the coupling between the first-stage decisions and adversarial evacuee count realizations, Section~\ref{S4:solution_approach} proposes two complementary bounding subproblems of the fully adaptive robust model: a finite-scenario relaxation that restricts \(\Xi\) to specific realizations to generate lower bounds, and a partition-based finitely adaptable restriction that assigns one recourse decision vector to each partition of \(\Xi\) to generate upper bounds. These subproblems do not eliminate the difficulty of the first-stage decisions; rather, in the piecewise-linear formulation used here, they yield MILP-representable bounding problems that are amenable to general-purpose mixed-integer optimization solvers.

\section{Solution Method}\label{S4:solution_approach}
A widely applied exact framework for two-stage adaptive robust optimization is constraint-and-column generation (CCG)~\citep{zeng2013}. In models with continuous recourse, the adversarial separation problem of CCG can be reformulated as a single-level problem by applying strong duality or optimality conditions to the inner minimization. In our setting, the adaptive pallet-shipment variables \(s^{2}_{\ell j}\) are integer-valued, so these duality-based reformulations are not directly applicable, and the separation problem remains computationally difficult (Proposition~\ref{prop:complexity-fixedplan}). 
A widely used approximation framework for two-stage adaptive robust optimization is the $K$-adaptability approach. Rather than optimizing a fully adaptive recourse policy, 
it determines $K$ candidate recourse decisions in advance, and the best feasible decision is selected after the uncertainty is realized. This induces an implicit partition of the uncertainty set~\citep{subramanyam2020k}.

We implement a partition-and-bound method that explicitly splits the uncertainty set into smaller subsets at each iteration and optimizes a separate recourse decision for each subset.  The computational efficiency and solution quality of this approach are primarily determined by the uncertainty set splitting strategy~\citep{dunning2016paper}. We design
 a targeted splitting rule to refine the recourse decisions at each iteration. The proposed rule selects an {\it active partition} whose value attains the current upper bound and splits {\it only} that partition into two child partitions using critical scenarios obtained from a slack-minimization subproblem over the second-stage constraints.
 The solution method is not tied to a particular uncertainty set; it is applicable whenever the corresponding bounding problems and slack-minimization subproblems can be formulated and solved.

\subsection{Partition-and-Bound Algorithm}

Figure \ref{F:new-problem-description} presents the main steps of our partition-and-bound algorithm, 
which maintains a binary tree: the root node represents the full uncertainty set $\Xi$, and the leaves represent its partitions. Each iteration solves the finite-scenario relaxation~\eqref{zLBmodel} to update the lower bound (Section~\ref{lowerboundsection}) and the finitely adaptable robust problem~\eqref{zpartmodel} to update the upper bound and incumbent first-stage plan (Section~\ref{upperboundsection}). If the termination criterion is not met, the algorithm selects an active leaf partition whose value attains the current upper bound, identifies two critical scenarios in that partition, and splits the active leaf into two child partitions (Section~\ref{selectactive}).

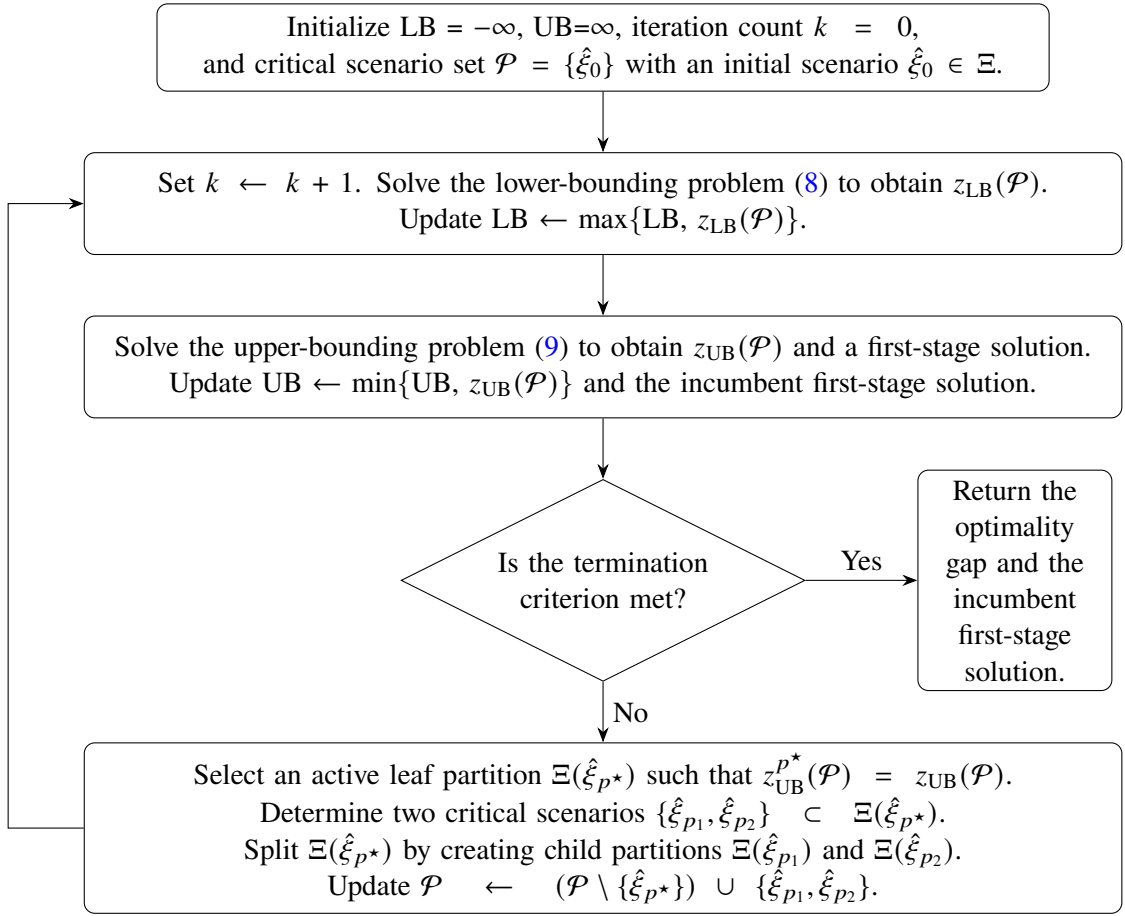
\begin{figure}[h!]
    \centering    
    \small
    \begin{tikzpicture}[
        node distance = 0.5cm and 0.5cm, 
        auto,
        process/.style={
            rectangle, 
            draw,  
            text width=35em, 
            text centered, 
            rounded corners, 
            minimum height=3.5em
        },
        process1/.style={
            rectangle, 
            draw,  
            text width=4em, 
            text centered, 
            rounded corners, 
            minimum height=3em
        },
        decision/.style={
            diamond, 
            draw, 
            text width=8em, 
            text centered, 
            minimum height=2em,
            aspect=2 
        },
        terminator/.style={
            rectangle, 
            draw, 
            text width=30em, 
            text centered, 
            rounded corners, 
            minimum height=3em
        },
        line/.style={
            draw, 
            -{Stealth[length=2mm]}
        }
    ]

    \node [terminator] (init) {Initialize LB = $-\infty$, UB=$\infty$, iteration count $k=0$, \\ and critical scenario set $\mathcal{P} = \{\hat{\xi}_0\}$ with an \rmb{initial scenario} $\hat{\xi}_0 \in \Xi$.};
    
    \node [process, below=of init, yshift=-0.3cm] (LB_ARMIO) {
    Set $k \leftarrow k+1$. Solve the lower-bounding problem~\eqref{zLBmodel} to \rmb{obtain} $z_\text{LB}(\mathcal{P})$.\\ Update LB \(\rmb{\leftarrow}\) max\{LB, $z_\text{LB}(\mathcal{P})$\}.};
    
    \node [process, below=of LB_ARMIO, yshift=-0.3cm] (UB_ARMIO) {Solve the upper-bounding problem~\eqref{zpartmodel} to \rmb{obtain} $z_\text{UB}(\mathcal{P} )$ and a first-stage solution.\\ Update UB \(\rmb{\leftarrow}\) min\{UB, $z_\text{UB}(\mathcal{P})$\} and the incumbent first-stage solution.
    };

    \node [decision, below=of UB_ARMIO, yshift=-0.3cm] (decision) {\rmb{Is the termination criterion met?}};
           
    \node [terminator, right=of decision, xshift=0.98cm, text width=6em] (terminate) {Return the optimality gap and the incumbent first-stage solution.};

    \node [process, below=of decision, yshift=-0.3cm] (split) {
    \rmb{Select an active leaf partition $\Xi(\hat{\xi}_{p^\star})$ such that $z_\text{UB}^{p^\star}(\mathcal{P})=z_\text{UB}(\mathcal{P})$.}\\
    Determine two critical scenarios $\{\hat{\xi}_\rmb{p_1},\hat{\xi}_\rmb{p_2}\} \subset \Xi(\hat{\xi}_\rmb{p^\star})$.\\  
    Split $\Xi(\hat{\xi}_\rmb{p^\star})$ by creating child partitions  $\Xi(\hat{\xi}_\rmb{p_1})$ and $\Xi(\hat{\xi}_\rmb{p_2})$. \\ \rmb{Update} $\mathcal{P} \leftarrow \rmb{(\mathcal{P}\setminus\{\hat{\xi}_\rmb{p^\star}\})} \cup \{ \hat{\xi}_\rmb{p_1},\hat{\xi}_\rmb{p_2}\}$.};

    \path [line] (init) -- (LB_ARMIO);
    \path [line] (LB_ARMIO) -- (UB_ARMIO);
    \path [line] (UB_ARMIO) -- (decision);
    \path [line] (decision) -- node[above] {\rmb{Yes}} (terminate);
    \path [line] (decision) -- node[right] {\rmb{No}} (split);
    \path [line] (split.west) -- ++(-1,0) |- (LB_ARMIO.west);
    
    \end{tikzpicture}
    \caption{The flowchart of the partition-and-bound algorithm.}
    \label{F:new-problem-description}
\end{figure}

The process of solving an example problem over three iterations is illustrated in Figure \ref{F:partition-and-bound-example}, depicting the binary search tree (left) and the partitioning of the uncertainty set $\Xi$ (right). We only present the generation of upper-bounds for clarity. In the first iteration, the upper-bounding problem~\eqref{zpartmodel} is solved at the root node without partitioning the uncertainty set to obtain a first-stage solution and an initial upper bound. Two critical scenarios, $\{\hat{\xi}_1, \hat{\xi}_2\}$, are identified based on the first-stage solution. The uncertainty set $\Xi$ is partitioned by creating two child nodes for partitions  $\Xi(\hat{\xi}_1)$ and $\Xi(\hat{\xi}_2)$.
In the second iteration, the upper-bounding problem is solved over these two partitions, yielding a first-stage solution and upper bounds of $z_{\text{UB}}^1 = 720$ for Node 1 and $z_{\text{UB}}^2 = 980$ for Node 2. Since $z_{\text{UB}}^2$ is larger, Node 2 is selected as the active partition for further partitioning, and critical scenarios $\{\hat{\xi}_3, \hat{\xi}_4\}$ are identified. This partition is then split, creating Nodes 3 and 4. In the third iteration, the upper-bounding problem is solved over the three current partitions $\{\Xi(\hat{\xi}_1), \Xi(\hat{\xi}_3), \Xi(\hat{\xi}_4)\}$. Node 3 yields the highest upper bound, and is further partitioned into Nodes 5 and 6. The global upper bound UB after three iterations is determined by $z^5_{\text{UB}} = 870$ at Node 5, the highest upper bound among the leaf nodes. We now describe the details of each step. We present the upper- and lower-bounding subproblems for the route assignment model~\eqref{originalmodel}. The subproblems for the route choice model~\eqref{prefmodel} are presented in the ~\ref{app:subprob_choice_model}.

\begin{figure}[t!]
\centering
\includegraphics[scale=0.77]{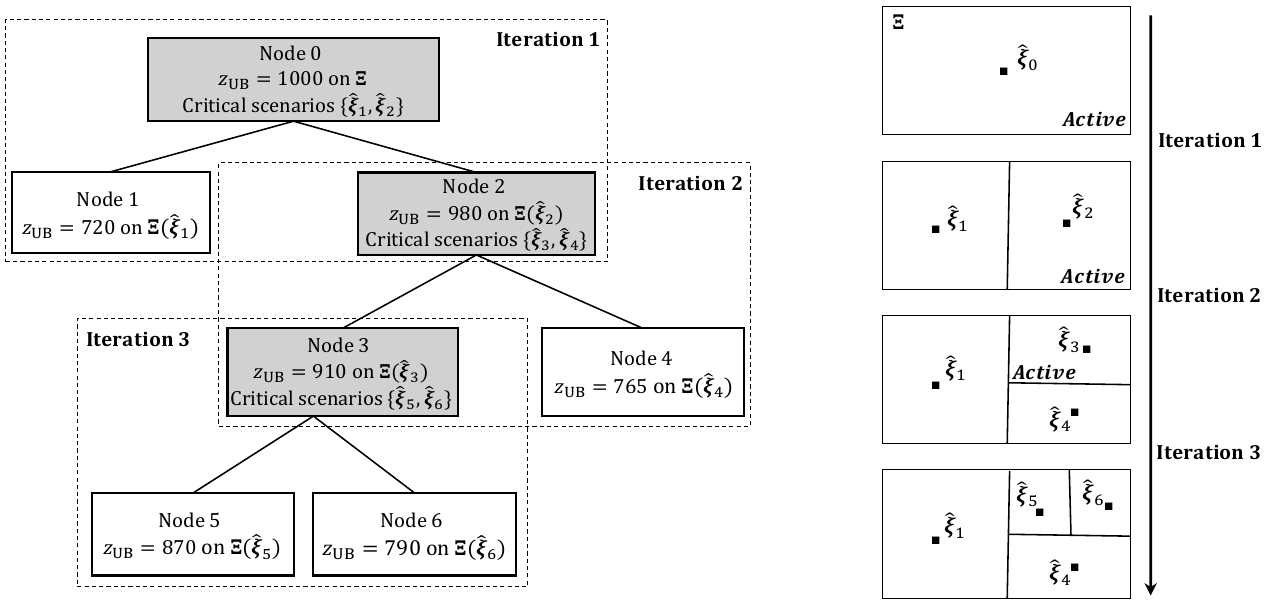}
\caption{An example partition-and-bound tree. The shaded nodes denote the active partition in each iteration.}
\label{F:partition-and-bound-example}
\end{figure}

\subsection{The Lower-Bounding Problem}\label{lowerboundsection}
Given a \rmb{finite} set of critical scenarios $\mathcal{P}=\{\hat{\xi}_p\}_{p \in P} \subset \Xi$, 
we obtain a lower bound by solving the following 
\rmb{two-stage stochastic mixed-integer program, denoted by \text{LB-ARMIO}}:
\begin{subequations}\label{zLBmodel}
\begin{align}
z_{\text{LB}}(\mathcal{P}) :=  \min \ &  z \\
\text{s.t.} \ & \eqref{eq:tot_open_shelter_model1}-\eqref{eq:route_restict_model1} \notag\\
& \sum_{j \in \mathcal{J}} u_{jp} + \lambda \sum_{a \in \mathcal{A}} \rmb{\sum_{\PWLIdx \in \PWLSet}t^0_a \PWLSlope_{\PWLIdx}x^{\PWLIdx}_{ap}} \leq z \label{LB-obj-function} & p \in P,\\
\label{eq:LB-ch4_rob_inv} &   \sum_{i \in \mathcal{O}} \sum_{r \in \mathcal{R}_{ij}} \hat{\xi}_{pi}y_{r} \leq  u_{jp} + m\sum_{\ell \in \rmb{\mathcal{L}_j}} s^{2}_{\ell jp}  & j \in \rmb{\mathcal{J}}, p \in P,\\
\label{eq:LB-arc_capa_model1_pwl} &   \sum_{i \in \mathcal{O}} \sum_{j \in \rmb{\mathcal{J}}} \sum_{r \in \mathcal{R}_{ij}:a \in r} \hat{\xi}_{pi}y_{r} \leq 
\rmb{\sum_{\PWLIdx \in \PWLSet}x^{\PWLIdx}_{ap}}
& a \in \mathcal{A}, p \in P,\\
& \rmb{x^{\PWLIdx}_{ap} \leq c_a\PWLDelta_{\PWLIdx}} & \rmb{\PWLIdx \in \PWLSet\setminus\{\PWLCount\}, a \in \mathcal{A}, p \in P,}\\
&\label{eq:LB-ch4_rob_s2}\sum_{\rmb{j \in \mathcal{J}} : \ell \in \rmb{\mathcal{L}_j}} s^{2}_{\ell jp} \leq K s^{1}_{\ell} &  \rmb{\ell \in \mathcal{L}}, p \in P,\\
\label{eq:LB-ch4_rob_non_neg1} & b_j \in \{0,1\}, s^1_{\rmb{\ell}} \in \mathbb{Z}_{+} , y_{r} \geq 0 & i \in \mathcal{O}, j \in \rmb{\mathcal{J}}, \rmb{\ell \in \mathcal{L}}, r \in \mathcal{R}_{ij},\\
\label{eq:LB-ch4_rob_non_neg2} & s^{2}_{\ell jp} \in \mathbb{Z}_{+}, u_{jp}, \rmb{x^{\PWLIdx}_{ap}} 
\geq 0 & \hspace{-2cm} j \in \rmb{\mathcal{J}}, \ell \in \rmb{\mathcal{L}}, a \in \mathcal{A}, p \in P, \rmb{\PWLIdx \in \PWLSet}.
\end{align}
\end{subequations}

LB-ARMIO contains recourse variables and constraints for each scenario $p\in P$. The epigraph variable $z$ and constraints~\eqref{LB-obj-function} represent the worst-case objective value over $p\in P$. Since $\mathcal{P}\subset\Xi$, this is a relaxation of the fully adaptive robust model; consequently, $z_{\text{LB}}(\mathcal{P}) \leq z_{\PWLCSO}$.

\subsection{The Upper-Bounding Problem}\label{upperboundsection}

\rmb{The critical scenario set $\mathcal{P}=\{\hat{\xi}_p\}_{p \in P}\subset\Xi$ indexes the leaves of the partition tree  denoted by $\{\Xi(\hat{\xi}_p)\}_{p\in P}$. The leaves are maintained as a cover of the uncertainty set such that $\bigcup_{p\in P}\Xi(\hat{\xi}_p)=\Xi$. At initialization, $\hat{\xi}_0$ indexes the root leaf $\Xi(\hat{\xi}_0)=\Xi$. When an active leaf $\Xi(\hat{\xi}_{p^\star})$ is split, it is replaced by two child leaves defined by the distance-based split in Section~\ref{selectactive}; these child partitions cover the parent partition. Hence, by induction, every $\xi\in\Xi$ belongs to at least one leaf partition at every iteration. Boundary scenarios may be assigned to leaves by any fixed rule.}

\rmb{We generate upper bounds by solving model~\eqref{zpartmodel}, denoted by UB-ARMIO, which is a partition-based finitely adaptable restriction of the \CSO{}. For each leaf partition $\Xi(\hat{\xi}_p)$, the formulation chooses \emph{static} relief recourse decisions $\{u_{jp},(s^2_{\ell jp})_{\ell \in \mathcal L _j}\}$ for each shelter $j \in \mathcal{J}$, and a \emph{static} PWL segment vector $(x^{\PWLIdx}_{ap})_{\PWLIdx \in \PWLSet}$ for each arc $a \in \mathcal{A}$. These decisions must be feasible for \emph{every} realization in $\Xi(\hat{\xi}_p)$. The epigraph variable $z$ represents the worst-case value over the current leaf partitions. Any feasible solution to UB-ARMIO yields a feasible policy for the original fully adaptive model. Consequently, $z_{\text{UB}}(\mathcal{P})\ge z_{\PWLCSO}$.}
\begin{subequations}\label{zpartmodel}
\begin{align}
\label{eq:ch4_rob_part_objfun} 
z_{\text{UB}}(\mathcal{P}) := \min \ & z \\
\text{s.t.} \ & \eqref{eq:tot_open_shelter_model1}-\eqref{eq:route_restict_model1} \notag\\
& \sum_{j \in \rmb{\mathcal{J}}} u_{jp} + \lambda \sum_{a \in \mathcal{A}} \rmb{\sum_{\PWLIdx \in \PWLSet}t^0_a \PWLSlope_{\PWLIdx}x^{\PWLIdx}_{ap}} 
\leq z \label{UB-obj-function} & p \in P,\\
\label{eq:unmet_model_UB} &   \sum_{i \in \mathcal{O}} \sum_{r \in \mathcal{R}_{ij}} \xi_{i}y_{r} \leq u_{jp} + m \sum_{\ell \in \rmb{\mathcal{L}_j}} s^{2}_{\ell jp}  & j \in \rmb{\mathcal{J}}, p \in P, \xi \in \Xi(\hat{\xi}_p),\\
\label{eq:arc_capa_model1_pwl_UB} &   \sum_{i \in \mathcal{O}} \sum_{j \in \rmb{\mathcal{J}}} \sum_{r \in \mathcal{R}_{ij}:a \in r} \xi_{i}y_{r} \leq \rmb{\sum_{\PWLIdx \in \PWLSet}x^{\PWLIdx}_{ap}} 
& a \in \mathcal{A}, p \in P, \xi \in \Xi(\hat{\xi}_p),\\
&\label{eq:pwl_segment_bound_UB} \rmb{x^{\PWLIdx}_{ap} \leq c_a\PWLDelta_{\PWLIdx}} & \rmb{\PWLIdx \in \PWLSet\setminus\{\PWLCount\}, a \in \mathcal{A}, p \in P,}\\
&\label{eq:flexible_trailer_capa_model1_UB}\sum_{\rmb{j \in \mathcal{J}} : \ell \in \rmb{\mathcal{L}_j}} s^{2}_{\ell jp} \leq K s^{1}_{\ell} &  \ell \in \rmb{\mathcal{L}}, p \in P,\\
\label{eq:non_neg_model1_UB} & b_j \in \{0,1\}, s^1_{\rmb{\ell}} \in \mathbb{Z}_{+}, y_{r} \geq 0 & i \in \mathcal{O}, j \in \rmb{\mathcal{J}}, \rmb{\ell \in \mathcal{L}}, r \in \mathcal{R}_{ij},\\
&\label{eq:non_neg_model1_UB_2}  s^{2}_{\ell jp} \in \mathbb{Z}_{+}, u_{jp}, \rmb{x^{\PWLIdx}_{ap}}
\geq 0 &\hspace{-0.5cm} j \in \rmb{\mathcal{J}}, \ell \in \rmb{\mathcal{L}_j}, a \in \mathcal{A}, p \in P, \rmb{\PWLIdx \in \PWLSet}.
\end{align}
\end{subequations}

For continuous or otherwise implicitly described uncertainty sets, constraints~\eqref{eq:unmet_model_UB}--\eqref{eq:arc_capa_model1_pwl_UB} are semi-infinite. Since the uncertain parameters enter linearly, such constraints are equivalent to bounding the support function of the current leaf partition in the corresponding coefficient direction. Thus, whenever these bounding problems admit exact tractable reformulations, as is the case for many common polyhedral or conic-representable uncertainty sets, the semi-infinite constraints can be replaced by finite LP, SOCP, or more general conic counterparts \rmb{see, e.g., \citet[Table~1]{gorissen2015practical}}. In our case study, the continuous budgeted uncertainty set and the distance-based splits yield polyhedral leaf partitions; ~\ref{app:reformulation} shows the LP-duality reformulation for 
\eqref{eq:unmet_model_UB} and~\eqref{eq:arc_capa_model1_pwl_UB}.

\subsection{Determining the Active Partition and the Critical Scenarios}\label{selectactive}
\rmb{In the optimum solution of UB-ARMIO, let $(\bar{b}_j,\bar{y}_r)$ be the first-stage shelter-opening and route-assignment decisions, and $(\bar{x}^{\PWLIdx}_{ap},\bar{s}^{2}_{\ell jp},\bar{u}_{jp})$ denote the recourse solution for each leaf partition $p\in P$. The value associated with partition $p$ is given by:}
\[
\rmb{z_{\text{UB}}^p(\mathcal{P})=\sum_{j\in\mathcal{J}}\bar{u}_{jp}+\lambda\sum_{a\in\mathcal{A}}\sum_{\PWLIdx\in\PWLSet}t_a^0\PWLSlope_{\PWLIdx}\bar{x}^{\PWLIdx}_{ap}.}
\]
\rmb{The epigraph constraint~\eqref{UB-obj-function} ensures that the current $z_{\text{UB}}(\mathcal{P}) = \max_{p \in P}z_{\text{UB}}^{p}(\mathcal{P})$. We define an \textit{active partition} as any leaf partition $\Xi(\hat{\xi}_{p^\star})$, $p^\star\in P$ that satisfies $z_{\text{UB}}^{p^\star}(\mathcal{P})=z_{\text{UB}}(\mathcal{P})$. Reducing the value of inactive partitions alone cannot improve the current upper bound. Therefore, in each iteration of the partition-and-bound algorithm,  we split an active partition $\Xi(\hat{\xi}_{p^\star})$ by identifying \textit{critical scenarios} 
 that make the recourse costly in that partition.}
If recourse decisions in the child partitions are allowed to adapt to these critical scenarios, the worst-case objective values of the child partitions may decrease.

When recourse variables are continuous, critical scenarios that determine the objective value of a partition can be identified by active second-stage constraints~\citep{postek2016multistage}. However, if binary or integer recourse variables exist, such as \(s^{2}_{\ell jp}\) in our problem, the optimal recourse solution may not satisfy any second-stage constraints at equality for any scenario. \cite{dunning2016paper} proposed splitting the uncertainty set using the critical scenarios that minimize the slack of second-stage constraints in each leaf partition $p \in P$. Their algorithm splits each leaf partition into multiple child partitions in every iteration; however, creating too many child partitions may reduce computational efficiency. 
\cite{romeijnders2021piecewise} introduced a branch-and-bound-based critical-scenario detection method for integer adjustable robust optimization, together with an optimality criterion indicating when further partitioning cannot improve the piecewise constant decision rule. Their framework assumes a bounded polyhedral uncertainty set and affine dependence on the uncertain parameters.

Our method splits an active partition in each iteration. We identify critical scenarios in $\Xi(\hat{\xi}_\mb{p^\star})$ that minimize the slack of the second-stage constraints. 
Although both~\eqref{eq:unmet_model_UB} and \eqref{eq:arc_capa_model1_pwl_UB} can be used for this task, after preliminary computational experiments, we chose to focus on the unmet demand constraint~\eqref{eq:unmet_model_UB}. 
\rmb{Let $\mathcal{J}^+=\{j\in\mathcal{J}:\bar{b}_j=1\}$ be the set of shelters opened by the current UB-ARMIO solution. For each $j\in\mathcal{J}^+$, we select an uncertainty realization by solving:}
\begin{align} \label{findcriticalscenario}
\rmb{\tilde{\xi}_{p^\star}^{j} \in} \argmin_{\xi \in \Xi(\hat{\xi}_\rmb{p^\star})} \Big\{ \bar{u}_{j\rmb{p^\star}} + m\sum_{\ell \in \rmb{\mathcal{L}_j}} \bar{s}^{2}_{\ell j\rmb{p^\star}}-\sum_{i \in \mathcal{O}} \sum_{r \in \mathcal{R}_{ij}} \xi_{i}\bar{y}_{r}\Big\}.
\end{align} 
\rmb{The objective in~\eqref{findcriticalscenario} is the slack of the unmet-demand constraint for shelter $j$ in the active partition, evaluated at the current UB-ARMIO solution. 
Any element in the argmin set can be selected. The set of selected scenarios is denoted by $\mathcal{P}^{\mathrm{crit}}_{p^\star}:=\{\tilde{\xi}_{p^\star}^{j}:j\in\mathcal{J}^+\}$.
}
\rmb{We choose two critical scenarios from $\mathcal{P}^{\mathrm{crit}}_{p^\star}$ to split the active partition. In particular, we order the pairs $(j,\tilde{\xi}_{p^\star}^{j})$, $j\in\mathcal{J}^+$, in nonincreasing order of $\bar{u}_{j p^\star}$ and keep the first two distinct scenarios, denoted by $\hat{\xi}_{p_1}$ and $\hat{\xi}_{p_2}$. If only one scenario is selected from $\mathcal{P}^{\mathrm{crit}}_{p^\star}$, any other scenario in the active partition can be used as the second critical scenario. 
Given $\hat{\xi}_{p_1}$ and $\hat{\xi}_{p_2}$, the active partition is replaced by the two child partitions:}
\[
\begin{aligned}
\Xi(\hat{\xi}_{p_1}) &= \{\xi\in\Xi(\hat{\xi}_{p^\star}):\|\xi-\hat{\xi}_{p_1}\|_2\leq\|\xi-\hat{\xi}_{p_2}\|_2\},\\
\Xi(\hat{\xi}_{p_2}) &= \{\xi\in\Xi(\hat{\xi}_{p^\star}):\|\xi-\hat{\xi}_{p_2}\|_2\leq\|\xi-\hat{\xi}_{p_1}\|_2\}.
\end{aligned}
\]
\rmb{The boundary between the two children is the set of points equidistant from $\hat{\xi}_{p_1}$ and $\hat{\xi}_{p_2}$, i.e., the perpendicular bisector hyperplane. The first child definition is equivalent to
\(
2(\hat{\xi}_{p_2}-\hat{\xi}_{p_1})^\top \xi \leq \|\hat{\xi}_{p_2}\|_2^2-\|\hat{\xi}_{p_1}\|_2^2,
\)
and the second child is obtained by reversing this inequality. Thus,  each child leaf is created by adding one linear inequality to the description of the parent scenario partition. 
}

\section{Case Study: Sioux Falls Network}\label{S4:case_study}

We report computational results using a case study based on the Sioux Falls Network~\citep{leblanc1975} presented in Figure~\ref{F:sioux-falls-network}. 
The network consists of 24 nodes and 76 directed arcs. Eight nodes are candidate shelter locations: $\{1, 4, 8, 13, 14, 18, 20, 22\}$~\citep{kulshrestha2011}. Evacuees are initially located at the remaining 16 nodes. For simplicity, we assume that the baseline traffic in the network is zero. We consider routes with travel distance up to 20\% longer than the shortest route between an origin and a shelter location. For each arc, the distance and capacity parameters are generated from uniform distributions $U(1,5)$ and $U(4000, 8000)$, respectively. The maximum allowable distance for relief supply distribution from a prepositioned trailer to a shelter is set to 5 km (the distances between candidate shelter locations in the network range from 1 to 18 km), and the number of relief supply units per pallet $m=5$.  Assuming a free-flow speed of 60 km/h across the network, the free-flow travel time $t^0_a$ for each road segment is calculated by dividing its distance by 60. 
\begin{figure}[b!]
\centering
\includegraphics[scale=0.75]{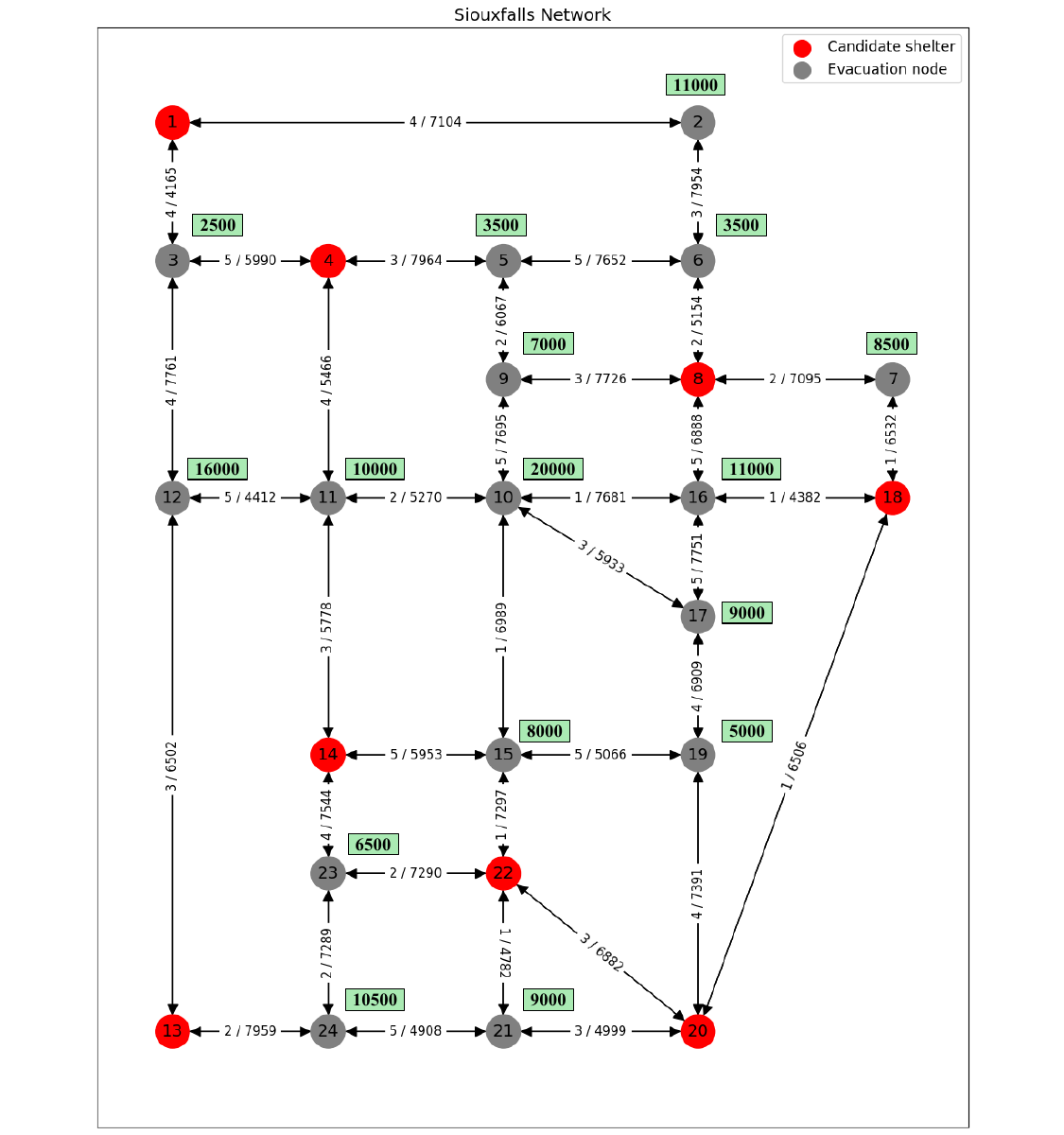}
\caption{The Sioux Falls network consists of 24 nodes and 76 directed arcs. Nominal shelter demand $\bar{w}_i$ (green squares) originates from 16 evacuation nodes (gray). Eight nodes (red) are candidate shelter locations. The label on each arc shows the distance/capacity between nodes.}
\label{F:sioux-falls-network}
\end{figure}
We present extensive computational experiments to analyze: 1) the trade-off between unmet relief item demand and evacuation time, 2) the effectiveness of centralized evacuation route assignment and adaptive allocation of relief supplies, and 3) the computational performance of the proposed tree-based partition-and-bound approach. We implement the solution algorithms in Python 3.12.7. In each iteration, we solve UB-ARMIO using the RSOME (Robust Stochastic Optimization Made Easy) package~\citep{chen2021rsome} with Gurobi 12, and LB-ARMIO using Gurobi 12. The algorithm terminates when the optimality gap, i.e., $\frac{\text{UB}-\text{LB}}{\text{UB}} \times 100$, where UB and LB denote the upper and lower bounds, is less than 1\% or 
\rmb{when the 3-hour time limit is reached. Because the time limit is checked between iterations, the reported solution times may exceed three hours when the final iteration starts before the limit.} 
Experiments are performed on a computer with an Intel i5-14500 CPU and 16 GB of memory. 


\subsection{Uncertainty Set}\label{uncertaintysetdefn}

We use a polyhedral uncertainty set in the case study. 
\rmb{This choice is compatible with the robust-counterpart reformulations of constraints~\eqref{eq:unmet_model_UB}--\eqref{eq:arc_capa_model1_pwl_UB}.} 
It also allows us to compare our solution approach with that of~\cite{romeijnders2021piecewise}, which requires a polyhedral uncertainty set. We assume that the realization of 
\rmb{evacuee count} $\xi_i$ at origin $i \in \mathcal{O}$ takes values in the interval $[\bar{w}_{i}(1-\epsilon), \bar{w}_{i}(1+\epsilon)]$, where  $\bar{w}_i$ is the nominal value (provided in Figure \ref{F:sioux-falls-network}), and $\epsilon > 0$ controls the range.  
To avoid overly conservative solutions, the total deviation from the nominal 
values is limited by a budget of uncertainty $\Gamma \geq 0$.
\begin{equation}
     \Xi = \big\{ \xi \; | \; \bar{w}_{i}(1-\epsilon) \leq \xi_i \leq \bar{w}_{i}(1+\epsilon),\ i \in \mathcal{O}, \; \sum_{i \in \mathcal{O}}\left|\frac{\xi_i - \bar{w}_{i}}{\bar{w}_{i}\epsilon}\right|  \leq \Gamma \big\}.
\end{equation}
The uncertainty set $\Xi$ expands as $\epsilon$ or $\Gamma$ increases, and the robust solutions become more conservative. We set $\epsilon = 0.3$ in all experiments, and the total relief supplies $S \times K \times m$ to cover $\coverage \in \{95\%, 99\%\}$ of the maximum total 
\rmb{evacuee count} realization over $\Xi$.

\subsection{Managerial Insights }\label{sec:managerial_insights}
\rmb{We use the case study to investigate two fundamental questions in evacuation planning. First, how does the trade-off between evacuation time and relief demand coverage manifest itself as planners vary the relative emphasis placed on these two objectives? Second, what are the individual and combined benefits of centralized route assignment and post-disaster relief redistribution in terms of improving relief demand coverage and reducing evacuation time? We address both questions under an uncertainty budget of $\Gamma = 0.5$ for all experiments in this section.}

\subsubsection{Relief Demand Coverage and Evacuation Time Trade-off.}\label{sec:trade_off}
\rmb{We vary the objective-weight parameter $\lambda \in \{0, 0.01, 0.1, 0.3, 0.5, 1\}$ to show how the planner's priority on evacuation time affects the robust relief demand coverage.} Figure~\ref{F:pareto-plot} 
\rmb{shows the trade-off curves between} unmet demand and evacuation time in the CSO route assignment (\CSO) and user route choice (\Pref) models \rmb{when} $N=4$, $S=2$, and $\coverage=95\%$. Table~\ref{tab:lambda_results} \rmb{reports the corresponding unmet demand, evacuation-time, bound, and runtime values.} 

\begin{figure}[h!]
\centering
\includegraphics[scale=0.55]{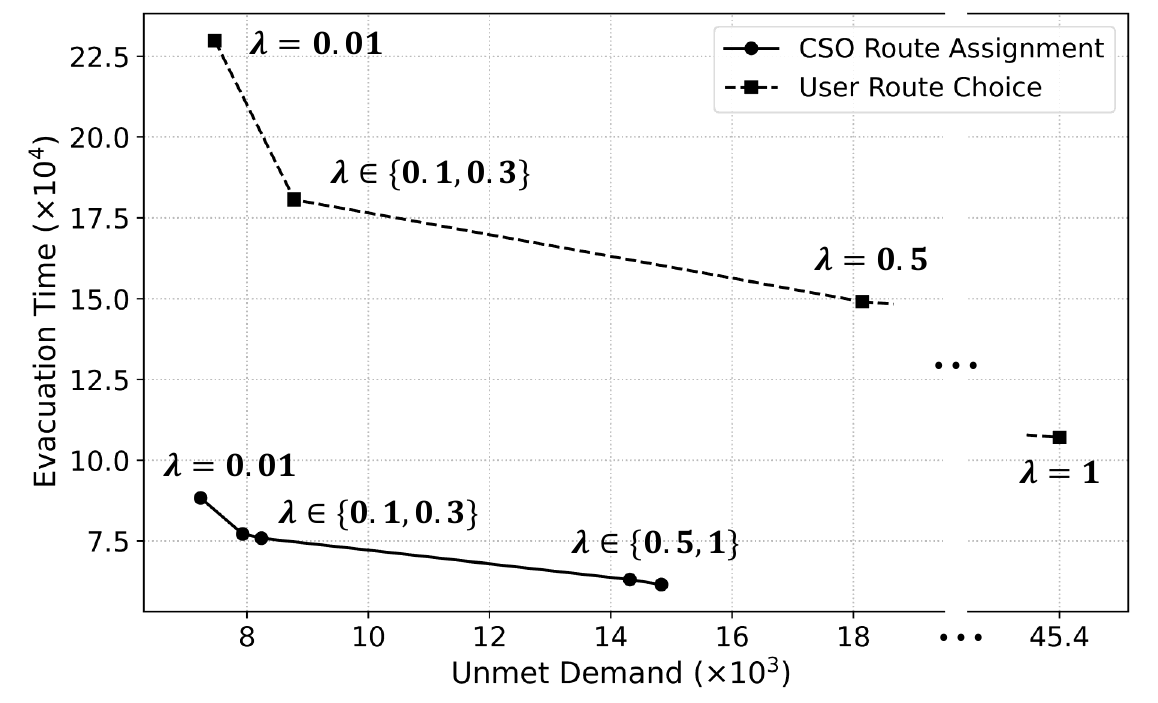}
\caption{Unmet relief item demand and evacuation time trade-off.}
\label{F:pareto-plot}
\end{figure}

\rmb{The main insight is that even a small positive weight on evacuation time (e.g., $\lambda=0.01$) can sharply reduce worst-case evacuation time, without a major impact on unmet demand; however, further increases in $\lambda$ eventually require a larger sacrifice in demand coverage.}
For all tested values of $\lambda$ in Figure~\ref{F:pareto-plot}, \CSO \  strictly dominates \Pref, achieving lower values for both unmet demand and evacuation time.
Furthermore, the preference-based model \Pref \  is more sensitive to $\lambda$, indicating that balancing evacuation time against relief-demand coverage is harder when evacuees are allowed to choose routes. Thus, the trade-off analysis is not only a sensitivity test on $\lambda$; it shows how centralized route assignment expands the set of attractive operating points available to the planner.


\begin{table}[htbp]
\centering
\caption{Computational results for $N=4$, $S=2$, and $\coverage=95\%$.}
\label{tab:lambda_results}
\renewcommand{\arraystretch}{0.995}
\resizebox{0.75\textwidth}{!}{
\begin{tabular}{cc| cccccc}
\specialrule{1.5pt}{0pt}{0pt}
 & &$\lambda = 0.0$ & $\lambda = 0.01$  & $\lambda = 0.1$ & $\lambda = 0.3$ & $\lambda = 0.5$ & $\lambda = 1$ \\
\specialrule{1pt}{0pt}{0pt}

\multirow{6}{*}{\CSO}
& Unmet demand & 7,200 & 7,234 & 7,928 & 8,268 & 14,316 & 14,834 \\
& Evac. time   & 1,628,598 & 88,350 & 77,184 & 75,868 & 63,120 & 61,597 \\
& UB          & 7,200 & 8,054 & 15,223 & 30,490  & 45,027 & 75,447  \\
& LB          & 7,200 & 7,888 & 14,022 & 27,868  & 40,459 & 67,811  \\
& Gap (\%)          & 0.0 & 2.1 & 7.9 & 8.6 & 10.1 & 10.1 \\
& Sol. time (sec) & 26 & 11,421 & 11,719 & 11,611  & 11,627 & 11,005  \\

\specialrule{1pt}{0pt}{0pt}

\multirow{6}{*}{\Pref}
& Unmet demand & 7,270 & 7,469 & 8,778  & 8,778 & 18,146 & 45,426 \\
& Evac. time   & 1,736,583 & 229,835 & 180,873  & 180,642 & 149,110 & 107,088 \\
& UB           & 7,270 & 9,682 & 26,497 & 62,564 & 92,268 & 152,344 \\
& LB          & 7,200 & 9,357 & 25,990 & 60,413  & 89,860 & 146,112 \\
& Gap (\%)         & 0.9 & 3.4 & 1.9 & 3.4 & 2.6 & 4.1 \\
& Sol. time (sec) & 1,481 & 12,608 & 11,748 & 11,068 & 11,056 & 11,708 \\

\specialrule{1.5pt}{0pt}{0pt}
\end{tabular}
}
\end{table}

\subsubsection{Effectiveness of Evacuation Route Assignment and Post-disaster Redistribution of Relief Supplies.}\label{sec:operational_effectiveness}
\phantom{} \rmb{The second set of experiments isolates the value of the two operational strategies in \CSO: centralized evacuation route assignment and post-disaster redistribution of prepositioned relief supplies.}  
Table~\ref{tab:cso-vs-others} compares 
\CSO\ \rmb{with three benchmarks:} 
(i) \Pref, in which evacuees choose routes and prepositioned relief supplies are redistributed in the second stage; (ii) \texttt{CSO-Fixed}, denoted by \CSOFix, in which evacuees use the assigned routes, but relief supplies are prepositioned at shelter locations $j \in \mathcal{J}$ and not redistributed post-disaster; and (iii) \texttt{URC-Fixed}, denoted by \PrefFix, in which evacuees choose routes, and prepositioned relief supplies at shelter locations are not redistributed post-disaster. The \Pref\ \rmb{model} is presented in Section~\ref{with_route_choice}. The \CSOFix\ and \PrefFix\ \rmb{models} are identical to \CSO\ and \Pref, respectively, except that constraints~\eqref{eq:unmet_model1} and \eqref{eq:unmet_model2} are replaced by
$\sum_{i \in \mathcal{O}} \sum_{r \in \mathcal{R}_{ij}} \xi_{i}y_{r} \leq u_j(\xi) + K m s^1_{j}$ for all $j \in \mb{\mathcal{J}}, \ \xi \in \Xi$
(with $\pi_r$ substituted for $y_r$ in \PrefFix), and constraints~\eqref{eq:flexible_trailer_capa_model1} are replaced by 
$s^1_{j} \leq S b_{j}$ for all $j \in \mb{\mathcal{J}}$.

We solve instances with the number of shelters $N \in \{2,4,6\}$, the number of trailers $S \in \{2,5,10\}$, and demand coverage level $\coverage \in \{95\%,99\%\}$. The parameter $\lambda$ is set to 0.1, corresponding to the elbow point of the \rmb{computed trade-off curves} for both models in Figure~\ref{F:pareto-plot}. Table~\ref{tab:cso-vs-others} reports the percentage improvement achieved by \CSO\ relative to each benchmark in terms of the worst-case objective upper bound (UB), unmet demand, and evacuation time. Several key \rmb{insights} 
emerge.

\begin{table}[h!]
\centering
\caption{Percent improvement in worst-case objective upper bound (UB), unmet demand, and evacuation time with centralized route assignment and post-disaster relief supply redistribution. ($\lambda=0.1$)}
\label{tab:cso-vs-others}
\renewcommand{\arraystretch}{0.98}
\setlength{\tabcolsep}{5pt}
\resizebox{1.0\textwidth}{!}{
\begin{tabular}{%
  c|
  c|
  c
  | c c c
  | c c c
  | c c c
}
\specialrule{1.5pt}{0pt}{0pt}
 \multicolumn{3}{c}{}&
 \multicolumn{3}{c}{${N=2}$} &
 \multicolumn{3}{c}{${N=4}$} &
 \multicolumn{3}{c}{${N=6}$} \\
\cmidrule(lr){4-6}\cmidrule(lr){7-9}\cmidrule(lr){10-12}
\multicolumn{1}{c}{$S$}  &
\multicolumn{1}{c}{$\coverage$}  &
\multicolumn{1}{c}{ }
& \CSO\ vs \CSOFix
& \CSO\ vs \Pref
& \CSO\ vs \PrefFix
& \CSO\ vs \CSOFix
& \CSO\ vs \Pref
& \CSO\ vs \PrefFix
& \CSO\ vs \CSOFix
& \CSO\ vs \Pref
& \CSO\ vs \PrefFix \\
\specialrule{1pt}{0pt}{0pt}

\multirow[c]{6}{*}{${2}$}
& \multirow[c]{3}{*}{$95\%$}
& UB
  & 7.2 & 39.1 & 43.1
  & 13.0 & 42.5 & 61.0
  & 18.9 & 54.5 & 73.2\\
&
& Unmet dem.
  & 6.9 & 0.5 & -1.4
  & 4.0 & 9.7 & 62.8
  & 4.5 & 58.5 & 80.8\\
&
& Evac. time
  & 10.2 & 51.5 & 55.2
  & 22.2 & 57.3 & 57.1
  & 31.5 & 47.7 & 52.6\\
\cmidrule(lr){2-12}

& \multirow[c]{3}{*}{${99\%}$}
& UB
  & 14.7 & 47.7 & 55.0
  & 24.8 & 53.6 & 71.0
  & 32.1 & 66.8 & 80.5\\
&
& Unmet dem.
  & 10.7 & -10.7 & 39.8
  & 25.0 & 1.1 & 91.2
  & 30.8 & 90.6 & 95.6\\
&
& Evac. time
  & 16.5 & 51.2 & 56.4
  & 27.4 & 57.8 & 46.1
  & 35.0 & 35.7 & 41.8\\
\specialrule{1pt}{0pt}{0pt}

\multirow[c]{6}{*}{${5}$}
& \multirow[c]{3}{*}{${95\%}$}
& UB
  & -0.5 & 38.6 & 49.8
  & 0.5 & 46.5 & 58.6
  & 1.0 & 65.7 & 68.4\\
&
& Unmet dem.
  & 1.4 & -1.2 & 5.4
  & -1.7 & 9.4 & 44.6
  & 2.0 & 60.5 & 68.9\\
&
& Evac. time
  & -0.9 & 50.4 & 61.3
  & 4.9 & 70.8 & 72.4
  & 1.5 & 74.2 & 66.9 \\
\cmidrule(lr){2-12}

& \multirow[c]{3}{*}{${99\%}$}
& UB
  & 2.6 & 47.2 & 62.9
  & 5.5 & 73.5 & 77.8
  & 4.7 & 82.7 & 83.7\\
&
& Unmet dem.
  & -3.7 & -11.2 & 69.0
  & 1.3 & 78.9 & 87.7
  & 11.6 & 89.8 & 92.0\\
&
& Evac. time
  & 2.8 & 50.9 & 61.8
  & 6.5 & 69.9 & 64.8
  & 2.4 & 72.6 & 63.7\\
\specialrule{1pt}{0pt}{0pt}

\multirow[c]{6}{*}{${10}$}
& \multirow[c]{3}{*}{${95\%}$}
& UB
  & 0.6 & 39.4 & 39.6
  & 0.9 & 43.3 & 43.4
  & -0.5 & 47.4 & 52.7\\
&
& Unmet dem.
  & 0.7 & -0.7 & -0.7
  & -1.1 & 17.9 & 18.0
  & -1.1 & 22.3 & 36.8\\
&
& Evac. time
  & 1.0 & 51.7 & 51.7
  & 2.2 & 63.9 & 64.1
  & 2.2 & 73.0 & 72.5\\
\cmidrule(lr){2-12}

& \multirow[c]{3}{*}{${99\%}$}
& UB
  & 4.8 & 48.0 & 48.0
  & 6.2 & 66.5 & 69.0
  & 3.4 & 71.8 & 74.1\\
&
& Unmet dem.
  & 13.2 & -2.3 & -2.5
  & 5.2 & 70.8 & 75.8
  & 0.6 & 43.8 & 76.1\\
&
& Evac. time
  & 3.0 & 51.0 & 51.0
  & 9.0 & 64.7 & 64.4
  & 6.8 & 79.3 & 71.8\\
\specialrule{1.5pt}{0pt}{0pt}

\end{tabular}
}
\end{table}

First, centralized route assignment 
\rmb{is the dominant and most stable source of improvement in} test instances. Relative to \Pref, \CSO\ reduces the worst-case objective value by 39--83\%. These \rmb{gains are driven primarily by} 
large reductions in evacuation time, often over 50\%, demonstrating the effectiveness of coordinated routing in mitigating congestion during evacuations.
Second, the value of post-disaster redistribution of relief supplies depends on the flexibility of the prepositioning strategy. Comparing \CSO\ with \CSOFix\, shows that
when only two trailers are available ($S=2$), relief supply redistribution reduces the objective value by up to 32\% and unmet demand by up to 31\% for $N=6$. 
\rmb{As the number of trailers increases, these marginal gains become smaller;} for $S=10$, the \rmb{objective-value} improvements are generally below 7\%.  
\rmb{This pattern suggests that redistribution is most important when the initial prepositioning plan cannot by itself cover the spatial imbalance in relief demand.} 

\rmb{Third,} the largest 
\rmb{gains arise} when centralized routing and redistribution of relief supplies are implemented simultaneously. Relative to \PrefFix, \CSO\ reduces the objective value by up to 84\%, unmet demand by up to 96\%, and evacuation time by up to 73\%. 
\rmb{The gains also} generally increase with the number of shelters. For example, when $S=2$ and $\coverage=99\%$, the reduction in the worst-case objective relative to \PrefFix\ increases from 55\% for $N=2$ to 81\% for $N=6$. A larger shelter network provides greater flexibility in both routing and resource allocation, enabling \CSO\ to exploit coordination opportunities more effectively.
~\ref{app:detailed_computational_results} \rmb{further} decomposes the observed improvements into those attributable to centralized route assignment and \rmb{those attributable to} redistribution of relief supplies. Specifically, Table~\ref{tab:val_of_assign_adaptive_allocation} isolates the benefits of centralized route assignment by comparing \CSOFix\ vs \PrefFix, and the benefits of relief supply redistribution by comparing \Pref\ vs \PrefFix. 
The decomposition reinforces the main finding: centralized route assignment accounts for most of the improvements in objective value and evacuation time across the majority of settings, whereas the marginal benefits of post-disaster relief redistribution are concentrated in scenarios with tighter relief supply or less spatially flexible prepositioning. 

Overall, the results suggest that centralized route assignment and adaptive post-disaster relief redistribution play distinct yet complementary roles. Centralized route assignment primarily serves as a congestion management mechanism by redirecting evacuee flows and balancing shelter utilization. In contrast, post-disaster relief redistribution primarily improves relief demand coverage, with its greatest marginal benefit arising when relief resources are scarce or evacuees are unevenly distributed across shelters. The value of the joint \CSO\ model lies in its ability to coordinate these two mechanisms, rather than optimizing each in isolation. 

\subsection{Performance of the Solution Approach}
\label{subsec:AlgorithmicPerformance}

\phantom{} \rmb{We next evaluate the computational performance of our proposed active-partition refinement strategy. Among uncertainty set partitioning-based alternatives, our preliminary experiments with the method of \cite{dunning2016paper} produced a large number of child partitions and was computationally inefficient. We therefore use the partitioning method of \cite{romeijnders2021piecewise}, denoted by \texttt{RP}, as the main computational benchmark. In what follows, our proposed partition-and-bound method is denoted by \texttt{PB}.}

Both \texttt{PB} and \texttt{RP} split the active partition that determines the worst-case objective function value of the finitely adaptable master problem into two child nodes in each iteration.  In \texttt{PB}, we determine critical scenarios based on constraint slacks. In \texttt{RP}, critical scenarios are obtained from the optimal dual solutions of the node subproblems encountered in the branch-and-bound algorithm when optimizing the mixed-integer linear programming reformulation of the finitely adaptable master problem. 
This approach is applicable to bounded polyhedral uncertainty sets, whereas our method can accommodate mixed-integer uncertainty sets. 
We implement \texttt{RP} to solve \CSO \ (see ~\ref{app:dual_based_scenario_identification}).

Table~\ref{tab:romeijnders_approach_comparision} and Figure~\ref{F:Proposed_vs_RP} present the computational results for $N = 4$, $S = 2$, $\coverage = 95\%$, with uncertainty budget $\Gamma \in \{0.1, 0.2, 0.5, 1.0\}$ and $\lambda \in \{0.1, 0.3\}$. The table \rmb{reports solution quality, bounds, final gaps, partition counts, and solution times,} 
while the figure shows the 
\rmb{evolution of} the optimality gap over time. Both algorithms \rmb{use the same stopping rule, terminating when the optimality gap is less than 1\% or when the three-hour time limit is reached between iterations.}   
In terms of solution quality, \rmb{the proposed} \texttt{PB} \rmb{method} consistently achieves a lower UB across all cases, indicating that it reliably produces better feasible solutions.
Furthermore, Figure~\ref{F:Proposed_vs_RP} shows our method maintains a lower optimality gap than \texttt{RP} throughout the solution process in all instances.

\begin{table}[h!]
\centering
\caption{Performance of the proposed \texttt{PB} method and \texttt{RP} for the  \CSO\ model.}
\label{tab:romeijnders_approach_comparision}
\renewcommand{\arraystretch}{0.995}
\resizebox{0.8\textwidth}{!}{ 
\begin{tabular}{cc| cc| cc| cc| cc} 
\specialrule{2pt}{0pt}{0pt}
\multirow{2}{*}{} & \multirow{2}{*}{} & \multicolumn{2}{c|}{\textbf{$\Gamma = 0.1$}} & \multicolumn{2}{c|}{\textbf{$\Gamma = 0.2$}} & \multicolumn{2}{c|}{\textbf{$\Gamma = 0.5$}} & \multicolumn{2}{c}{\textbf{$\Gamma = 1.0$}}\\
\cmidrule(lr){3-4} \cmidrule(lr){5-6} \cmidrule(lr){7-8} \cmidrule(lr){9-10} 
& & \texttt{RP} & \texttt{PB} & \texttt{RP} & \texttt{PB} & \texttt{RP} & \texttt{PB} & \texttt{RP} & \texttt{PB}\\
\specialrule{2pt}{0pt}{0pt}

\multirow{7}{*}{$\boldsymbol{\lambda = 0.1}$}
    & \textbf{Unmet demand}         & 4,954  & 5,008 & 5,585 & 5,563 & 7,669 & 7,928 & 10,933 & 11,106  \\
     & \textbf{Evac. Time}           & 71,783 & 70,281 & 74,427 & 72,960 & 79,272 & 77,184 & 85,707 & 85,976\\
     & \textbf{UB}                   & 12,132 & 11,959 & 13,028 & 12,778 & 15,596 & 15,223 & 19,503 & 19,376 \\
     & \textbf{LB}                   & 11,657 & 11,655 & 12,274 & 12,194 & 15,036 & 14,022 & 17,120 & 17,132\\
     & \textbf{Gap (\%)}             & 3.9 & 2.5 &  5.8  & 4.6 & 9.7 & 7.9 & 12.2 & 11.6\\
     & \textbf{Partitions}      & 18 & 19 & 22 & 19 & 19 & 18 & 19 & 20 \\
     & \textbf{Sol. time (sec)}     & 13,749 & 11,557 & 15,895 & 11,193 & 15,036 & 11,719 & 13,199 & 12,216  \\
\specialrule{2pt}{0pt}{0pt}

\multirow{7}{*}{$\boldsymbol{\lambda = 0.3}$}
    & \textbf{Unmet demand}         & 5,150 & 5,098 & 6,103 & 5,916 & 8,286 & 8,268 & 11,916 & 11,477  \\
     & \textbf{Evac. Time}           & 69,947 & 69,567 & 70,518 & 70,654 & 78,102 & 75,868 & 86,362 & 82,589\\
     & \textbf{UB}                   & 26,134 & 25,879 & 27,258 & 26,885 & 31,717 & 30,490 & 37,825 & 35,944 \\
     & \textbf{LB}                   & 25,322 & 25,302 &  25,926   & 25,834 &  27,971   & 27,868 &  31,063  & 30,611\\
     & \textbf{Gap (\%)}                  & 3.1  & 2.2 &   4.9   & 3.9 &  11.8  & 8.6 &    17.9   & 14.8\\
     & \textbf{Partitions}       & 19 & 19 & 20 & 20 & 18  & 18 & 20 & 20  \\
     & \textbf{Sol. time (sec)}             & 11,525 & 11,186 & 13,165 & 11,214 & 12,308 & 11,611 & 15,079 & 11,431  \\
\specialrule{2pt}{0pt}{0pt}

\end{tabular}
}
\end{table}

\begin{figure}[h!]
\centering
\includegraphics[scale=0.33]{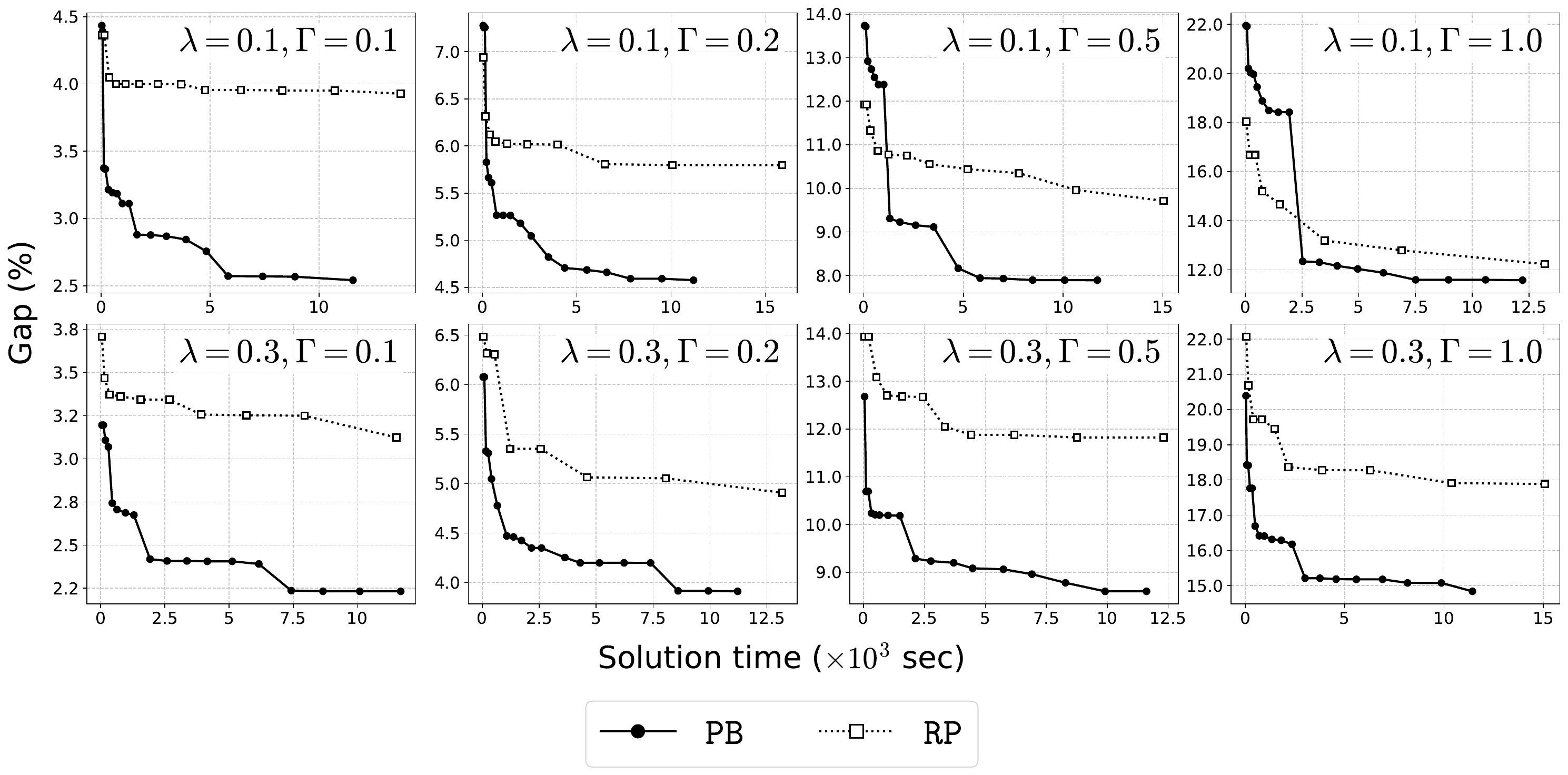}
\caption{Optimality gap and solution time of the proposed \texttt{PB} method and \texttt{RP} for the \CSO\ model.}
\label{F:Proposed_vs_RP}
\end{figure}

\section{Conclusion}\label{S4:conclusions}
\rmb{We formulated a two-stage adaptive robust mixed-integer optimization framework for disaster evacuation preparedness under uncertain evacuee counts. The model jointly optimizes shelter locations, evacuation route assignment, relief supply trailer prepositioning, and adaptive post-disaster distribution of prepositioned supplies while balancing worst-case unmet relief item demand and congestion-dependent evacuation time. There is a trade-off between relief demand coverage and evacuation time because shelter and trailer locations that improve relief coverage may concentrate traffic on a part of the road network. To the best of our knowledge, this is the first robust optimization model that integrates these preparedness and post-disaster distribution decisions in a single adaptive framework. We also developed user-route-choice and fixed allocation benchmark models to quantify the value of centralized route assignment and adaptive relief supply redistribution.}

\rmb{We provided complexity insights, showing that although the fixed-scenario recourse problem is polynomially solvable, optimizing the first-stage preparedness decisions is NP-hard even without uncertainty, robust verification of a fixed plan is coNP-hard, and the decision version of \PWLCSO{} with a bounded polyhedral uncertainty set is $\Sigma^p_2$-complete. Motivated by these results, we developed a partition-and-bound algorithm that combines finite-scenario lower bounds with finitely adaptable upper bounds and refines the uncertainty set through a tree-based active-partition rule. This design controls the representation of uncertainty and the granularity of adaptability while maintaining valid bounds on the fully adaptive robust optimum.}

\rmb{The Sioux Falls case study illustrates the operational and computational value of the framework. Centralized route assignment improves the trade-off between worst-case unmet demand and evacuation time relative to user route choice, primarily by reducing evacuation time. Adaptive post-disaster redistribution provides additional value for covering relief demand, especially when trailer resources are limited and prepositioned supplies are spatially less flexible. The comparison with the benchmarks shows that centralized routing and adaptive redistribution play complementary roles, with the largest benefits arising when both are optimized jointly. The computational comparison further indicates that the proposed active-partition refinement strategy produces high-quality feasible solutions and competitive bounds in the tested instances.}

Future work can incorporate 
other uncertainties, such as road capacities, travel times, shelter availability, or supply disruptions; extend 
to multistage settings with evolving demand and network conditions; and develop richer behavioral models of evacuee compliance and route choice.

\bibliographystyle{apalike}
\bibliography{evacuation_bib,consolidated_bib}



%
\begin{APPENDICES}
\renewcommand{\thesection}{Appendix \Alph{section}}

\section{Projection Proof}
\label{app:projection}

\projection*
\begin{proof}
(\rmb{$\subseteq$}): Let $(\hat\pi,\hat\tau)\in \operatorname{Proj}_{(\pi,\tau)}(Q)$. Then there exists
$\hat q$ such that $(\hat\pi,\hat\tau,\hat q)\in Q$. \rmb{For each $i\in\mathcal{O}$, the first two equalities in $Q$ imply} $\hat q_i\sum_{\rmb{j \in \mathcal{J}}}\sum_{r\in \mathcal{R}_{ij}} v_{r}\hat\tau_{r} = 1$. \rmb{Hence}, $\sum_{\rmb{j \in \mathcal{J}}}\sum_{r\in \mathcal{R}_{ij}} v_{r}\hat\tau_{r} > 0$. \rmb{Since $v_r>0$, at least one route has $\hat\tau_r=1$, so~\eqref{eq:route_assign_model2} holds. Moreover,} 
we have
$\hat{q}_i = \frac{1}{\sum_{\rmb{j \in \mathcal{J}}}\sum_{r' \in \mathcal{R}_{ij}} v_{r'}\hat{\tau}_{r'}}$. Substituting this into the first equality in $Q$ \rmb{gives~\eqref{choice-model-const_}. Therefore $(\hat\pi,\hat\tau)\in\mathcal{V}$.}

(\rmb{$\supseteq$}): Let $(\tilde{\pi},\tilde{\tau})\in \mathcal{V}$. For each $i\in\mathcal{O}$, define
$
\tilde{q}_i:=
\frac{1}{\sum_{\rmb{j \in \mathcal{J}}}\sum_{r\in\mathcal{R}_{ij}} v_{r}\tilde{\tau}_{r}}.
$,  \rmb{which is well-defined since its denominator is positive by~\eqref{eq:route_assign_model2} and $v_r>0$.} 
\rmb{Constraint~\eqref{choice-model-const_} then gives $\tilde\pi_r=v_r\tilde\tau_r\tilde q_i$ for every $r\in\mathcal{R}_{ij}$. Summing 
them
over all routes from origin $i$ gives $\sum_{j\in\mathcal{J}}\sum_{r\in\mathcal{R}_{ij}}\tilde\pi_r=1$. Thus $(\tilde\pi,\tilde\tau,\tilde q)\in Q$, and $(\tilde\pi,\tilde\tau)\in\operatorname{Proj}_{\pi,\tau}(Q)$.}
\end{proof}

\section{Complexity Proofs}
\label{app:complexity_proofs}

This appendix restates and proves the results in Section~\ref{sec:complexity}. 

\proprhstocoeffARO*

\begin{proof}
Under the finite-rational encoding convention given in Section~\ref{sec:complexity}, the feasibility question for a two-stage ARO problem is given by: decide on the first-stage variables; then, for every uncertainty realization, decide on the recourse variables, and check whether the linear constraints are satisfied. This has the quantifier pattern 
$
  \exists\ \texttt{first-stage decision},
$
$
  \forall\ \texttt{uncertainty realization},
$
$
  \exists\ \texttt{recourse decision}.
$
For any choice of these three objects, it remains only to verify the first-stage constraints and either determine that the uncertainty realization lies outside the rational polyhedral set \(\Xi\), or, if it lies in \(\Xi\), verify the feasibility of the finite system of second-stage linear inequalities. Each of these checks can be performed in polynomial time. Hence, the feasibility problem belongs to $\Sigma^p_3$.

For hardness, it suffices to identify a $\Sigma^p_3$-hard subclass of the stated technology-uncertain ARO family. \rmb{We use the right-hand-side uncertain ARO result of}~\cite{goerigk2024complexity}. \rmb{Their Theorem 5} establishes $\Sigma^p_3$-\rmb{hardness} for a two-stage ARO feasibility class with mixed-binary recourse,  deterministic technology and recourse matrices, and continuous budgeted uncertainty in the right-hand side. \rmb{Since mixed-binary recourse is a special case of mixed-integer recourse, this source class is contained in the recourse class considered here.} A representative source instance \rmb{after collecting deterministic domain and side constraints} has constraints of the form
\[
\exists x\;\forall \xi\in\Xi\;\exists z(\xi):\quad Ax+Bz(\xi)\le d+D\xi,
\]
where $\Xi$ is the continuous budgeted uncertainty set, $A$ is the deterministic technology matrix, $B$ is the deterministic recourse matrix, and the uncertainty appears on the right-hand side.

We construct a transformed instance with an uncertain technology matrix as follows. Add one first-stage variable $q$, impose $q=1$ by deterministic rational inequalities $q\le1$ and $-q\le-1$, and restate the constraint with uncertain right-hand side by
\[
\exists x\;\forall \xi\in\Xi\;\exists z(\xi):\quad Ax+Bz(\xi)-(D\xi)q\le d.
\]
\rmb{Equivalently, with augmented first-stage vector \((x,q)\), the technology matrix \([A; -D\xi]\) is affine in the uncertainty vector \(\xi\).}
The term $-(D\xi)$ is an uncertain coefficient that multiplies a first-stage variable $q$, so the uncertainty enters affinely in the technology matrix, whereas the right-hand side is deterministic. Side constraints and variable-domain restrictions from the source instance are kept unchanged, \rmb{except that the constraints fixing $q$ are added as the deterministic linear side constraints}.

The transformation preserves feasibility. If the source instance is feasible, the same $x$ and recourse decisions $z(\xi)$, together with $q=1$, satisfy the transformed constraints for every $\xi\in\Xi$. Conversely, any feasible solution of the transformed instance has $q=1$, and substituting this value recovers the source constraints with right-hand-side uncertainty. Thus, the source and transformed instances are equivalent in terms of feasibility.

The transformation also preserves the structural features needed for Theorem~5 in \citet{goerigk2024complexity}. The recourse matrix $B$ is unchanged and remains deterministic; the uncertainty set $\Xi$ is unchanged and remains continuous budgeted; the decision timing remains $\exists\,\forall\,\exists$; and the construction has polynomial size, adding one scalar variable, two deterministic rational inequalities, and \rmb{an affine technology-matrix representation using} the same uncertain linear terms moved from the right-hand side to the technology matrix. Hence, 
\rmb{the stated technology matrix-uncertain ARO class} contains a polynomially embedded $\Sigma^p_3$-hard subclass. Along with membership in $\Sigma^p_3$, this proves $\Sigma^p_3$-completeness. 
\end{proof}

\subsection{Hardness of the Recourse Problem}\label{app:tractable_recourse}

Given first-stage decisions  $(b,s^1,y)$ and realized evacuee counts $\xi$, let
\[
D_j:=\sum_{i\in\mathcal{O}}\sum_{r\in\mathcal{R}_{ij}}\xi_i y_r,\quad C_\ell:=Ks^1_\ell.
\]
\rmb{Because evacuee counts and route assignments are nonnegative, each $D_j$ is nonnegative. Because $K$ and $s^1_\ell$ are integer-valued, each $C_\ell$ is an integer capacity. We prove a slightly more general statement with nonnegative shelter weights $w_j$; the \PWLCSO{} objective corresponds to $w_j=1$ for all $j \in\mathcal{J}$.} For nonnegative shelter weights $w_j$, the weighted-unmet-demand shipment problem is
\begin{equation}\label{eq:appendix_fixed_scenario}
\begin{aligned}
\min_{s^2,u}\quad & \sum_{j\in\mathcal{J}}w_j u_j \\
\text{s.t.}\quad
& D_j \le u_j+m\sum_{\ell\in\mathcal{L}_j}s^2_{\ell j} && j\in\mathcal{J},\\
& \sum_{j\in\mathcal{J}\,:\,\ell\in\mathcal{L}_j}s^2_{\ell j}\le C_\ell && \ell\in\mathcal{L},\\
& s^2_{\ell j}\in\mathbb{Z}_+, && j\in\mathcal{J},\ \ell\in\mathcal{L}_j,\\
& u_j\ge 0 && j\in\mathcal{J}.
\end{aligned}
\end{equation}

\propcomplexityrecourse*

\begin{proof}
The proof interprets relief supply pallets as flows and each pallet shipment as a marginal reduction in unmet demand. For shelter $j$, let $q_j=\lfloor D_j/m\rfloor$ and $r_j=D_j-q_jm$. Since $D_j\ge0$ and $m>0$, this gives $q_j\in\mathbb Z_+$ and $0\le r_j<m$. The first $q_j$ pallets delivered to shelter $j$ each reduce weighted unmet demand by $w_jm$; if $r_j>0$, one additional pallet has residual value $w_jr_j$; further pallets have no value. \rmb{Equivalently, if \(p_j=\sum_{\ell\in\mathcal L_j}s^2_{\ell j}\) pallets are shipped to shelter \(j\), then the optimal unmet demand at that shelter is \(u_j=\max\{0,D_j-mp_j\}\), and its contribution to the objective is}
\[
\rmb{
w_j\max\{0,D_j-mp_j\}
=
w_jD_j-w_j\min\{D_j,mp_j\}.}
\]
\rmb{Thus minimizing weighted unmet demand is equivalent to maximizing the total weighted covered demand \(\sum_{j\in\mathcal J}w_j\min\{D_j,mp_j\}\).}

We build a directed network with a source $o$, trailer-location nodes $L_\ell$, shelter nodes $J_j$, and sink $t$. Add the arc $o\to L_\ell$ with capacity $C_\ell$ and cost zero; for each pair $(\ell,j)$ such that $\ell\in\mathcal{L}_j$, add $L_\ell\to J_j$ with capacity $C_\ell$ and cost zero; add an unused-pallet arc $L_\ell\to t$ with capacity $C_\ell$ and cost zero. For each shelter $j$, add a full-pallet arc $J_j\to t$ with capacity $q_j$ and cost $-w_jm$; if $r_j>0$, add a residual arc $J_j\to t$ with capacity one and cost $-w_jr_j$. Total supply at \(o\) is \(\sum_{\ell\in\mathcal L} C_\ell\) units  and demand at \(t\) is \(-\sum_{\ell\in\mathcal L} C_\ell\). Figure~\ref{fig:recourse_network_examples} illustrates this construction for an example with two shelters and two trailer locations, showing how the optimal flow changes as trailer capacity is reduced.

\tikzset{
  source/.style={
    draw,
    circle,
    minimum size=6.5mm,
    inner sep=0pt
  },
  nodebox/.style={
    draw,
    rounded corners,
    minimum width=9mm,
    minimum height=6.5mm,
    align=center
  },
  edgelabel/.style={
    font=\scriptsize,
    fill=white,
    inner sep=1pt,
    align=center
  },
  optone/.style={
    draw=blue,
    very thick,
    postaction={decorate},
    decoration={
      markings,
      mark=at position 0.52 with {
        \arrow{Triangle[length=3.2mm,width=3.0mm]}
      }
    }
  },
  opttwo/.style={
    draw=blue,
    very thick,
    postaction={decorate},
    decoration={
      markings,
      mark=at position 0.46 with {
        \arrow{Triangle[length=3.2mm,width=3.0mm]}
      },
      mark=at position 0.58 with {
        \arrow{Triangle[length=3.2mm,width=3.0mm]}
      }
    }
  },
  optthree/.style={
    draw=blue,
    very thick,
    postaction={decorate},
    decoration={
      markings,
      mark=at position 0.42 with {
        \arrow{Triangle[length=3.2mm,width=3.0mm]}
      },
      mark=at position 0.52 with {
        \arrow{Triangle[length=3.2mm,width=3.0mm]}
      },
      mark=at position 0.62 with {
        \arrow{Triangle[length=3.2mm,width=3.0mm]}
      }
    }
  },
  avail/.style={
    ->,
    thin,
    gray!75
  },
  unused/.style={
    ->,
    thin,
    dashed,
    gray!70
  }
}

\begin{figure}[t]
\centering

\begin{subfigure}[t]{0.48\textwidth}
\centering
\resizebox{\textwidth}{!}{%
\begin{tikzpicture}[x=1cm,y=1cm,>=Latex,every node/.style={font=\small}]

\node[source]  (o)  at (0,0) {$o$};
\node[nodebox] (LA) at (2.4, 1.7) {$L_A$};
\node[nodebox] (LB) at (2.4,-1.7) {$L_B$};
\node[nodebox] (J1) at (6.1, 1.7) {$J_1$};
\node[nodebox] (J2) at (6.1,-1.7) {$J_2$};
\node[source]  (t)  at (9.9,0) {$t$};

\node[font=\scriptsize] at (2.4, 2.3) {$C_A=3$};
\node[font=\scriptsize] at (2.4,-2.3) {$C_B=2$};
\node[font=\scriptsize] at (6.15, 2.55) {$D_1=25$};
\node[font=\scriptsize] at (6.15, 2.25) {$w_1=1$};
\node[font=\scriptsize] at (6.15,-2.55) {$w_2=2$};
\node[font=\scriptsize] at (6.15,-2.25) {$D_2=18$};

\draw[optthree] (o) -- (LA);
\node[edgelabel] at (0.9, 1.3) {cap \(3\)\\cost \(0\)};

\draw[opttwo] (o) -- (LB);
\node[edgelabel] at (0.9,-1.3) {cap \(2\)\\cost \(0\)};

\draw[optthree] (LA) -- (J1);
\node[edgelabel] at (4.1, 2.25) {cap \(3\)\\cost \(0\)};

\draw[avail] (LA) to[bend right=30] (J2);
\node[edgelabel] at (2.2,0.7) {cap \(3\)\\cost \(0\)};

\draw[avail] (LB) to[bend left=30] (J1);
\node[edgelabel] at (2.2, -0.5) {cap \(2\)\\cost \(0\)};

\draw[opttwo] (LB) -- (J2);
\node[edgelabel] at (4.25,-2.15) {cap \(2\)\\cost \(0\)};

\draw[unused] (LA) to[bend left=55] (t);
\node[edgelabel] at (5.15, 3.2) {cap \(3\), cost \(0\)};

\draw[unused] (LB) to[bend right=58] (t);
\node[edgelabel] at (5.15,-3.3) {cap \(2\), cost \(0\)};

\draw[opttwo] (J1) to[bend left=28] (t);
\node[edgelabel] at (7.25, 2.1) {cap \(2\)\\cost \(-10\)};

\draw[optone] (J1) to[bend right=12] (t);
\node[edgelabel] at (6.7, 0.6) {cap \(1\)\\cost \(-5\)};

\draw[optone] (J2) to[bend left=12] (t);
\node[edgelabel] at (6.58,-0.6) {cap \(1\)\\cost \(-20\)};

\draw[optone] (J2) to[bend right=28] (t);
\node[edgelabel] at (7.4,-2.1) {cap \(1\)\\cost \(-16\)};

\end{tikzpicture}%
}

\vspace{0.35em}
\scriptsize
\begin{tabular}{@{}lll@{}}
\toprule
Optimal shipments & Unmet demand & Objective \\
\midrule
\(p_1=3,\ p_2=2\) &
\(u_1^\ast=0,\ u_2^\ast=0\) &
\(61+\kappa^\ast=0\) \\
\bottomrule
\end{tabular}

\caption{\small Capacity \(C_A+C_B=5\), with \(\kappa^\ast=-61\).}
\label{fig:recourse_network_full}
\end{subfigure}
\hfill
\begin{subfigure}[t]{0.48\textwidth}
\centering
\resizebox{\textwidth}{!}{%
\begin{tikzpicture}[x=1cm,y=1cm,>=Latex,every node/.style={font=\small}]

\node[source]  (o)  at (0,0) {$o$};
\node[nodebox] (LA) at (2.4, 1.7) {$L_A$};
\node[nodebox] (LB) at (2.4,-1.7) {$L_B$};
\node[nodebox] (J1) at (6.1, 1.7) {$J_1$};
\node[nodebox] (J2) at (6.1,-1.7) {$J_2$};
\node[source]  (t)  at (9.9,0) {$t$};

\node[font=\scriptsize] at (2.4, 2.3) {$C_A=2$};
\node[font=\scriptsize] at (2.4,-2.3) {$C_B=1$};
\node[font=\scriptsize] at (6.15, 2.55) {$D_1=25$};
\node[font=\scriptsize] at (6.15, 2.25) {$w_1=1$};
\node[font=\scriptsize] at (6.15,-2.55) {$w_2=2$};
\node[font=\scriptsize] at (6.15,-2.25) {$D_2=18$};

\draw[opttwo] (o) -- (LA);
\node[edgelabel] at (0.9, 1.3) {cap \(2\)\\cost \(0\)};

\draw[optone] (o) -- (LB);
\node[edgelabel] at (0.9, -1.3) {cap \(1\)\\cost \(0\)};

\draw[optone] (LA) -- (J1);
\node[edgelabel] at (4.1, 2.25) {cap \(2\)\\cost \(0\)};

\draw[optone] (LA) to[bend right=30] (J2);
\node[edgelabel] at (2.2,0.7) {cap \(2\)\\cost \(0\)};

\draw[avail] (LB) to[bend left=30] (J1);
\node[edgelabel] at (2.2, -0.5) {cap \(1\)\\cost \(0\)};

\draw[optone] (LB) -- (J2);
\node[edgelabel] at (4.25,-2.15) {cap \(1\)\\cost \(0\)};

\draw[unused] (LA) to[bend left=55] (t);
\node[edgelabel] at (5.15, 3.2) {cap \(2\), cost \(0\)};

\draw[unused] (LB) to[bend right=58] (t);
\node[edgelabel] at (5.15,-3.3) {cap \(1\), cost \(0\)};

\draw[optone] (J1) to[bend left=28] (t);
\node[edgelabel] at (7.25, 2.1) {cap \(2\)\\cost \(-10\)};

\draw[avail] (J1) to[bend right=12] (t);
\node[edgelabel] at (6.72, 0.6) {cap \(1\)\\cost \(-5\)};

\draw[optone] (J2) to[bend left=12] (t);
\node[edgelabel] at (6.58,-0.6) {cap \(1\)\\cost \(-20\)};

\draw[optone] (J2) to[bend right=28] (t);
\node[edgelabel] at (7.4,-2.1) {cap \(1\)\\cost \(-16\)};

\end{tikzpicture}%
}

\vspace{0.35em}
\scriptsize
\begin{tabular}{@{}lll@{}}
\toprule
Optimal shipments & Unmet demand & Objective \\
\midrule
\(p_1=1,\ p_2=2\) &
\(u_1^\ast=15,\ u_2^\ast=0\) &
\(61+\kappa^\ast=15\) \\
\bottomrule
\end{tabular}

\caption{\small Capacity \(C_A+C_B=3\), with \(\kappa^\ast=-46\).}
\label{fig:recourse_network_limited}
\end{subfigure}

\caption{Minimum-cost network flow interpretation of deterministic recourse problem. Each pallet is one unit of flow. Arc labels show capacity and cost. Unused pallets flow through dashed arcs. Blue arcs show an optimal flow; the number of triangular markers on a blue arc gives the amount of flow on that arc. Gray arcs are available but unused. In both panels, \(m=10\), \(D_1=25\), \(D_2=18\), \(w_1=1\), and \(w_2=2\).}
\label{fig:recourse_network_examples}
\end{figure}

Flow conservation at $J_j$ forces the $L_\ell\to J_j$ inflow to equal the $J_j\to t$ outflow, whose total capacity is $q_j+\mathbf{1}_{\{r_j>0\}}$. Hence, excess \rmb{pallets} at $\ell$ flow through $L_\ell\to t$. \rmb{Moreover, because \(-w_jm\le -w_jr_j\) for \(0\le r_j<m\) and \(w_j\ge0\), an optimal min-cost flow fills the full-pallet marginal arc before using the residual marginal arc, up to irrelevant ties when \(w_j=0\). Therefore, for any integer inflow \(p_j\) to \(J_j\), the minimum possible cost of the arcs leaving \(J_j\) is \(-w_j\min\{D_j,mp_j\}\).} 
Consider any feasible solution to~\eqref{eq:appendix_fixed_scenario}. If a shelter receives shipments more than its last useful pallet, reroute those excess pallets through unused-pallet arcs \(L_\ell\to t\); send the remaining useful pallets on arcs \(L_\ell\to J_j\), route the corresponding
marginal reductions through the arcs \(J_j\to t\). \rmb{This yields a feasible flow whose cost is}
\[
\rmb{
-\sum_{j\in\mathcal J}w_j\min\{D_j,mp_j\}.}
\]
Hence, the network optimizes exactly the same
objective function as \eqref{eq:appendix_fixed_scenario}.

With integer capacities, the integrality theorem for minimum-cost network flow yields an optimal integral flow. Let $p_j^\ast$ be the total inflow to shelter node $J_j$, which corresponds to the total number of pallets shipped to shelter $j$. Then, $(s^{2\ast},u^\ast)$ is an optimal solution to~\eqref{eq:appendix_fixed_scenario}, where  $s^{2\ast}_{\ell j}$ is equal to the flow $L_\ell \to J_j$ and $u_j^\ast=\max\{0,\,D_j-mp_j^\ast\}$. \rmb{The trailer-capacity constraints follow from flow conservation at the \(L_\ell\) nodes and the capacities of the arcs \(o\to L_\ell\), while the unmet-demand constraints follow from the definition of \(u_j^\ast\).} Since flow costs are the negatives of marginal reductions in weighted unmet demand, the optimal unmet-demand objective equals \(\sum_{j\in\mathcal J}w_jD_j+\kappa^\ast\), where $\kappa^\ast$ is the optimal min-cost network flow value. The network has $O(|\mathcal{J}|+|\mathcal{L}|)$ nodes and $O\!\left(|\mathcal{L}|+|\mathcal{J}|+\sum_{j\in\mathcal{J}}|\mathcal{L}_j|\right)$ arcs, so standard min-cost network flow algorithms run in polynomial time in the rational encoding length.  
\end{proof}

\propcomplexitydet*

\begin{proof}
Membership in NP follows from the finite-rational encoding convention stated in
Section~\ref{sec:complexity}. In the deterministic case, a certificate
consists of the shelter opening variables, trailer prepositioning variables, route assignment
variables, shipment variables, unmet demand variables, and PWL segment flow variables.
For any fixed certificate, constraints and the threshold objective inequality
can be checked in polynomial time by evaluating a finite system of rational linear
constraints. Hence, the deterministic decision problem is in NP.

For NP-hardness, we reduce from the exact-\(p\) median decision problem. An
instance consists of a finite customer set \(I\), a finite candidate facility set \(J\),
positive rational costs \(c_{ij}\) for assigning customer \(i\) to facility \(j\), an integer
\(p\) with \(1\le p\le |J|\), and a rational threshold \(B\). The question is whether there
exists a set \(P\subseteq J\) with \(|P|=p\) such that $\sum_{i\in I}\min_{j\in P} c_{ij}\le B$.
This decision problem is NP-hard \citep{hakimi1979algorithmic}\footnote{\cite{hakimi1979algorithmic} states the result for network-distance \(p\)-median instances, which may include zero self-assignment distances. For the case with strictly positive assignment costs, adding one to every assignment cost shifts every exact-\(p\) solution by \(|I|\) and preserves the threshold decision problem.}.

The reduction is direct: customers become origins, candidate facilities become candidate
shelters, opening exactly \(p\) facilities becomes the shelter-opening constraint \(N=p\),
and private route arcs encode the assignment costs. Build a deterministic \(\PWLCSO{}\)
instance with \(\mathcal{O}=I\), \(\mathcal{J}=J\), \(\mathcal{L}=J\), a single scenario
\(\xi_i=1\), \(N=p\), and threshold \(\eta=B\). For each pair \((i,j)\), create a private route
\(r_{ij}\) consisting of an arc \(a_{ij}\) with capacity $c_{a_{ij}}=1$. 
\rmb{Write \(y_{ij}:=y_{r_{ij}}\), and note that no other route uses arc \(a_{ij}\).} 
For each origin \(i\), set all route lengths
\(d_{r_{ij}}\) equal. Then \(\hat d_{ik}\) equals this common value for every candidate
shelter \(k\in\mathcal{J}\), and the strict inequality \(d_r>\hat d_{ik}\) never holds, so
the route-length acceptability constraint is non-binding.

\rmb{Route lengths are used only for the acceptability constraints; the PWL travel-cost data on the private arcs are specified independently in this instance.} Choose PWL parameters and arc-specific free-flow times so that the first segment covers unit
load, i.e., \(c_a(\rho^1-\rho^0)\ge1\), and has positive rational slope \(\PWLSlope_1\).
Set \(t^0_{a_{ij}}=c_{ij}/\PWLSlope_1\) and \(\lambda=1\). 
Since \(\sum_j y_{ij}=1\) and
\(y_{ij}\ge0\), sending load
\(y_{ij}\in[0,1]\) along arc \(a_{ij}\) stays on the first PWL segment and costs
\(t^0_{a_{ij}}\PWLSlope_1 y_{ij}=c_{ij}y_{ij}\). 

Make unmet relief demand zero by setting \(\mathcal{L}_j=\mathcal{L}\), \(m=1\), \(S=1\), and
\(K=|I||J|\). Every trailer location can serve every shelter because
\(\mathcal{L}_j=\mathcal{L}\). Moreover, under any feasible route assignment, the demand
at each shelter is at most \(|I|\); hence, a single trailer with \(|I||J|\) pallets is enough
to eliminate unmet demand at all shelters. \rmb{More explicitly, if \(D_j=\sum_{i\in I}y_{ij}\), then \(D_j\le |I|\) for every \(j\), so \(\sum_{j\in J}\lceil D_j\rceil\le |I||J|=K\). Thus the integer pallet-shipment variables can cover all shelter demands, possibly by oversupplying fractional remainders.}
Since the unmet demand is zero, for any open shelter set \(P\) of size \(p\), the route-assignment problem reduces to
\[
\min \sum_{i\in I}\sum_{j\in P} c_{ij}y_{ij}
\quad
\text{s.t.}\quad
\sum_{j\in P}y_{ij}=1,\quad y_{ij}\ge0.
\]
Its optimum assigns each origin to the cheapest open shelter. Therefore, the constructed
\(\PWLCSO{}\) objective coincides with the exact-\(p\) median objective for the corresponding
instance, and the threshold \(\eta=B\) is preserved. \rmb{The construction creates \(|I||J|\) routes and \(|I||J|\) private arcs and uses rational parameters with polynomial encoding length, so it is polynomial in the exact-\(p\) median input size.}

\rmb{It remains only to state the equivalence explicitly. If the exact-\(p\) median instance has a facility set \(P\subseteq J\) with \(|P|=p\) and \(\sum_{i\in I}\min_{j\in P}c_{ij}\le B\), open exactly the shelters in \(P\), assign each origin \(i\) to a cheapest open shelter in \(P\), preposition the single trailer at any location, and ship enough pallets to make all unmet demand zero. The resulting \(\PWLCSO{}\) solution has objective value at most \(B=\eta\). Conversely, suppose the constructed \(\PWLCSO{}\) instance has a feasible solution with objective value at most \(\eta=B\). Let \(P=\{j\in J:b_j=1\}\). Then \(|P|=p\), and the open-shelter constraints imply \(y_{ij}=0\) for all \(j\notin P\). Since unmet demand is nonnegative and the PWL travel contribution is at least \(\sum_{i\in I}\sum_{j\in P}c_{ij}y_{ij}\), we have}
\[
\rmb{
\sum_{i\in I}\min_{j\in P}c_{ij}
\le
\sum_{i\in I}\sum_{j\in P}c_{ij}y_{ij}
\le B.}
\]
\rmb{Hence the exact-\(p\) median instance is a yes-instance if and only if the constructed deterministic \(\PWLCSO{}\) instance is a yes-instance.} 
NP-hardness follows. 
\end{proof}

\propcomplexityfixedplan*

\begin{proof}
We reduce from the complement of CLIQUE. An instance of this problem is a simple graph \(G=(V,E)\), \(|V|=n\), and an integer \(2\le k\le n\); the yes-instances are those in which \(G\) has no \(k\)-clique. Create one
origin in \PWLCSO{}\ per vertex $v \in V$, that is \(\mathcal{O}=\{o_v:v\in V\}\). Set \(N=1\) with
\(\mathcal{J}=\mathcal{L}=\{d\}\). Use the budgeted uncertainty set
\[
\Xi=\left\{\xi\in[0,1]^n:\sum_{i=1}^n\xi_i\le k\right\}.
\]
Total evacuee count in any \(\xi\in\Xi\) is at most \(k\). To ensure that unmet demand does not affect the reduction, set \(S=1\), \(K=k\), \(m=1\), and \(\mathcal L_d=\{d\}\). A feasible first-stage plan can place the single trailer at \(d\), and the recourse choice \(s^2_{dd}=k\) satisfies the demand-covering constraint with \(u_d=0\). Hence, unmet demand is zero in this instance. Set \(\lambda=1\).

\rmb{We now construct the road network for the \PWLCSO{}\ instance, which consists of a single route $r_v$ between each origin $o_v$ and shelter location $d$. For each edge \(e\in E\), create one shared arc \(a_e\) between two auxiliary nodes in the road network. These auxiliary nodes serve only as endpoints for \(a_e\) and the route-private connector arcs described below.  For each \(v\in V\), order the edges $e \in E$ incident to \(v\) arbitrarily as \(e^v_1,\ldots,e^v_{\deg(v)}\). If \(\deg(v)>0\), the route \(r_v\) that starts at \(o_v\), uses a route-private connector arc to reach the tail of \(a_{e^v_1}\), traverses the shared arcs \(a_{e^v_1},\ldots,a_{e^v_{\deg(v)}}\)  using route-private connectors from the head of \(a_{e^v_h}\) to the tail of \(a_{e^v_{h+1}}\) for \(h=1,\ldots,\deg(v)-1\), and then uses a route-private connector arc leaving the head of \(a_{e^v_{\deg(v)}}\) to the shelter node \(d\). Thus, before padding, \(r_v\) contains \(\deg(v)\) shared arcs and \(\deg(v)+1\) private connector arcs. If \(\deg(v)=0\), route \(r_v\) is initially a single arc from \(o_v\) to \(d\). Finally, extend the terminal route-private connector to $d$ into a route-private chain by adding dummy arcs, so that the route between every origin $o_v$ and $d$ uses exactly \(R:=2\Delta(G)+1\) arcs, where \(\Delta(G)\) is the maximum degree of \(G\). For \(\deg(v)=0\), turn \(r_v\) into a route-private chain from \(o_v\) to \(d\) of length \(R\) by adding dummy arcs. Only the arcs \(a_e\) are shared between routes; connector and dummy arcs are route-private. 
Furthermore,  \(d_{r_v}=\hat d_{o_vd}=R\) for every \(v\in V\)  so that every listed route is acceptable when \(d\) is open. Figures~\ref{fig:fixed_plan_clique_example}a and \ref{fig:fixed_plan_clique_example}b illustrate the four-vertex source graph $G(V,E)$ and the routes between each origin and shelter location with private and shared arcs, respectively.}

\begin{figure}[t]
\centering
\begin{subfigure}[t]{0.34\textwidth}
\centering
\resizebox{0.9\textwidth}{!}{%
\begin{tikzpicture}[
    vertex/.style={circle, draw, minimum size=7mm, inner sep=0pt, font=\small},
    cliquev/.style={circle, draw, fill=blue!12, minimum size=7mm, inner sep=0pt, font=\small},
    gedge/.style={line width=0.7pt, gray!70},
    cliqueedge/.style={line width=1.2pt, blue!75}
]

\node[cliquev] (v1) at (0,1.7) {$1$};
\node[cliquev] (v2) at (-1.35,0) {$2$};
\node[cliquev] (v3) at (1.35,0) {$3$};
\node[vertex]  (v4) at (2.8,0) {$4$};

\draw[cliqueedge] (v1) -- (v2);
\draw[cliqueedge] (v1) -- (v3);
\draw[cliqueedge] (v2) -- (v3);
\draw[gedge]      (v3) -- (v4);

\node[font=\scriptsize, align=center] at (1.15,-1.05)
{\(V=\{1,2,3,4\}\)\\
\(E=\{\{1,2\},\{1,3\},\{2,3\},\{3,4\}\}\)};

\end{tikzpicture}%
}
\caption{\small Source graph \(G\).}
\label{fig:fixed_plan_clique_graph}
\end{subfigure}
\hfill
\begin{subfigure}[t]{0.62\textwidth}
\centering
\resizebox{0.98\textwidth}{!}{%
\begin{tikzpicture}[
    >=Latex,
    arcbox/.style={
      draw,
      rounded corners,
      minimum width=9mm,
      minimum height=6mm,
      inner sep=1pt,
      font=\scriptsize,
      align=center
    },
    origin/.style={circle, draw=green!75!black, fill=green!8, minimum size=6mm, inner sep=0pt, font=\scriptsize},
    shelter/.style={rectangle, draw=green!75!black, fill=green!8, rounded corners=1pt, minimum width=5mm, minimum height=30mm, font=\scriptsize},
    shared/.style={arcbox, fill=blue!10},
    private/.style={arcbox, fill=gray!8},
    dummy/.style={arcbox, fill=gray!20, dashed},
    routearrow/.style={->, thin, gray!75},
    rowlabel/.style={font=\small, align=right}
]

\node[rowlabel] (r1) at (-1.75, 0) {$r_1:$};
\node[rowlabel] (r2) at (-1.75,-0.85) {$r_2:$};
\node[rowlabel] (r3) at (-1.75,-1.70) {$r_3:$};
\node[rowlabel] (r4) at (-1.75,-2.55) {$r_4:$};

\node[origin] (o1) at (-1.05, 0)     {\(o_1\)};
\node[origin] (o2) at (-1.05,-0.85)  {\(o_2\)};
\node[origin] (o3) at (-1.05,-1.70)  {\(o_3\)};
\node[origin] (o4) at (-1.05,-2.55)  {\(o_4\)};

\node[private] (r1c1) at (0,0)       {\(p^0_1\)};
\node[shared]  (r1c2) at (1.15,0)    {\(a_{12}\)};
\node[private] (r1c3) at (2.30,0)    {\(p^1_1\)};
\node[shared]  (r1c4) at (3.45,0)    {\(a_{13}\)};
\node[private] (r1c5) at (4.60,0)    {\(p^2_1\)};
\node[dummy]   (r1c6) at (5.75,0)    {\(\delta^1_1\)};
\node[dummy]   (r1c7) at (6.90,0)    {\(\delta^2_1\)};

\node[private] (r2c1) at (0,-0.85)       {\(p^0_2\)};
\node[shared]  (r2c2) at (1.15,-0.85)    {\(a_{12}\)};
\node[private] (r2c3) at (2.30,-0.85)    {\(p^1_2\)};
\node[shared]  (r2c4) at (3.45,-0.85)    {\(a_{23}\)};
\node[private] (r2c5) at (4.60,-0.85)    {\(p^2_2\)};
\node[dummy]   (r2c6) at (5.75,-0.85)    {\(\delta^1_2\)};
\node[dummy]   (r2c7) at (6.90,-0.85)    {\(\delta^2_2\)};

\node[private] (r3c1) at (0,-1.70)       {\(p^0_3\)};
\node[shared]  (r3c2) at (1.15,-1.70)    {\(a_{13}\)};
\node[private] (r3c3) at (2.30,-1.70)    {\(p^1_3\)};
\node[shared]  (r3c4) at (3.45,-1.70)    {\(a_{23}\)};
\node[private] (r3c5) at (4.60,-1.70)    {\(p^2_3\)};
\node[shared]  (r3c6) at (5.75,-1.70)    {\(a_{34}\)};
\node[private] (r3c7) at (6.90,-1.70)    {\(p^3_3\)};

\node[private] (r4c1) at (0,-2.55)       {\(p^0_4\)};
\node[shared]  (r4c2) at (1.15,-2.55)    {\(a_{34}\)};
\node[private] (r4c3) at (2.30,-2.55)    {\(p^1_4\)};
\node[dummy]   (r4c4) at (3.45,-2.55)    {\(\delta^1_4\)};
\node[dummy]   (r4c5) at (4.60,-2.55)    {\(\delta^2_4\)};
\node[dummy]   (r4c6) at (5.75,-2.55)    {\(\delta^3_4\)};
\node[dummy]   (r4c7) at (6.90,-2.55)    {\(\delta^4_4\)};

\node[shelter] (d) at (8.1,-1.26) {$d$};

\coordinate (d1) at ($(d.west)+(0,1)$);
\coordinate (d2) at ($(d.west)+(0,0.25)$);
\coordinate (d3) at ($(d.west)+(0,-0.25)$);
\coordinate (d4) at ($(d.west)+(0,-1)$);

\draw[routearrow] (o1.east) -- (r1c1.west);
\draw[routearrow] (r1c1.east) -- (r1c2.west);
\draw[routearrow] (r1c2.east) -- (r1c3.west);
\draw[routearrow] (r1c3.east) -- (r1c4.west);
\draw[routearrow] (r1c4.east) -- (r1c5.west);
\draw[routearrow] (r1c5.east) -- (r1c6.west);
\draw[routearrow] (r1c6.east) -- (r1c7.west);
\draw[routearrow] (r1c7.east) -- (d1.west);

\draw[routearrow] (o2.east) -- (r2c1.west);
\draw[routearrow] (r2c1.east) -- (r2c2.west);
\draw[routearrow] (r2c2.east) -- (r2c3.west);
\draw[routearrow] (r2c3.east) -- (r2c4.west);
\draw[routearrow] (r2c4.east) -- (r2c5.west);
\draw[routearrow] (r2c5.east) -- (r2c6.west);
\draw[routearrow] (r2c6.east) -- (r2c7.west);
\draw[routearrow] (r2c7.east) -- (d2.west);

\draw[routearrow] (o3.east) -- (r3c1.west);
\draw[routearrow] (r3c1.east) -- (r3c2.west);
\draw[routearrow] (r3c2.east) -- (r3c3.west);
\draw[routearrow] (r3c3.east) -- (r3c4.west);
\draw[routearrow] (r3c4.east) -- (r3c5.west);
\draw[routearrow] (r3c5.east) -- (r3c6.west);
\draw[routearrow] (r3c6.east) -- (r3c7.west);
\draw[routearrow] (r3c7.east) -- (d3.west);

\draw[routearrow] (o4.east) -- (r4c1.west);
\draw[routearrow] (r4c1.east) -- (r4c2.west);
\draw[routearrow] (r4c2.east) -- (r4c3.west);
\draw[routearrow] (r4c3.east) -- (r4c4.west);
\draw[routearrow] (r4c4.east) -- (r4c5.west);
\draw[routearrow] (r4c5.east) -- (r4c6.west);
\draw[routearrow] (r4c6.east) -- (r4c7.west);
\draw[routearrow] (r4c7.east) -- (d4.west);

\node[font=\scriptsize, align=center] at (3.45,0.68)
{All routes are padded to \(R=2\Delta(G)+1=7\) arcs.};

\end{tikzpicture}%
}
\caption{\small \rmb{Constructed routes. As an input to \PWLCSO{}, there is one route \(r_v\) in the route set of each origin \(o_v\) to shelter \(d\). Blue boxes denote shared arcs between routes, light gray boxes denote route-private connector arcs, and dashed gray boxes denote route-private dummy arcs used for padding at the end of each route; occurrence of a shared arc on two different routes indicates both routes use that arc.}}
\label{fig:fixed_plan_clique_routes}
\end{subfigure}

\vspace{0.8em}


\begin{subfigure}[t]{0.98\textwidth}
\centering
\resizebox{0.9\textwidth}{!}{%
\begin{tikzpicture}[
    thresholdbox/.style={
      draw,
      rounded corners,
      fill=blue!5,
      inner sep=5pt,
      minimum width=6.0cm,
      align=center,
      font=\scriptsize
    },
    box/.style={
      draw,
      rounded corners,
      fill=gray!6,
      inner sep=5pt,
      minimum width=6.25cm,
      align=center,
      font=\scriptsize
    }
]


\node[box] (clique) at (0,0)
{\textbf{Subset scenario} \(S=\{1,2,3\}\)\\[0.5mm]
Shared load-two arcs: \(a_{12},a_{13},a_{23}\)\\
\(|E(S)|=3\), so cost \(=h(1)Rk+\gamma\cdot 3=21+6=27\)\\[0.5mm]
\(\mathbf{27>\Theta}\): there is a 3-clique};

\node[box] (nonclique) at (7.1,0)
{\textbf{Non-clique subset scenario} \(S=\{1,3,4\}\)\\[0.5mm]
Shared load-two arcs: \(a_{13},a_{34}\)\\
\(|E(S)|=2\), so cost \(=h(1)Rk+\gamma\cdot 2=21+4=25\)\\[0.5mm]
\(\mathbf{25\le\Theta}\): no 3-clique};

\end{tikzpicture}%
}
\caption{\small Example for \(3\)-clique. With \(h(0)=0\), \(h(1)=1\), \(h(2)=4\), we have \(\gamma=2\).
Since $k=3$, the threshold is \(\Theta=21+2(3-\tfrac12)=26\). }
\label{fig:fixed_plan_clique_costs}
\end{subfigure}

\caption{Illustration of the reduction from the complement of CLIQUE.}
\label{fig:fixed_plan_clique_example}
\end{figure}

 It remains to choose the breakpoints of the PWL approximation to the $g(\rho)$ function, denoted by $g_{\texttt{PWL}}(\rho)$. The intended effect is as follows: in the worst case, each route with a flow contributes a fixed base amount, and each edge $e \in E$ whose both endpoints are selected to have a nonzero evacuee count realization contributes an additional increment to the approximation due to higher flow on the shared arc $a_e$. We use breakpoints $\{1,2\}$ for $g_{\texttt{PWL}}(\rho)$. Then, the slope of the first segment $\PWLSlope_1 = \frac{g(1)}{1} < \PWLSlope_2 = \frac{g(2)-g(1)}{1}$ since $g$ is a strictly increasing convex function.  For arc $a$ with load \(x\), let $\rho_a = \frac{x}{c_a}$, and \(h_a(x)=t_a^0c_ag_{\texttt{PWL}}(\rho_a)\) be the PWL approximate travel time. Set \(c_a=1\) and \(t^0_a=1\) for every arc in the road network. Then, $\rho_a = \frac{x}{c_a}=x$ and $h_a(x)=g_{\texttt{PWL}}(x)$, and we have $h(0)=0,h(1)=\PWLSlope_1$, and $h(2)=\PWLSlope_1+\PWLSlope_2$. The extra cost of carrying two units of flow on one shared arc rather than one unit on each of two private arcs is $\gamma:=h(2)-2h(1)=\PWLSlope_2-\PWLSlope_1>0$.


Every private arc appears on a single route, and every shared arc \(a_e\) belongs to exactly the two routes associated with the endpoints of \(e\). In addition, \(0\le \xi_i\le1\). Thus, no arc load exceeds two, and the specified part of \(h\) between 0 and 2 is sufficient for every possible uncertainty realization.
For a realization \(\xi\), let \(X_a(\xi)\) be the load on arc \(a\). Since $\lambda=1$, the objective value is
\[
F(\xi):=\sum_{a\in\mathcal A} h(X_a(\xi)).
\]
The function \(F\) is convex in \(\xi\) because each arc load is affine in \(\xi\) and \(h\) is convex. Therefore, a maximizer over the polytope \(\Xi\) will be at an extreme point.
Since \(k\) is integral, the extreme points of \(\Xi\) are incidence vectors of subsets \(S\subseteq V\) with \(|S|\le k\). Because all PWL slopes used in the construction are positive, no worst-case realization leaves budget unused: if \(\sum_i \xi_i<k\), then some component \(\xi_v<1\) can be increased while preserving feasibility, which strictly increases the route cost. Hence a maximizing extreme point can be chosen with \(|S|=k\).
Fix such a subset \(S\), and let \(X_a(S)\) denote the induced load on arc \(a\). If all selected routes used private arcs only, the total cost would be \(h(1)R|S|\), because each of the \(|S|\) selected routes contains exactly \(R\) arcs. The actual construction differs from this private arc baseline only on shared arcs \(a_e\). Let $E(S) \subseteq E$ be the set of edges with both endpoints in $S$. If \(e \in E(S)\), then the two route incidences that would contribute \(2h(1)\) in the private arc baseline instead appear as one shared arc with load two and cost \(h(2)\). This adds \(\gamma=h(2)-2h(1)\). If \(e\) has zero or one endpoint in \(S\), there is no such increment. Hence
\[
\sum_{a\in\mathcal A}h(X_a(S))=h(1)R|S|+\gamma |E(S)|.
\]
Since \(|S|=k\), maximizing the scenario value is equivalent to maximizing \(|E(S)|\) over all \(k\)-vertex subsets of \(G\).
Set the validation threshold to
\[
\Theta:=h(1)Rk+\gamma\left(\binom{k}{2}-\frac12\right).
\]
If \(G\) contains a \(k\)-clique, then a $k$-subset $S \subseteq V$ satisfies \(|E(S)|=\binom{k}{2}\), and the corresponding scenario value is strictly greater than \(\Theta\). Thus, the first-stage decisions fail the robust threshold test. Conversely, if \(G\) has no \(k\)-clique, then every \(k\)-subset satisfies \(|E(S)|\le \binom{k}{2}-1 < \Theta\). By the extreme-point argument, this bounds the maximum over all \(\xi\in\Xi\). Therefore, the fixed plan satisfies the robust threshold test {\it if and only if} \(G\) has no \(k\)-clique. Figure~\ref{fig:fixed_plan_clique_example}c gives an example of threshold calculation for 3-CLIQUE on a four-vertex graph.

The construction has \(O(|E|+n\Delta(G))\subseteq O(n^2)\) shared arcs and route-arc incidences and uses rational data of polynomial encoding length. Since the complement of CLIQUE is coNP-complete, as CLIQUE is NP-complete \citep{karp1972reducibility}, fixed-plan robust validation is coNP-hard. 
\end{proof}

\propcomplexitysigmahard*

\begin{proof}
quantified complement form 
 threshold problem for Max-3SAT to the decision version of \PWLCSO{}\ with a bounded polyhedral uncertainty set. The input is a proper 3-Conjunctive Normal Form (3-CNF) formula
\[
F(\mathbf{\chi},\mathbf{\upsilon})=\bigwedge_{k=1}^q d_k
\]
with existential variables \(\mathbf{\chi}=(\chi_1,\ldots,\chi_n)\), universal variables
\(\upsilon=(\upsilon_1,\ldots,\upsilon_M)\), and an integer threshold \(B\). Each clause $d_k=(\ell_{k1}\vee \ell_{k2}\vee \ell_{k3})$ has exactly three literals corresponding to distinct variables or their negations. The question is
whether
\[
\exists \chi\in\{0,1\}^n:
\operatorname{sat}_F(\chi,\upsilon)\le B, \forall \upsilon \in\{0,1\}^M
\]
where \(\operatorname{sat}_F(\chi,\upsilon)\) is the number of satisfied
clauses in \(F\). Without loss of generality, we assume \(M\ge1\). 
\rmb{This threshold problem is \(\Sigma^p_2\)-complete: membership follows from its \(\exists\forall\) certificate structure. For hardness, it suffices to consider the special case \(B=q-1\): then \(\operatorname{sat}_F(\boldsymbol\chi, \boldsymbol\upsilon)\le q-1\) is equivalent to falsifying the 3-CNF formula \(F\), so this special case is the complement of the standard \(\forall\exists\)-3SAT problem, which is \(\Pi^p_2\)-complete~\citep{schaefer2002completeness}.}

The reduction follows the timing of \PWLCSO{}: the first-stage shelter opening
decisions encode \(\chi\), the uncertainty realization encodes
\(\upsilon\), and the PWL evacuation time counts the number of satisfied clauses.

\emph{Step 1:} For each existential variable \(\chi_i\), create one origin \(o_i^{\chi}\), and two candidate shelters \(j_i^+\) and \(j_i^-\). Origin \(o_i^\chi\) has two designated routes: route \(r_{i,+}\) to \(j_i^+\) and route \(r_{i,-}\) to \(j_i^-\).

For each universal variable \(\upsilon_m\), create two origins \(o_m^{\upsilon+}\) and
\(o_m^{\upsilon-}\), each with a single designated route to shelter \(j_U\), called the universal shelter. Let
\[
\mathcal O = \{o_1^{\chi},\ldots,o_n^{\chi},o_1^{\upsilon+},\dots, o_M^{\upsilon+}, o_1^{\upsilon-},\ldots, o_M^{\upsilon-}\}, \ \mathcal J= \mb{\mathcal L}=
\{j_U\}\cup\{j_1^+,\hdots,j_n^+,j_1^-,\hdots,j_n^-\},
\text{ and }
N=n+1.
\]
The origins created for universal variables have routes only to \(j_U\). Therefore,
constraints~\eqref{eq:route_assign_model1}--\eqref{eq:route_open_shelter_model1}
force \(b_{j_U}=1\). Since exactly $N=n+1$ shelters are opened,  \(n\)
shelters among the \(2n\) shelters \(j_i^+,j_i^-\) are open. Furthermore, from~\eqref{eq:route_assign_model1}--\eqref{eq:route_open_shelter_model1}, each origin \(o_i^\chi\) must assign evacuees to 
open \(j_i^+\), \(j_i^-\) shelters. Since \(n\)
shelters among the \(2n\) shelters \(j_i^+,j_i^-\) are open, exactly one shelter in each pair \(\{j_i^+,j_i^-\}\) is open. We interpret
\[
b_{j_i^+}=1 \rightarrow \chi_i=1,
\qquad
b_{j_i^-}=1 \rightarrow \chi_i=0.
\]

\emph{Step 2:} Define the uncertainty set
\[
\Xi=
\left\{
\xi:
\xi_{o_i^\chi}=1,\  i=1,\ldots,n,\quad
\xi_{o_m^{\upsilon+}}+\xi_{o_m^{\upsilon-}}=1, \ \xi_{o_m^{\upsilon+}},\xi_{o_m^{\upsilon-}}\ge 0,\  m=1,\ldots,M
\right\}.
\]
This is a rational polytope. We interpret its extreme points as universal variable \(\upsilon\in\{0,1\}^M\)
assignments, that is
\[
 \xi_{o_m^{\upsilon+}}=1 \rightarrow \upsilon_m=1,
\qquad
\xi_{o_m^{\upsilon-}}=1 \rightarrow\upsilon_m=0.
\]
Note that, for every \(\xi\in\Xi\), total evacuee count is constant:
\[
\sum_{i=1}^n \xi_{o_i^\chi}
+
\sum_{m=1}^M
\left(\xi_{o_m^{\upsilon+}}+\xi_{o_m^{\upsilon-}}\right)
=
n+M.
\]

\emph{Step 3:}
For each clause
$d_k=(\ell_{k1}\vee \ell_{k2}\vee \ell_{k3})$
 consider seven sign variant clauses
\[
\mathcal G_k
=
\left\{
(m_1\vee m_2\vee m_3):
m_h\in\{\ell_{kh},\overline{\ell}_{kh}\},\ h=1,2,3
\right\}
\setminus \{d_k\},
\]
where \(\overline{\ell}\) denotes the complementary signed literal.                                                     
For example, if $d_1=(x_1\vee \neg y_1\vee y_2)$, then seven variant clauses of $d_1$ are $\mathcal G_1=\{(\neg x_1\vee \neg y_1\vee y_2),(\neg x_1\vee y_1\vee y_2), (\neg x_1\vee \neg y_1\vee \neg y_2),(\neg x_1\vee y_1\vee \neg y_2),(x_1\vee y_1\vee y_2), (x_1\vee \neg y_1\vee \neg y_2),(x_1\vee y_1\vee \neg y_2)
\}$.

For each clause
$d_k$, create one shared arc
\(a_{kg}\) for every variant clause \(g\in\mathcal G_k\). This arc is used by the three routes corresponding to the three
literals of \(g\) being false:
\[
\begin{array}{c|c|c}
\text{Literal in }g &\PWLCSO{}\ \text{variable corresponding to that literal being false} & \text{Route that includes }a_\rmb{kg}\\
\hline
\chi_i & b_{j_i^-}=1 &r_{i,-}\\
\neg \chi_i & b_{j_i^+}=1& r_{i,+}\\
\upsilon_m & \xi_{o_m^{\upsilon-}}=1 &\text{route from }o_m^{\upsilon-}\text{ to }j_U\\
\neg \upsilon_m & \xi_{o_m^{\upsilon+}}=1 &\text{route from }o_m^{\upsilon+}\text{ to }j_U.
\end{array}
\]
All other arcs on considered routes are route-private. Add route-private dummy arcs as needed so that every route contains the same number of arcs\mb{, i.e., \(d_r=R\)}. \mb{Set the route-acceptability threshold parameter \(\hat d_{ok}=R\) for every origin \(o\in\mathcal O\) and candidate shelter \(k\in\mathcal J\).} 
Then constraint~\eqref{eq:route_restict_model1} does not prohibit any designated route because route lengths are all equal. Route feasibility is determined solely by the shelter-opening decisions.

For a solution of the created \PWLCSO{}\ instance corresponding to Boolean assignment \((\chi,\upsilon)\), the traffic load on shared arc
\(a_\rmb{kg}\) is
\[
x_{a_\rmb{kg}}(\chi,\upsilon)
=
\#\{\text{literals of }g\text{ that are false under }(\chi,\upsilon)\}.
\]
Thus \(x_{a_{kg}}=3\) if and only if the variant clause \(g\) is false.

For clause \(d_k\), exactly one of the
eight sign variants of \(d_k\) is false under any Boolean assignment. If \(d_k\) is false, then no generated clause in
\(\mathcal G_k\) is false. If \(d_k\) is true, then the unique false sign variant
belongs to \(\mathcal G_k\), so exactly one generated clause in \(\mathcal G_k\) is
false. Therefore,
\[
\sum_{g\in\mathcal G_k}
\mathbf 1\{x_{a_{kg}}=3\}
=
\mathbf 1\{d_k\text{ is satisfied}\}.
\]

\emph{Step 4:} Set \(\lambda=1\).
To ensure that unmet demand does not affect the threshold calculation for the intended first-stage plans, set \(m=1\), \(S=1\), \(K=n+M\), and \(\mathcal L_j=\{j_U\}\) for every shelter \(j\in\mathcal J\). A feasible first-stage plan can place the single trailer at \(j_U\). Under any shelter-opening pattern induced by Step 1, each open existential shelter receives exactly one unit of demand, and \(j_U\) receives exactly \(M\) units because \(\xi_{o_m^{\upsilon+}}+\xi_{o_m^{\upsilon-}}=1\). Thus, the integer recourse can ship one pallet to each open existential shelter and \(M\) pallets to \(j_U\), using exactly \(n+M\) pallets, and can set \(u_j(\xi)=0\) for all \(j\) and \(\xi\). For any feasible first-stage plan without such relief placement, unmet demand is nonnegative and can only increase the objective.

\emph{Step 5:} Set \(c_a=1\) and \(t_a^0=1\) for every constructed arc \(a\). In the PWL approximation of the strictly convex $g(\rho)$ function, denoted by $g_{\texttt{PWL}}(\rho)$, let \(\rho_0=0<\rho_1<\rho_2\) be the  breakpoints, such that $2\le \rho_1<3$ and $3\le \rho_2$. The slope of the first segment $\PWLSlope_1 = \frac{g(\rho_1)}{\rho_1} < \PWLSlope_2 = \frac{g(\rho_2)-g(\rho_1)}{\rho_2 - \rho_1}$.

Let \(h_a(x)=t_a^0c_ag_{\texttt{PWL}}(\rho_a)\) be the PWL approximate travel time for an arc with load \(x\). By construction, $\rho_a = \frac{x}{c_a}=x$ and $h_a(x)=g_{\texttt{PWL}}(x)$. Due to the segment-filling property of the PWL formulation, loads \(x=0,1,2\) use only the first segment of the approximation, while load \(3\) uses a positive amount of the second segment. Hence, $h(0)=0,
 h(1)=\PWLSlope_1,
 h(2)=2\PWLSlope_1,
 h(3)=3\PWLSlope_1+\delta$,
where
$\delta=(\PWLSlope_2-\PWLSlope_1)(3-\rho_1)>0$.

No arc load in the construction exceeds three: each shared arc is used by at most the three routes associated with its three literals, and each route-private arc is used by only one route. Therefore, for a feasible solution of the created \PWLCSO{} instance corresponding to Boolean assignment \((\chi,\upsilon)\), $h(x_{a_\rmb{kg}})=\PWLSlope_1 x_{a_\rmb{kg}}+\delta\,\mathbf 1\{x_{a_\rmb{kg}}=3\}$
for every shared arc $a_\rmb{kg}$, while for route-private arcs $h(x_a) = \PWLSlope_1 x_a$ as their load $x_a \in \{0,1\}$.

Every designated route contains exactly \(R\) arcs in the created \PWLCSO{} instance, and total demand is equal to \(n+M\) for all evacuee count realization $\xi\in\Xi$. Thus
\[
\sum_{a\in\mathcal A}x_a(\xi)=R(n+M)
\qquad
\forall \xi\in\Xi,
\]
so the first-segment contribution \(\PWLSlope_1R(n+M)\) is constant. At binary extreme points, the only solution-dependent portion is
\[
\delta
\sum_{k=1}^q
\sum_{g\in\mathcal G_k}
\mathbf 1\{\rmb{x}_{a_\rmb{kg}}=3\}.
\]
By Step 3, this expression equals \(\delta\,\operatorname{sat}_F(\chi,\upsilon)\).


\emph{Step 6:}
Set the threshold $\eta=\PWLSlope_1R(n+M)+\delta B$ for the decision version of the constructed \PWLCSO{}\ instance.

\emph{Correctness.}
\rmb{If the source instance is a yes-instance, choose the shelter openings corresponding to a witnessing assignment \(\chi\), open \(j_U\), assign each origin to its designated route to an open shelter, and place the single trailer at \(j_U\). By Step 4, the adaptive recourse can set \(u_j(\xi)=0\) for all \(j\) and all \(\xi\in\Xi\). For this fixed plan, the PWL evacuation time as a function of \(\xi\) is convex: each arc load is affine in \(\xi\), and the segment-flow PWL value function \(h\) is convex. The uncertainty set \(\Xi\) is a polytope. Therefore, a worst-case realization for this plan can be chosen at an extreme point of \(\Xi\), which corresponds to a Boolean assignment \(\upsilon\in\{0,1\}^M\) by Step 2. At these extreme points, Steps 3 and 5 give the objective value}
\(
\rmb{\PWLSlope_1R(n+M)
+
\delta
\operatorname{sat}_F(\chi,\upsilon).}
\)
\rmb{Since \(\operatorname{sat}_F(\chi,\upsilon)\le B\) for all \(\upsilon\in\{0,1\}^M\), the worst-case objective value of this plan is at most}
\(
\rmb{\PWLSlope_1R(n+M)+\delta B
=
\eta.}
\)
\rmb{Thus the constructed \PWLCSO{} instance is a yes-instance.}

\rmb{Conversely, suppose the constructed \PWLCSO{} instance is a yes-instance, and let \((b,s^1,y)\) be a feasible first-stage plan with worst-case objective value at most \(\eta\). By Step 1, this plan}
encodes a unique assignment
\(\chi\in\{0,1\}^n\). 
\rmb{For every 
assignment \(\upsilon\in\{0,1\}^M\), consider the corresponding extreme point of \(\Xi\). The total objective at this realization is at most \(\eta\). Since unmet demand is nonnegative, the PWL evacuation-time component is also at most \(\eta\). By Steps 3 and 5,}
\[
\PWLSlope_1R(n+M)
+
\delta
\operatorname{sat}_F(\chi,\upsilon)
\le
\PWLSlope_1R(n+M)+\delta B
\quad
\forall \upsilon\in\{0,1\}^M.
\]
Since \(\delta>0\), this is equivalent to
$\exists \chi\in\{0,1\}^n :
\operatorname{sat}_F(\chi,\upsilon)\le B \quad \forall \upsilon \in\{0,1\}^M$,
which is the source problem.

The construction creates \(n+2M\) origins, \(2n+1\) shelters and trailer locations,
\(7q\) shared arcs, and a polynomial number of route-private arcs. 
Hence,
this is a polynomial-time many-one reduction from a \(\Sigma^p_2\)-complete problem
to the decision version of \PWLCSO{}. Therefore, \PWLCSO{} is
\(\Sigma^p_2\)-hard. 

\end{proof}

The following example illustrates the reduction described above.

\begin{example}
\color{black}
Consider the source instance with two existential variables, one universal variable, two clauses, and threshold \(B=1\): $F(\chi,\upsilon)=d_1\wedge d_2,
\
 d_1=(\chi_1\vee \chi_2\vee \upsilon_1),
\
 d_2=(\chi_1\vee \chi_2\vee \neg \upsilon_1)$.

The source question asks whether there exists an assignment of \(\chi\) such that at most one clause is satisfied for every value of \(\upsilon_1\). The assignment \(\chi_1=\chi_2=0\) satisfies this condition: if \(\upsilon_1=1\), then only \(d_1\) is satisfied, while if \(\upsilon_1=0\), then only \(d_2\) is satisfied. Hence \(\operatorname{sat}_F(\chi,\upsilon)=1\le B\) for both universal assignments.

\emph{Step 1:}
The constructed \PWLCSO{} instance has origins
\(
\mathcal O=\{o_1^\chi,o_2^\chi,o_1^{\upsilon+},o_1^{\upsilon-}\},
\)
shelters and trailer locations
\(
\mathcal J=\mathcal L=\{j_U,j_1^+,j_2^+,j_1^-,j_2^-\}
\), and \(N=3\).
Opening \(j_i^+\) encodes \(\chi_i=1\), and opening \(j_i^-\) encodes \(\chi_i=0\). Origin \(o_i^\chi\) has two designated routes, \(r_{i,+}\) to \(j_i^+\) and \(r_{i,-}\) to \(j_i^-\). For the universal variable, let \(r_{U,+}\) denote the route from \(o_1^{\upsilon+}\) to \(j_U\), and let \(r_{U,-}\) denote the route from \(o_1^{\upsilon-}\) to \(j_U\). The assignment \(\chi_1=\chi_2=0\) is represented by opening \(j_U\), \(j_1^-\), and \(j_2^-\); the existential routes carrying demand are then \(r_{1,-}\) and \(r_{2,-}\).

\emph{Step 2:}
The uncertainty set is
\(
\Xi=\{\xi:\xi_{o_1^\chi}=\xi_{o_2^\chi}=1,\;\xi_{o_1^{\upsilon+}}+\xi_{o_1^{\upsilon-}}=1,\;\xi_{o_1^{\upsilon+}},\xi_{o_1^{\upsilon-}}\ge0\}.
\)
Its two extreme points encode the two values of the universal variable: \(\upsilon_1=1\) corresponds to \(\xi_{o_1^{\upsilon+}}=1\), and \(\upsilon_1=0\) corresponds to \(\xi_{o_1^{\upsilon-}}=1\). Thus \(r_{U,+}\) carries one unit of demand when \(\upsilon_1=1\), whereas \(r_{U,-}\) carries one unit of demand when \(\upsilon_1=0\).

\emph{Step 3:}
For each source clause \(d_k\), the constructed road network contains seven shared arcs \(a_{k,g}\), one for each sign variant \(g\in\mathcal G_k\). The index \(g\) denotes a sign variant of \(d_k\) other than \(d_k\) itself. Each arc \(a_{k,g}\) is placed on three designated routes, determined by the literals of \(g\): for each literal in \(g\), the route representing the assignment that makes that literal false contains \(a_{k,g}\). Consequently, \(a_{k,g}\) attains load three exactly when all three literals of \(g\) are false under the opened shelters and the realized universal demand.

In this example, the construction creates seven shared arcs for \(d_1\) and seven shared arcs for \(d_2\). The following table reports two of these fourteen arcs, namely the arcs that attain load three under the assignment \(\chi_1=\chi_2=0\) for the two possible values of \(\upsilon_1\).
\medskip

\begin{center}
\resizebox{0.56\textwidth}{!}{
\begin{tabular}{c c c}
\toprule
Source clause & Sign variant \(g\) & Routes containing \(a_{k,g}\)\\
\midrule
\(d_1\) & \(g_1=(\chi_1\vee\chi_2\vee\neg\upsilon_1)\in\mathcal G_1\) & \(r_{1,-},\,r_{2,-},\,r_{U,+}\)\\
\(d_2\) & \(g_2=(\chi_1\vee\chi_2\vee\upsilon_1)\in\mathcal G_2\) & \(r_{1,-},\,r_{2,-},\,r_{U,-}\)\\
\bottomrule
\end{tabular}
}
\end{center}

\medskip
The first row follows because route \(r_{1,-}\) represents \(\chi_1=0\), route \(r_{2,-}\) represents \(\chi_2=0\), and route \(r_{U,+}\) carries demand when \(\upsilon_1=1\), which makes \(\neg\upsilon_1\) false. Hence, when \((\chi_1,\chi_2,\upsilon_1)=(0,0,1)\), all three routes containing \(a_{1,g_1}\) carry one unit of demand, so \(a_{1,g_1}\) has load three. For clause \(d_2\), the unique sign variant that is false at \((0,0,1)\) is \(d_2\) itself, which is excluded from \(\mathcal G_2\); therefore, no generated arc associated with \(d_2\) has load three in this realization. The case \(\upsilon_1=0\) is symmetric: \(a_{2,g_2}\) has load three and no generated arc associated with \(d_1\) has load three. So, for this assignment, each satisfied source clause contributes exactly one generated load-three arc, whereas a false source clause contributes none.

\emph{Step 4:}
With \(m=1\), \(S=1\), \(K=3\), and \(\mathcal L_j=\{j_U\}\) for all \(j\), placing the single trailer at \(j_U\) provides three pallets, thus the total demand \(n+M=3\). So, all unmet demand variables can be set to zero for both uncertainty extreme points.

\emph{Step 5:}
For the PWL data, take \(g(\rho)=\rho+\rho^5\), \(c_a=t_a^0=1\), \(\rho_1=2\), and \(\rho_2=3\). Then
\(
\PWLSlope_1=g(2)/2=17, 
\PWLSlope_2=g(3)-g(2)=212, 
\delta=(212-17)(3-2)=195.
\)
Select a large enough route length \(R>0\), and pad all designated routes with private dummy arcs as needed. The constant first-segment contribution is $\PWLSlope_1R(n+M)=17\cdot R\cdot3=51R$, and the decision threshold from Step 6 is $\eta=51R+195$.
The objective values for the first-stage plan encoding \(\chi_1=\chi_2=0\) are
\medskip

\begin{center}
\resizebox{0.75\textwidth}{!}{
\begin{tabular}{c c c c}
\toprule
Universal assignment & Satisfied source clauses & Generated load-three arc & Objective value\\
\midrule
\(\upsilon_1=1\) & \(d_1\) & \(a_{1,g_1}\) & \(51R+195\)\\
\(\upsilon_1=0\) & \(d_2\) & \(a_{2,g_2}\) & \(51R+195\)\\
\bottomrule
\end{tabular}
}
\end{center}
\medskip

Thus, the constructed plan has worst-case objective value \(51R+195=\eta\). In contrast, if the shelter-opening pattern encodes, for example, \(\chi_1=1\), then both clauses are satisfied for both values of \(\upsilon_1\). In that case, two generated arcs attain load three, and the PWL component is \(51R+2\cdot195>\eta\). The threshold, therefore, separates assignments satisfying \(\max_{\upsilon}\operatorname{sat}_F(\chi,\upsilon)\le B\) from assignments that violate this condition.
\end{example}

\propcomplexitysigmamember*

{\color{black}
\begin{proof}
The existential certificate is a candidate first-stage plan \((b,s^1,y)\). The verifier checks in polynomial time that this plan satisfies the deterministic first-stage constraints. The universal certificate is a candidate realization \(\xi\in\Xi\).

For a fixed feasible first-stage plan and a fixed admissible realization \(\xi\), the realized arc loads and the PWL travel-time contribution are computable in polynomial time. The optimal unmet demand recourse value is also computable in polynomial time by Proposition~\ref{prop:complexity-recourse}. Hence, the verifier can evaluate the realized optimal value \(V(b,s^1,y,\xi)\) and check \(V(b,s^1,y,\xi)\le \eta\) in polynomial time. 

Thus, the decision problem admits an \(\exists\forall\) certificate structure with a polynomial-time predicate:
\(
\exists (b,s^1,y)\;\forall \xi\in\Xi:\quad V(b,s^1,y,\xi)\le \eta,
\). This is the defining characterization of the complexity class
\(\Sigma^p_2\).
Therefore, the decision version of \(\PWLCSO{}\) with a bounded polyhedral uncertainty set belongs to \(\Sigma^p_2\). 
\end{proof}
}



\section{Subproblems for User Route Choice Model}\label{app:subprob_choice_model}
Given a set of critical scenarios $\mathcal{P}=\{\hat{\xi}_p\}_{p \in P} \subset \Xi$, with $|\mathcal{P}|<\infty$, we present the upper- and lower-bounding subproblems for the 
user route choice \rmb{(\PWLURC) model obtained by applying the piecewise-linear approximation in Section~\ref{ttf} to model~\eqref{prefmodel}.} \rmb{The route-choice constraints \eqref{eq:route_assign_model2} and~\eqref{choice-model-const_} are implemented through the lifted McCormick reformulation described after Proposition~\ref{projection}.}
\subsubsection*{The Upper-Bounding Subproblem.}
The upper-bounding formulation is written as follows:
\begin{subequations}\label{prefmodel_UB}
\begin{align}
z_{\PrefUB}(\mathcal{P}) = \min \ & z \label{obj-function_model2_UB} \hspace{-1.3cm} &  \\
\text{s.t.} \ & \eqref{eq:tot_open_shelter_model1}, \eqref{eq:flexible_trailer_capa_model1}, \eqref{eq:route_select_model2}-\eqref{choice-model-const_}, \eqref{eq:non_neg_model22}, \nonumber\\
& \sum_{j \in \rmb{\mathcal{J}}} u_{jp} + \lambda \sum_{a \in \mathcal{A}} \rmb{\sum_{\PWLIdx \in \PWLSet}t^0_a \PWLSlope_{\PWLIdx}x^{\PWLIdx}_{ap}} \leq z \label{eq:obj_prefmodel_UB} & p \in P,\\
\label{eq:unmet_model2_UB} &   \sum_{i \in \mathcal{O}} \sum_{r \in \mathcal{R}_{ij}} \xi_{i} \pi_{r} \leq u_{jp} + m \sum_{\ell \in \rmb{\mathcal{L}_j}} s^{2}_{\ell jp}  & j \in \rmb{\mathcal{J}}, p \in P, \xi \in \Xi(\hat{\xi}_p),\\
\label{eq:arc_capa_model2_UB} &   \sum_{i \in \mathcal{O}} \sum_{j \in \rmb{\mathcal{J}}} \sum_{r \in \mathcal{R}_{ij}:a \in r} \xi_{i}\pi_{r} \leq \rmb{\sum_{\PWLIdx \in \PWLSet}x^{\PWLIdx}_{ap}} & a \in \mathcal{A}, p \in P, \xi \in \Xi(\hat{\xi}_p),\\
& \rmb{x^{\PWLIdx}_{ap} \leq c_a\PWLDelta_{\PWLIdx}} & \rmb{\PWLIdx \in \PWLSet\setminus\{\PWLCount\}, a \in \mathcal{A}, p \in P,}\\
&\label{eq:non_neg_model2_UB} s^{2}_{\ell jp} \in \mathbb{Z}_{+}, u_{jp}, \rmb{x^{\PWLIdx}_{ap}} \geq 0 & \hspace{-0.5cm} j \in \rmb{\mathcal{J}}, \ell \in \rmb{\mathcal{L}}, a \in \mathcal{A}, p \in P, \rmb{\PWLIdx \in \PWLSet}.
\end{align}
\end{subequations}

\subsubsection*{The Lower-Bounding Subproblem.}
The lower-bounding formulation is written as follows:
\begin{subequations}\label{prefmodel_LB}
\begin{align}
z_{\PrefLB}(\mathcal{P}) = \min \ & z &  \\
\text{s.t.} \  &\eqref{eq:tot_open_shelter_model1}, \eqref{eq:flexible_trailer_capa_model1}, \eqref{eq:route_select_model2}-\eqref{choice-model-const_}, \eqref{eq:non_neg_model22}, \nonumber\\
& \sum_{j \in \rmb{\mathcal{J}}} u_{jp} + \lambda \sum_{a \in \mathcal{A}} \rmb{\sum_{\PWLIdx \in \PWLSet}t^0_a \PWLSlope_{\PWLIdx}x^{\PWLIdx}_{ap}} \leq z \label{obj-function_model2_LB} & p \in P,\\
\label{eq:unmet_model2_LB} &   \sum_{i \in \mathcal{O}} \sum_{r \in \mathcal{R}_{ij}} \hat{\xi}_{pi} \pi_{r} \leq u_{jp} + m \sum_{\ell \in \rmb{\mathcal{L}_j}} s^{2}_{\ell jp}  & j \in \rmb{\mathcal{J}}, p \in P,\\
\label{eq:arc_capa_model2_LB} &   \sum_{i \in \mathcal{O}} \sum_{j \in \rmb{\mathcal{J}}} \sum_{r \in \mathcal{R}_{ij}:a \in r} \hat{\xi}_{pi}\pi_{r} \leq \rmb{\sum_{\PWLIdx \in \PWLSet}x^{\PWLIdx}_{ap}} & a \in \mathcal{A}, p \in P,\\
& \rmb{x^{\PWLIdx}_{ap} \leq c_a\PWLDelta_{\PWLIdx}} & \rmb{\PWLIdx \in \PWLSet\setminus\{\PWLCount\}, a \in \mathcal{A}, p \in P,}\\
&\label{eq:non_neg_model2_LB} s^{2}_{\ell jp} \in \mathbb{Z}_{+}, u_{jp}, \rmb{x^{\PWLIdx}_{ap}} \geq 0 & \hspace{-0.5cm} j \in \rmb{\mathcal{J}}, \ell \in \rmb{\mathcal{L}}, a \in \mathcal{A}, p \in P, \rmb{\PWLIdx \in \PWLSet}.
\end{align}
\end{subequations}

\section{\rmb{Reformulation of Semi-Infinite Upper-Bound Constraints}} \label{app:reformulation}
\rmb{This appendix derives the LP-duality reformulation used for the semi-infinite upper-bound constraints in UB-ARMIO under the continuous budgeted uncertainty set in the case study. The unmet-demand constraints~\eqref{eq:unmet_model_UB} and the traffic-volume constraints~\eqref{eq:arc_capa_model1_pwl_UB} have the same support function structure because the dependence on the uncertainty vector \(\xi\) is affine in both constraint families. As described in Section~\ref{selectactive}, each child leaf is obtained by adding one linear bisector inequality to the description of its parent leaf. Hence, for a current leaf \(p\), let \(H_p\) denote the finite set of linear split inequalities inherited along the path from the root to that leaf, and write}
\[
\Xi(\hat{\xi}_p)
=
\left\{
\xi\in\Xi:
\sum_{i\in\mathcal O}g_{hpi}\xi_i\le d_{hp},\ h\in H_p
\right\},
\]
\rmb{Here, \(H_p=\emptyset\) for the root leaf. Each time a leaf is split using two scenarios \(\hat{\xi}_{p_1}\) and \(\hat{\xi}_{p_2}\), the child \(\Xi(\hat{\xi}_{p_1})\) inherits the linear inequality}
\(
2(\hat{\xi}_{p_2}-\hat{\xi}_{p_1})^\top\xi\le \|\hat{\xi}_{p_2}\|_2^2-\|\hat{\xi}_{p_1}\|_2^2,
\)
\rmb{and the child \(\Xi(\hat{\xi}_{p_2})\) inherits the same inequality with both sides multiplied by \(-1\). Thus, for any current leaf \(p\), \(H_p\) is the set of split inequalities accumulated along its path from the root, written generically as \(\sum_{i\in\mathcal O}g_{hpi}\xi_i\le d_{hp}\), \(h\in H_p\).}

\rmb{Both robust constraint families can be written in the generic form}
\begin{align}
\label{eq:generic_support_UB}
\max_{\xi\in\Xi(\hat{\xi}_p)}\sum_{i\in\mathcal O}\psi_i\xi_i \le \vartheta.
\end{align}
\rmb{For the unmet-demand constraint~\eqref{eq:unmet_model_UB}, the substitution is}
$\psi_i = \sum_{r\in\mathcal R_{ij}}y_r$ for $i\in\mathcal O$ and $\vartheta = u_{jp}+m\sum_{\ell\in\mathcal L_j}s^{2}_{\ell jp}$ for $j\in\mathcal J,\ p\in P.$
\rmb{For the traffic-volume constraint~\eqref{eq:arc_capa_model1_pwl_UB}, the substitution is}
$\psi_i = \sum_{j\in\mathcal J}\sum_{r\in\mathcal R_{ij}:a\in r}y_r$ for $i\in\mathcal O$ and $\vartheta = \sum_{\PWLIdx\in\PWLSet}x^{\PWLIdx}_{ap}$ for $a\in\mathcal A,\ p\in P.$
Thus, in both cases, the coefficient vector of the support function is determined by the first-stage route-assignment variables. For the budgeted uncertainty set used in the case study, and for \(\epsilon>0\), the support-function problem in~\eqref{eq:generic_support_UB} can be written as
\[
\begin{aligned}
\max_{\xi}\quad & \sum_{i\in\mathcal O}\psi_i\xi_i\\
\text{s.t.}\quad
& \bar w_i(1-\epsilon)\le \xi_i\le \bar w_i(1+\epsilon), && i\in\mathcal O,\\
& \sum_{i\in\mathcal O}\left|\frac{\xi_i-\bar w_i}{\bar w_i\epsilon}\right|\le \Gamma,\\
& \sum_{i\in\mathcal O}g_{hpi}\xi_i\le d_{hp}, && h\in H_p.
\end{aligned}
\]
\rmb{Introducing auxiliary variables \(q_i\) to linearize the absolute-value budget yields the following LP:}
\begin{subequations}
\begin{align}
     \label{eq:inner_max} \text{max } & \sum_{i \in \mathcal{O}} \psi_i\xi_i\\
     \text{s.t.} \  & \xi_i \leq \bar{w}_{i}(1+\epsilon), \quad -\xi_i \leq -\bar{w}_{i}(1-\epsilon), & i \in \mathcal{O}\\
     & \xi_i - \bar{w}_{i}\epsilon q_i \leq \bar{w}_{i}, \quad -\xi_i - \bar{w}_{i}\epsilon q_i \leq -\bar{w}_{i}, \quad  q_i \geq 0, & i \in \mathcal{O}\\
     \label{eq:inner_max_fin} & \sum_{i \in \mathcal{O}} q_i \leq \Gamma, \qquad \sum_{i\in\mathcal O}g_{hpi}\xi_i\le d_{hp}, \ h\in H_p.
\end{align}
\end{subequations}
\rmb{The dual of~\eqref{eq:inner_max}--\eqref{eq:inner_max_fin} is}
\begin{subequations}
\begin{align}
    \label{eq:inner_dual} \text{min } & \sum_{i\in\mathcal O}\left(\alpha_i\bar w_i(1+\epsilon)-\beta_i\bar w_i(1-\epsilon)+\gamma_i\bar w_i-\delta_i\bar w_i\right)
    +\mu\Gamma +\sum_{h\in H_p}\kappa_h d_{hp}\\
    \label{eq:inner_dual_const} \text{s.t.  } & \alpha_i-\beta_i+\gamma_i-\delta_i+\sum_{h\in H_p}g_{hpi}\kappa_h=\psi_i & i\in\mathcal O\\
    & \bar w_i\epsilon(\gamma_i+\delta_i)\leq \mu, \quad \alpha_i,\beta_i,\gamma_i,\delta_i\geq0 & i\in\mathcal O\\
   \label{eq:inner_dual_fin} & \mu\geq0, \qquad \kappa_h\geq0,\ h\in H_p.
\end{align}
\end{subequations}
\rmb{By LP strong duality, the generic robust inequality~\eqref{eq:generic_support_UB} is enforced by introducing dual variables satisfying~\eqref{eq:inner_dual_const}--\eqref{eq:inner_dual_fin} and requiring the dual objective~\eqref{eq:inner_dual} to be no larger than \(\vartheta\).} By applying the robust reformulation to constraints~\eqref{eq:unmet_model_UB} and \eqref{eq:arc_capa_model1_pwl_UB}, we obtain constraints~\eqref{eq:RC_unmet_demand_1}--\eqref{eq:RC_unmet_demand_5} and \eqref{eq:RC_evac_time_1}--\eqref{eq:RC_evac_time_5}, respectively. The resulting primal robust counterpart formulation of \CSO{} is given by:
\begin{subequations}\label{primal_model_RC}
\begin{align}
\min \ & z \\
\text{s.t.} \ & \eqref{eq:tot_open_shelter_model1}-\eqref{eq:route_restict_model1}, \eqref{UB-obj-function}, \rmb{\eqref{eq:pwl_segment_bound_UB}}-\eqref{eq:non_neg_model1_UB_2} \notag\\
& \sum_{i \in \mathcal{O}} \left(\alpha'_{ijp} \bar{w}_{i} (1+\epsilon) 
    - \beta'_{ijp} \bar{w}_{i} (1-\epsilon) + \gamma'_{ijp} \bar{w}_{i} 
    - \delta'_{ijp} \bar{w}_{i}\right) \notag \\
    & \quad + \rmb{\sum_{h \in H_p} \kappa'_{jph} d_{hp}} + \mu'_{jp} \Gamma - u_{jp} \leq m \sum_{\ell \in \rmb{\mathcal{L}_j}} s^{2}_{\ell jp} & j \in \rmb{\mathcal{J}}, p \in P, \label{eq:RC_unmet_demand_1} \\
    & \alpha'_{ijp} - \beta'_{ijp} + \gamma'_{ijp} - \delta'_{ijp} + \rmb{\sum_{h \in H_p} g_{hpi}\kappa'_{jph}} = \sum_{r \in \mathcal{R}_{ij}} y_{r} & i \in \mathcal{O}, j \in \rmb{\mathcal{J}}, p \in P, \label{eq:RC_unmet_demand_2} \\
    & \bar{w}_{i}\epsilon \gamma'_{ijp} + \bar{w}_{i}\epsilon \delta'_{ijp} \leq \mu'_{jp} & i \in \mathcal{O}, j \in \rmb{\mathcal{J}}, p \in P, \label{eq:RC_unmet_demand_3}\\
    & \alpha'_{ijp}, \beta'_{ijp}, \gamma'_{ijp}, \delta'_{ijp}, \mu'_{jp} \geq 0 & i \in \mathcal{O}, j \in \rmb{\mathcal{J}}, p \in P, \label{eq:RC_unmet_demand_4}\\
    & \rmb{\kappa'_{jph} \geq 0} & j \in \rmb{\mathcal{J}}, p \in P, \rmb{h\in H_p,} \label{eq:RC_unmet_demand_5} \\
    & \sum_{i \in \mathcal{O}} \left(\alpha_{iap} \bar{w}_{i} (1+\epsilon) 
     - \beta_{iap} \bar{w}_{i} (1-\epsilon) + \gamma_{iap} \bar{w}_{i} 
     - \delta_{iap} \bar{w}_{i}\right) \notag \\
     & \quad + \rmb{\sum_{h \in H_p} \kappa_{aph}d_{hp}} + \mu_{ap} \Gamma \leq \rmb{\sum_{\PWLIdx\in\PWLSet}x^{\PWLIdx}_{ap}} & a \in \mathcal{A}, p \in P, \label{eq:RC_evac_time_1}\\
     & \alpha_{iap} - \beta_{iap} + \gamma_{iap} - \delta_{iap} + \rmb{\sum_{h \in H_p} g_{hpi}\kappa_{aph}} = \sum_{j \in \rmb{\mathcal{J}}} \sum_{r \in \mathcal{R}_{ij}:a \in r} y_{r} & i \in \mathcal{O}, a \in \mathcal{A}, p \in P,\label{eq:RC_evac_time_2}\\
    & \bar{w}_{i}\epsilon \gamma_{iap} + \bar{w}_{i}\epsilon \delta_{iap} \leq \mu_{ap} & i \in \mathcal{O}, a \in \mathcal{A}, p \in P,\label{eq:RC_evac_time_3}\\
    & \alpha_{iap}, \beta_{iap}, \gamma_{iap}, \delta_{iap}, \mu_{ap} \geq 0 & i \in \mathcal{O}, a \in \mathcal{A}, p \in P, \label{eq:RC_evac_time_4}\\
    & \label{eq:RC_evac_time_5} \rmb{\kappa_{aph} \geq 0} & a \in \mathcal{A}, p \in P, \rmb{h\in H_p.}
\end{align}
\end{subequations}

\section{Dual-based Critical Scenario Identification}\label{app:dual_based_scenario_identification}
This appendix details the dual-based critical scenario identification procedure in \citet{romeijnders2021piecewise}, which requires a bounded polyhedral uncertainty set. This approach iteratively partitions the uncertainty set into refined subsets $\mathcal{P}_r$. In iteration $r$, the robust counterpart of the finitely adaptable problem~\eqref{zpartmodel} is solved using a branch-and-bound (B\&B) algorithm to obtain $z_{UB}(\mathcal{P}_r)$. Let $\mathcal{N}_r$ denote the nodes of the B\&B tree. Then, a critical cutset $\mathcal{O}_r \subset \mathcal{N}_r$ consists of nodes $n$ such that $(i)$ $z_{LP}^{n} \geq z_{UB}(\mathcal{P}_r)$, and $(ii)$ $\mathcal{O}_r \cap \Pi(\ell) \neq \emptyset$, where $\Pi(\ell)$ represents the path from the root node to the leaf node $\ell$ in the B\&B tree. Critical scenarios in each uncertainty subset are identified using the dual solutions of the nodes in $\mathcal{O}_r$. Finally, the active subset that determines the worst-case objective value after the r-th splitting round is split into two subsets using the bisecant plane between the two critical scenarios that are furthest apart from each other.

\subsection{Scenario Identification from Optimal Dual Solutions}
In iteration $r$, consider a B\&B tree node $n \in \mathcal{O}_r$. We drop the node index $n$ in the rest of this section for clarity. Let $\omega^{(1)}_{jp}, \nu^{(1)}_{ijp}$, and $\eta^{(1)}_{ijp}$ be the dual variables associated with constraints~\eqref{eq:RC_unmet_demand_1}--\eqref{eq:RC_unmet_demand_3}, respectively. For uncertainty subset $p \in P_r$, the dual constraints for  $\alpha'_{ijp}, \beta'_{ijp}, \gamma'_{ijp}, \delta'_{ijp}, \mu'_{jp}, \kappa'_{jph}$ variables are given by:
\begin{subequations}
    \begin{align}
        &\bar{w}_{i}(1-\epsilon)\omega^{(1)}_{jp} \leq \nu^{(1)}_{ijp} \leq \bar{w}_{i}(1+\epsilon)\omega^{(1)}_{jp} & i \in \mathcal{O}, \label{eq:dual_rec_box}\\
        &\bar{w}_{i}\epsilon \eta^{(1)}_{ijp} \geq |\nu^{(1)}_{ijp}-\bar{w}_{i}\omega^{(1)}_{jp}| & i \in \mathcal{O}, \label{eq:dual_rec_abs}\\
        &\sum_{i \in \mathcal{O}} \eta^{(1)}_{ijp} \leq \omega^{(1)}_{jp} \Gamma, \label{eq:dual_rec_budget}\\
        &\rmb{\sum_{i\in\mathcal O}g_{hpi}\nu^{(1)}_{ijp}\le d_{hp}\omega^{(1)}_{jp}} & \rmb{h\in H_p.} \label{eq:dual_rec_partition}
    \end{align}
\end{subequations}
If $\omega^{(1)}_{jp}>0$, then by substituting $\xi^{(1)}_{ijp}=\nu^{(1)}_{ijp} / \omega^{(1)}_{jp}$, we can reconstruct the constraints of the uncertainty set partition $p \in P_r$. Consequently, the critical scenario $\xi^{(1)}_{jp}  = \{\xi^{(1)}_{ijp}\}_{i \in \mathcal{O}}$ for constraint~\eqref{eq:unmet_model_UB} can be directly obtained. Similarly, let $\omega^{(2)}_{ap}, \nu^{(2)}_{iap}$, and $\eta^{(2)}_{iap}$ be the dual variables associated with constraints~\eqref{eq:RC_evac_time_1}--\eqref{eq:RC_evac_time_3}. By substituting \rmb{$\xi^{(2)}_{iap} = \nu^{(2)}_{iap} / \omega^{(2)}_{ap}$}, whenever $\omega^{(2)}_{ap}>0$, the critical scenario \(\xi^{(2)}_{ap}=\{\xi^{(2)}_{iap}\}_{i\in\mathcal O}\) for constraint~\eqref{eq:arc_capa_model1_pwl_UB} can be obtained in an identical manner. 



\section{Additional Computational Results}\label{app:detailed_computational_results}
This appendix provides additional computational results for Section~\ref{sec:operational_effectiveness}. The experiments are conducted under varying numbers of shelters $N \in \{2, 4, 6\}$, trailers $S \in \{2, 5, 10\}$, and maximum demand coverage percentages $\coverage \in \{95\%, 99\%\}$, setting the weight on evacuation time in the objective function to $\lambda = 0.1$. Table~\ref{tab:val_of_assign_adaptive_allocation} presents the percentage improvements in the worst-case objective upper bound (UB), unmet demand, and evacuation time due to centralized route assignment
(\CSOFix\ vs \PrefFix) and adaptive redistribution of relief supplies
(\Pref\ vs \PrefFix).
The results in Table~\ref{tab:detailed_results} show the worst-case unmet demand and evacuation time, upper and lower bounds, optimality gaps, and solution times for the four different models (\CSO, \CSOFix, \Pref, and \PrefFix). These results are used to analyze the effectiveness of centralized evacuation route planning and adaptive allocation of relief supplies.

\begin{table}[h!]
\centering
\caption{Percent improvement due to centralized route assignment
(\CSOFix\ vs \PrefFix) and adaptive redistribution of relief supplies
(\Pref\ vs \PrefFix).}
\label{tab:val_of_assign_adaptive_allocation}
\renewcommand{\arraystretch}{0.78}
\setlength{\tabcolsep}{9pt}
\resizebox{0.85\textwidth}{!}{
\begin{tabular}{%
  c|
  c|
  c
  | c c
  | c c
  | c c
}
\specialrule{1.5pt}{0pt}{0pt}
 \multicolumn{3}{c}{}& \multicolumn{2}{c}{${N=2}$} & \multicolumn{2}{c}{${N=4}$} & \multicolumn{2}{c}{${N=6}$} \\
\cmidrule(lr){4-5}\cmidrule(lr){6-7}\cmidrule(lr){8-9}
\multicolumn{1}{c}{$S$}  & \multicolumn{1}{c}{$\coverage$}  &\multicolumn{1}{c}{ }  
& \CSOFix\ vs \PrefFix
& \Pref\ vs \PrefFix
& \CSOFix\ vs \PrefFix
& \Pref\ vs \PrefFix
& \CSOFix\ vs \PrefFix
& \Pref\ vs \PrefFix
\\
\specialrule{1pt}{0pt}{0pt}

\multirow{6}{*}{$2$}

& \multirow{3}{*}{$95\%$}
& UB
& 38.7 & 6.6
& 55.1 & 32.0
& 67.0 & 41.2 \\

&
& Unmet dem.
& -8.9 & -1.8
& 61.2 & 58.8
& 79.9 & 53.7 \\

&
& Evac. time
& 50.1 & 7.6
& 44.9 & -0.5
& 30.8 & 9.4 \\

\cmidrule(lr){2-9}

& \multirow{3}{*}{$99\%$}
& UB
& 47.3 & 14.0
& 61.4 & 37.4
& 71.2 & 41.2 \\

&
& Unmet dem.
& 32.6 & 45.6
& 88.3 & 91.1
& 93.7 & 53.7 \\

&
& Evac. time
& 47.8 & 10.6
& 25.8 & -27.8
& 10.4 & 9.4 \\

\specialrule{1pt}{0pt}{0pt}

\multirow{6}{*}{$5$}

& \multirow{3}{*}{$95\%$}
& UB
& 50.1 & 18.3
& 58.4 & 22.6
& 68.1 & 8.1 \\

&
& Unmet dem.
& 4.1 & 6.5
& 45.6 & 38.9
& 68.3 & 21.3 \\

&
& Evac. time
& 61.7 & 22.0
& 71.0 & 5.5
& 66.3 & -28.4 \\

\cmidrule(lr){2-9}

& \multirow{3}{*}{$99\%$}
& UB
& 61.9 & 29.7
& 76.5 & 16.3
& 82.9 & 6.1 \\

&
& Unmet dem.
& 70.1 & 72.1
& 87.6 & 41.8
& 91.0 & 22.2 \\

&
& Evac. time
& 60.7 & 22.3
& 62.4 & -16.9
& 62.8 & -32.5 \\

\specialrule{1pt}{0pt}{0pt}

\multirow{6}{*}{$10$}

& \multirow{3}{*}{$95\%$}
& UB
& 39.2 & 0.4
& 42.8 & 0.1
& 53.0 & 10.1 \\

&
& Unmet dem.
& -1.4 & 0.0
& 18.9 & 0.2
& 37.5 & 18.6 \\

&
& Evac. time
& 51.2 & 0.0
& 63.3 & 0.3
& 71.9 & -1.7 \\

\cmidrule(lr){2-9}

& \multirow{3}{*}{$99\%$}
& UB
& 45.3 & -0.1
& 67.0 & 7.5
& 73.2 & 8.2 \\

&
& Unmet dem.
& -18.2 & -0.3
& 74.4 & 16.9
& 75.9 & 57.4 \\

&
& Evac. time
& 49.5 & 0.0
& 60.8 & -1.0
& 69.7 & -36.4 \\

\specialrule{1.5pt}{0pt}{0pt}

\end{tabular}
}
\end{table}

\begin{table}[h!]
\centering
\caption{Detailed computational results of the \CSO, \CSOFix, \Pref \ and \PrefFix \ models under different resource settings with $\lambda = 0.1$.}
\label{tab:detailed_results}
\renewcommand{\arraystretch}{0.95}
\resizebox{1.0\textwidth}{!}{
\begin{tabular}{%
  c
  c
  | c c c c
  | c c c c
  | c c c c
}
\specialrule{1.5pt}{0pt}{0pt}
 & & \multicolumn{4}{c|}{${N=2}$} & \multicolumn{4}{c|}{${N=4}$} & \multicolumn{4}{c}{${N=6}$} \\
\cmidrule(lr){3-6}\cmidrule(lr){7-10}\cmidrule(lr){11-14}
\multicolumn{2}{c|}{ }  
  & \CSO & \CSOFix & \Pref & \PrefFix
  & \CSO & \CSOFix & \Pref & \PrefFix 
  & \CSO & \CSOFix & \Pref & \PrefFix \\
\specialrule{1pt}{0pt}{0pt}


\multirow[c]{6}{*}{\begin{tabular}{@{}c@{}}$S=2$ \\ $\coverage=95\%$\end{tabular}}
& Unmet demand
  & 7,358 & 7,901 & 7,392 & 7,258
  & 7,924 & 8,250 & 8,778 & 21,285
  & 7,923 & 8,294 & 19,077 & 41,168 \\
& Evac. time
  & 118,147 & 131,521 & 243,364 & 263,462
  & 77,212 & 99,204 & 180,873 & 179,992
  & 77,923 & 113,771 & 149,063 & 164,469 \\
& UB
  & 18,892 & 20,363 & 31,013 & 33,195
  & 15,223 & 17,491 & 26,495 & 38,988
  & 15,352 & 18,930 & 33,715 & 57,368 \\
& LB
  & 17,697 & 18,790 & 30,334 & 32,578
  & 14,022 & 16,082 & 25,991 & 38,158
  & 14,026 & 16,951 & 32,319 & 56,422  \\
& Gap (\%)
  & 6.3 & 7.7 & 2.2 & 1.9  
  & 7.9 & 8.1 & 1.9 & 2.1
  & 8.6 & 10.5 & 4.1 & 1.6  \\
& Sol. time (sec)
  & 12,213 & 11,185 & 11,137 & 10,837 
  & 11,459 & 11,979 & 11,218 & 11,373
  & 11,987 & 11,997 & 11,147 & 11,526  \\
\cmidrule(lr){1-14}

\multirow[c]{6}{*}{\begin{tabular}{@{}c@{}}$S=2$ \\ $\coverage=99\%$\end{tabular}}
& Unmet demand
  & 1,702 & 1,905 & 1,537 & 2,825
  & 1,867 & 2,489 & 1,887 & 21,285
  & 1,796 & 2,597 & 19,077 & 41,168 \\
& Evac. time
  & 118,097 & 141,383 & 242,046 & 270,724
  & 96,960 & 133,487 & 229,989 & 179,899
  & 95,798 & 147,304 & 149,063 & 164,469 \\
& UB
  & 13,207 & 15,477 & 25,271 & 29,377
  & 11,317 & 15,053 & 24,386 & 38,973
  & 11,203 & 16,511 & 33,715 & 57,368 \\
& LB
  & 11,938 & 14,119 & 24,574 & 27,372
  & 9,569 & 13,720 & 23,501 & 38,158
  & 9,639 & 14,936 & 32,316 & 56,422  \\
& Gap (\%)
  & 9.6 & 8.8 & 2.8 & 6.8
  & 15.4 & 8.9 & 3.6 & 2.1
  & 14.0 & 9.5 & 4.1 & 1.6  \\
& Sol. time (sec)
  & 12,052 & 11,217 & 11,086 & 11,277 
  & 11,462 & 11,507 & 11,033 & 12,040
  & 12,286 & 10,806 & 11,950 & 11,471 \\
\specialrule{1pt}{0pt}{0pt}


\multirow[c]{6}{*}{\begin{tabular}{@{}c@{}}$S=5$ \\ $\coverage=95\%$\end{tabular}}
& Unmet demand
  & 7,385 & 7,491 & 7,301 & 7,809
  & 7,453 & 7,325 & 8,226 & 13,453
  & 7,479 & 7,628 & 18,949 & 24,085 \\
& Evac. time
  & 120,727 & 119,665 & 243,364 & 312,072
  & 36,792 & 38,690 & 126,126 & 133,441
  & 28,751 & 29,199 & 111,417 & 86,754 \\
& UB
  & 19,102 & 19,007 & 31,104 & 38,081
  & 11,046 & 11,100 & 20,651 & 26,675
  & 10,323 & 10,428 & 30,067 & 32,714 \\
& LB
  & 17,697 & 17,569 & 30,334 & 37,051 
  & 10,259 & 10,189 & 19,808 & 26,408
  & 9,500 & 8,843 & 28,469 & 31,127  \\
& Gap (\%)
  & 7.4 & 7.6 & 2.5 & 2.7
  & 7.1 & 8.2 & 4.1 & 1.0  
  & 8.0 & 15.2 & 5.3 & 4.8  \\
& Sol. time (sec)
  & 11,258 & 12,497 & 11,893 & 10,955
  & 12,134 & 10,809 & 12,013 & 3,951
  & 10,996 & 12,009 & 12,746 & 11,749\\
\cmidrule(lr){1-14}

\multirow[c]{6}{*}{\begin{tabular}{@{}c@{}}$S=5$ \\ $\coverage=99\%$\end{tabular}}
& Unmet demand
  & 1,709 & 1,648 & 1,537 & 5,505
  & 1,648 & 1,670 & 7,827 & 13,441
  & 1,642 & 1,858 & 16,050 & 20,629 \\
& Evac. time
  & 119,191 & 122,615 & 242,535 & 312,158
  & 37,617 & 40,225 & 124,945 & 106,901
  & 31,513 & 32,283 & 114,966 & 86,755 \\
& UB
  & 13,300 & 13,656 & 25,211 & 35,840
  & 5,288 & 5,595 & 19,951 & 23,832
  & 4,759 & 4,994 & 27,474 & 29,258 \\
& LB
  & 11,938 & 12,113 & 24,574 & 34,747
  & 4,493 & 4,903 & 19,494 & 23,443
  & 3,541 & 3,705 & 26,762 & 27,671 \\
& Gap (\%)
  & 10.2 & 11.3 & 2.5 & 3.1
  & 15.0 & 12.4 & 2.3 & 1.6
  & 25.6 & 25.8 & 2.6 & 5.4  \\
& Sol. time (sec)
  & 11,906 & 11,935 & 11,848 & 11,260
  & 11,601 & 11,435 & 11,563 & 11,259  
  & 11,569 & 12,189 & 12,311 & 12,069  \\
\specialrule{1pt}{0pt}{0pt}


\multirow[c]{6}{*}{\begin{tabular}{@{}c@{}}$S=10$ \\ $\coverage=95\%$\end{tabular}}
& Unmet demand
  & 7,349 & 7,401 & 7,297 & 7,297
  & 7,496 & 7,415 & 9,129 & 9,143
  & 7,781 & 7,695 & 10,014 & 12,303\\
& Evac. time
  & 117,252 & 118,471 & 242,535 & 242,535
  & 37,830 & 38,679 & 104,937 & 105,277
  & 27,520 & 28,132 & 101,840 & 100,158 \\
& UB
  & 18,811 & 18,928 & 31,031 & 31,146
  & 11,099 & 11,204 & 19,586 & 19,604
  & 10,501 & 10,446 & 19,978 & 22,216\\
& LB
  & 17,526 & 17,732 & 30,334 & 30,332
  & 10,345 & 10,398 & 18,907 & 18,905
  & 9,264 & 9,333 & 19,634 & 21,938\\
& Gap (\%)
  & 6.8 & 6.3 & 2.2 & 2.6  
  & 6.8 & 7.2 & 3.5 & 3.6
  & 11.8 & 10.7 & 1.7 & 1.3\\
& Sol. time (sec)
  & 11,494 & 12,458 & 11,907 & 11,367  
  & 11,919 & 12,001 & 11,373 & 11,392
  & 10,908 & 11,444 & 11,890 & 11,516\\
\cmidrule(lr){1-14}

\multirow[c]{6}{*}{\begin{tabular}{@{}c@{}}$S=10$ \\ $\coverage=99\%$\end{tabular}}
& Unmet demand
  & 1,576 & 1,816 & 1,541 & 1,537
  & 1,652 & 1,742 & 5,666 & 6,816
  & 2,020 & 2,032 & 3,597 & 8,449\\
& Evac. time
  & 118,835 & 122,463 & 242,535 & 242,535
  & 37,542 & 41,269 & 106,464 & 105,404
  & 28,238 & 30,299 & 136,567 & 100,158\\
& UB
  & 13,207 & 13,877 & 25,403 & 25,386
  & 5,322 & 5,671 & 15,879 & 17,160
  & 4,765 & 4,931 & 16,871 & 18,369 \\
& LB
  & 11,856 & 12,213 & 24,574 & 24,572
  & 4,580 & 4,768 & 15,735 & 16,779
  & 3,571 & 3,696 & 15,618 & 18,085 \\
& Gap (\%)
  & 10.2 & 12.0 & 3.3 & 3.2
  & 13.9 & 15.9 & 0.9 & 2.2
  & 25.1 & 25.0 & 7.4 & 1.5\\
& Sol. time (sec)
  & 12,073 & 11,786 & 11,869 & 10,899
  & 11,423 & 11,926 & 9,349 & 11,971
  & 12,066 & 11,983 & 10,891 & 11,481\\
\specialrule{1.5pt}{0pt}{0pt}


\end{tabular}
}
\end{table}

\end{APPENDICES}

\end{document}